\documentclass[12pt]{amsart}
\usepackage{tikz}
\usepackage{amssymb}
\usepackage[margin=1.1in]{geometry}
\usepackage{hyperref}
\def\Vert{{\rm Vert}}
\def\RR{{\mathcal{R}}}
\def\Q{{\mathbb Q}}
\def\R{{\mathbb R}}
\def\C{{\mathbb C}}
\def\Z{{\mathbb Z}}
\def\F{{\mathbb F}}
\def\tF{{\tilde F}}

\def\Spec{\operatorname{Spec}}

\def\Conv{{\operatorname{Conv}}}
\def\Trop{{\operatorname{Trop}}}

\def\D{{\mathcal{D}}}
\def\sp{\operatorname{span}}

\def\A{{\mathcal{A}}}
\def\G{{\mathcal{G}}}
\def\SS{{\mathcal{S}}}
\def\I{{\mathcal{I}}}

\def\T{{\mathcal{T}}}
\def\u{{\mathbf{u}}}
\def\a{{\mathbf{a}}}
\def\y{{\mathbf{y}}}
\def\z{{\mathbf{z}}}
\def\Gr{{\rm Gr}}
\def\hGr{\widehat{\Gr}}

\def\ZZ{{\mathcal{Z}}}
\def\oGr{{\mathring{\Gr}}}
\def\CC{{\mathcal{C}}}
\def\M{{\mathcal{M}}}
\def\bM{\overline{\M}}
\def\tM{{\widetilde{\M}}}
\def\GL{{\rm GL}}
\def\Conf{{\operatorname{Conf}}}
\def\Ch{{\operatorname{Ch}}}
\def\bConf{\Ch}
\def\oConf{{\mathring{\Conf}}}
\def\oTheta{{\mathring{\Theta}}}
\def\tConf{{\widetilde{\Ch}}}
\def\Vol{{\rm Vol}}
\def\val{{\rm val}}

\def\Pluck{{\rm Dr}}
\newcommand{\PP}{{\mathbb P}}
\def\Supp{{\rm Supp}}
\def\F{{\mathcal F}}
\def\hT{{\widehat{ T}}}
\def\T{{\mathcal T}}

\def\tDelta{{\tilde \Delta}}
\def\N{{\bf N}}
\def\hPi{{\widehat{\Pi}}}
\def\oPi{{\mathring{\Pi}}}
\def\tPi{{\widetilde{\Pi}}}
\def\vv{{\mathbf{v}}}
\def\x{{\mathbf{x}}}
\def\sbinom{{\mbox{$\binom{[n]}{k}$}}}

\newtheorem{conjecture}{Conjecture}
\newtheorem{theorem}[conjecture]{Theorem}
\newtheorem{lemma}[conjecture]{Lemma}
\newtheorem{proposition}[conjecture]{Proposition}
\newtheorem{corollary}[conjecture]{Corollary}

\theoremstyle{remark}
\newtheorem{remark}[conjecture]{Remark}
\newtheorem{question}[conjecture]{Question}

\numberwithin{conjecture}{section}
\numberwithin{equation}{section}

\author{Nima Arkani-Hamed}
\address{School of Natural Sciences, Institute for Advanced Study, Princeton, NJ 08540}
\address{Center of Mathematical Sciences and Applications, Harvard University, Cambridge, MA 02138}
\email{arkani@ias.edu}
\author{Thomas Lam}
\address{Department of Mathematics, University of Michigan, Ann Arbor, MI 48109}
\address{Department of Mathematics, Massachusetts Institute of Technology, Cambridge, MA 02139}
\email{tfylam@umich.edu}
\author{Marcus Spradlin}
\address{Department of Physics and Brown Theoretical Physics Center, Brown University, Providence, RI 02912}
\email{marcus\_spradlin@brown.edu}

\begin{document}
\title{Positive configuration space}
\begin{abstract}
We define and study the totally nonnegative part of the Chow quotient of the Grassmannian, or more simply the {\it nonnegative configuration space}.  This space has a natural stratification by {\it positive Chow cells}, and we show that nonnegative configuration space is homeomorphic to a polytope as a stratified space.  We establish bijections between positive Chow cells and the following sets: (a) regular subdivisions of the hypersimplex into positroid polytopes, (b) the set of cones in the positive tropical Grassmannian, and (c) the set of cones in the positive Dressian.  Our work is motivated by connections to super Yang-Mills scattering amplitudes, which will be discussed in a sequel.

\end{abstract}
\maketitle
\tableofcontents
\section{Introduction}
This is the first in a sequence of papers where we define and study the totally nonnegative part of the Chow quotient of the Grassmannian, or more simply the {\it nonnegative configuration space}.  In this paper, we focus on the combinatorics and topology of this space.  In a sequel \cite{ALS2}, we will further study the geometry and its relations to cluster algebras, canonical bases, and scattering amplitudes.  Some of the applications of our work to ${\mathcal N}=4$ super Yang-Mills amplitudes were announced in the note \cite{ALS}.

\subsection{}
Lusztig \cite{Lus} and Postnikov \cite{Pos} defined the totally nonnegative Grassmannian $\Gr(k,n)_{\geq 0}$ (informally, the positive Grassmannian), the closed subspace of the real Grassmannian $\Gr(k,n)(\R)$ of $k$-planes in $\R^n$ cut out by the condition that all Pl\"ucker coordinates are nonnegative.  Postnikov \cite{Pos} (see also \cite{Rie}) studied the stratification of $\Gr(k,n)_{\geq 0}$ by positroid cells $\Pi_{\M,>0}$ and Galashin, Karp, and Lam \cite{GKL,GKL2} showed that positroid cells endow $\Gr(k,n)_{\geq 0}$ with the structure of a regular CW-complex homeomorphic to a closed ball.  Positroid cells are indexed by a class of matroids called positroids.  Positroids have been completely classified \cite{Oh} and are in bijection with Grassmann necklaces \cite{Pos} and bounded affine permutations \cite{KLS}, among other combinatorial objects.  This is in stark contrast to the situation of matroids.  There is no explicit classification of (realizable) matroids, and the geometry of matroid strata \cite{GGMS} is notoriously complicated \cite{Mn}.
%
%Postnikov also developed a deep theory relating $\Gr(k,n)_{\geq 0}$ to plabic graphs, extending classical works relating total positivity to networks \cite{....}.
%Among the many ways to index positroid cells, the most important to us will be: (a) by positroids, a special class of matroids \cite{Oh}, (b) by Grassmann necklaces \cite{Pos}, and (c) by bounded affine permutations \cite{KLS}.

Recently, the positive Grassmannian has made a prominent appearance in the study of scattering amplitudes \cite{book, AT}, where the boundary structure of $\Gr(k,n)_{\geq 0}$ was connected to the singularities of tree-level scattering amplitudes in maximally supersymmetric Yang-Mills theory.  Part of this connection is formalized in the statement that $\Gr(k,n)_{\geq 0}$ is a positive geometry \cite{ABL}, and one driving force in this developing subject, and of the present work, is to find positive geometries that have relations to physical problems. %\fixit{Mark+Nima: please add more here}

\subsection{}
The Grassmannian $\Gr(k,n)$ has a natural action of a torus $T$ that acts by scaling the basis vectors of the underlying vector space $\C^n$.  The quotient $\Gr(k,n)/T$ is closely related to the configuration space $\Conf(k,n)$ of $n$ points in $\PP^{k-1}$.  Kapranov \cite{Kap} studied the Chow quotient $\Ch(k,n)$ of the Grassmannian, which is a compactification of the subspace $\oConf(k,n) \subset \Conf(k,n)$ of generic configurations.  In the case $k = 2$, the Chow quotient $\Ch(2,n)$ is isomorphic to the Deligne-Knudsen-Mumford space $\bM_{0,n}$ of $n$-pointed stable rational curves.  

The image of the positive Grassmannian $\Gr(k,n)_{>0}$ in $\oConf(k,n) \subset \Ch(k,n)$ is the positive component $\Conf(k,n)_{>0}$ of generic configuration space.  We define the {\it totally nonnegative part of the Chow quotient of the Grassmannian}, or simply {\it nonnegative configuration space} $\Ch(k,n)_{\geq 0}$, as the closure of this positive component inside $\Ch(k,n)$.  As we shall explain in more detail in a sequel \cite{ALS2}, see also \cite{ALS,GGSVV}, scattering amplitudes can be thought of as ``functions" on $\Gr(4,n)/T$.  Just as the combinatorics of $\Gr(k,n)_{\geq 0}$ controls the singularities of the amplitude at tree-level, we expect the combinatorics of $\Ch(4,n)_{\geq 0}$ to be closely related to the singularities of the full scattering amplitude (integrated, and at all loops).  Some connections of (the closely related) tropical Grassmannians to amplitudes have also been discussed in \cite{CEGM,CR,DFOK1,SG,DFOK2,HP,DFOK3}.
\subsection{}
Our first aim in this work is to study the combinatorics of a natural stratification $\{\Theta_{\tDelta,>0}\}$ of $\Ch(k,n)_{\geq 0}$, extending the combinatorics associated to positroid cells of $\Gr(k,n)_{\geq 0}$.  We call these strata {$\{\Theta_{\tDelta,>0}\}$} {\it positive Chow cells}.

\begin{theorem}\label{thm:introbij}
There are canonical bijections between the following sets. 
\begin{enumerate}
\item The set $\{\Theta_{\tDelta,>0}\}$ of positive Chow cells of $\Ch(k,n)_{\geq 0}$.
\item The set $\D(k,n)$ of regular subdivisions of the hypersimplex $\Delta(k,n)$ into positroid polytopes.
\item The set of cones in the positive tropical Grassmannian $\Trop_{>0}\Gr(k,n)$, the space of valuations of positive Puiseux series points $\Gr(k,n)(\RR_{>0})$.
\item The set of cones in the space $\Pluck(k,n)_{>0} \subset \R^{\binom{[n]}{k}}$ (called the positive Dressian) of vectors satisfying the positive tropical (three-term) Pl\"ucker relations.
\end{enumerate}
\end{theorem}

Our second main result is a description of the topology of $\Ch(k,n)_{\geq 0}$, a variant of the results of \cite{Pos,GKL,GKL2}.  Somewhat surprisingly, while the geometry of the Chow quotient is considerably more complicated than that of the Grassmannian, the following result is easier than its Grassmannian counterpart.

\begin{theorem}\label{thm:introtop}
There is a stratification-preserving homeomorphism from nonnegative configuration space to a polytope.  In particular, each positive Chow cell $\Theta_{\tDelta,>0} \subset \Ch(k,n)_{\geq 0}$ is homeomorphic to an open ball.
\end{theorem}

We remark that $\Gr(k,n)_{\geq 0}$ is {\it not} homeomorphic to a polytope as a stratified space.
We now explain each of the objects in Theorem~\ref{thm:introbij} in turn.  Let $[n]:= \{1,2,\ldots,n\}$ and let $\binom{[n]}{k}$ denote the set of all $k$-element subsets of $[n]$.

\subsection{} %Each point $V \in \Gr(k,n)$ is represented by a subspace $V \subset \C^n$.
Each point $X \in \Ch(k,n)$ is represented by an algebraic cycle inside the Grassmannian.  If $X \in \oConf(k,n) \subset \Ch(k,n)$ then $X$ is represented by the torus orbit closure $\overline{T \cdot V}$ of a generic point $V$ in the Grassmannian.  The toric variety $T \cdot V$ is isomorphic to the projective toric variety $X_{\Delta(k,n)}$ associated to the hypersimplex $\Delta(k,n)$, the convex hull of all vectors $e_I \in \R^n$, $I \in \binom{[n]}{k}$ with $k$ 0-s and $(n-k)$ 1-s.  A general point $X \in \Ch(k,n)$ is represented by a union of toric varieties $X_{P_1}, \ldots,X_{P_m}$, where $P_1,\ldots,P_m$ are polytopes that form a regular subdivision of the hypersimplex into matroid polytopes (see \S\ref{sec:matroid}).  We thus obtain a stratification of $\Ch(k,n)$ by matroid subdivisions of the hypersimplex, see \cite{Kap,KT,Laf}.  However, it is a difficult question to describe which matroid subdivisions of the hypersimplex occur in this way.  This is a variant of the (also difficult) question of which matroids are realizable.

When $X \in \Ch(k,n)_{\geq 0}$ is nonnegative, the matroid polytopes $P_1,\ldots,P_m$ are positroid polytopes (Proposition~\ref{prop:conf}).  Since positroids have been completely classified, subdivisions of the hypersimplex by positroid polytopes are far more tractable.  Indeed, Theorem~\ref{thm:introbij} states that any regular subdivision of the hypersimplex $\Delta(k,n)$ into positroid polytopes appears in nonnegative configuration space.

\subsection{}   By definition, a regular subdivision of the hypersimplex arises from a weight vector $p_\bullet \in \R^{\binom{[n]}{k}}$: we lift each vertex $e_I$ of $\Delta(k,n)$ to height $p_I$, and project the lower faces of the resulting convex hull down to obtain a subdivision $\tDelta(p_\bullet)$.  Speyer \cite{Spe} showed that the subdivision $\tDelta(p_\bullet)$ is into matroid polytopes if and only if $p_\bullet$ satisfies the three-term tropical Pl\"ucker relations.  

We say that $p_\bullet$ satisfies the three-term positive tropical Pl\"ucker relations if for every $S$ and $a<b<c<d$ not contained in $S$, we have that
\begin{equation}\label{eq:trop}
p_{Sac} + p_{Sbd} = \min(p_{Sab} + p_{Scd} , p_{Sad} + p_{Sbc}).
\end{equation}
We show in Proposition~\ref{prop:posdiv} that the subdivision $\tDelta(p_\bullet)$ is into positroid polytopes if and only if $p_\bullet$ satisfies \eqref{eq:trop}.  Following \cite{HJJS,HJS,OPS2}, we call the space of vectors satisfying \eqref{eq:trop} the positive Dressian, and denote it by $\Pluck(k,n)_{>0}$.  More generally, for each positroid $\M$, we have a positive local Dressian $\Pluck(\M)_{>0}$.

\subsection{} Speyer and Sturmfels \cite{SS} studied the tropical Grassmannian $\Trop\, \Gr(k,n)$ which parametrizes tropical linear spaces.  Every point $p_\bullet \in \Trop\, \Gr(k,n)$ satisfies the three-term tropical Pl\"ucker relations, but in general the converse is not true \cite{HJJS}.

Speyer and Williams \cite{SW} defined the positive tropical Grassmannian $\Trop_{>0}\Gr(k,n) \subset \Trop \,\Gr(k,n)$.  Let $\RR= \bigcup_{n=1}^\infty \R((t^{1/n})) $ denote the field of Puiseux series and $\RR_{>0} \subset \RR$ those Puiseux series whose leading (lowest) coefficient is positive.  Then $\Trop_{>0}\Gr(k,n)$ is defined to be the closure of the set of valuations $p_\bullet = (p_I = \val( \Delta_I(V)) \mid I \in \binom{[n]}{k})$ for $V \in \Gr(k,n)(\RR_{>0})$.
We generalize this by also considering the positive tropical positroid cell $\Trop_{>0} \Pi_\M$.  It is immediate that every point $p_\bullet \in \Trop_{>0} \Gr(k,n)$ or $p_\bullet \in \Trop_{>0} \Pi_\M$ satisfies the three-term positive tropical Pl\"ucker relations \eqref{eq:trop}.  A key technical result is that the converse holds (Theorem~\ref{thm:main}).

\begin{theorem}\label{thm:mainintro}
Every rational positive tropical Pl\"ucker vector is realizable.  
Thus 
$$
\Trop_{> 0}\Gr(k,n) = \Pluck(k,n)_{>0}, \;\;\;\Trop_{>0} \Pi_{\M} = \Pluck(\M)_{>0}.
$$
\end{theorem}
%
%
%We show that the converse holds, i.e., every vector satisfying \eqref{eq:trop} is ``realizable" as a point in the positive tropical Grassmannian.  Indeed, we generalize (Theorem~\ref{thm:main}) this statement to the positive tropicalization $\Trop_{>0}(\Pi_\M)$ of each positroid variety $\Pi_\M$.
We give two proofs of Theorem~\ref{thm:mainintro}.  The first one uses \eqref{eq:trop} directly. The second one uses a tropical bridge reduction (Section \S\ref{sec:bridge}) for positive tropical Pl\"ucker vectors, a variant of the bridge reduction algorithm for points $V \in \Gr(k,n)_{\geq 0}$ in \cite{LamCDM}.  Whereas usual bridge reduction gives parametrizations of $\Pi_{\M,>0}$, tropical bridge reduction gives parametrizations of $\Pluck(\M)_{>0} = \Trop_{>0}\Pi_\M$.

The space $\Pluck(\M)_{>0}$ has a number of different fan structures.  Two of them, the secondary fan structure and the Pl\"ucker fan structure, were shown to agree in \cite{OPS2}.  We study a class of positive fan structures (generalizing \cite{SW}) coming from positive parametrizations of $\Pi_{\M,>0}$ and show in Theorem~\ref{thm:coincide} that all the fan structures coincide.

\subsection{} Let us give an example illustrating the bijections of Theorem~\ref{thm:introbij}.  Full definitions are given in the main text.  Consider the curve $V: [0,\infty) \to \Gr(2,5)$ given by
$$
V(t) = \begin{bmatrix} 0 & 1&  1& 1&1\\
-1 & 0& 1& 1+ t & 1+2t  
\end{bmatrix}
$$
which has Pl\"ucker coordinates $\Delta_{12}=\Delta_{13}=\Delta_{14}=\Delta_{15}=\Delta_{23}=1$ and
$$
\Delta_{24}=1+t \qquad  \Delta_{25} = 1+2t \qquad \Delta_{34} = t \qquad \Delta_{35}=2t  \qquad \Delta_{45}=t.
$$
Thus the curve $V(t)$ lies in the positive Grassmannian $\Gr(2,5)_{>0}$ for $t > 0$, and defines a curve $X(t) \in \Conf(2,5)_{>0} = \Gr(2,5)_{>0}/T_{>0}$ in the positive component of generic configuration space.  Let $X:= \lim_{t \to 0} X(t) \in \Ch(2,5)_{\geq 0}$.  We have 
\begin{align*}
V_1&:=\lim_{t \to 0} V(t) =  \begin{bmatrix} 0 & 1&  1& 1&1\\
-1 & 0& 1& 1 & 1
\end{bmatrix} 
\end{align*}
and setting $V'(t) = V(t) \cdot  {\rm diag}(t^{1/2},t^{1/2},t^{-1/2},t^{-1/2},t^{-1/2}) $, we have 
\begin{align*}
V_2:= \lim_{t \to 0}V'(t) &= \lim_{t\to 0}\begin{bmatrix} 0 & \sqrt{t}&  1/\sqrt{t}& 1/\sqrt{t}&1/\sqrt{t}\\
-\sqrt{t} & 0& 1/\sqrt{t}& (1+t)/\sqrt{t} & (1+2t)/\sqrt{t}
\end{bmatrix} \\
&=\lim_{t\to 0}\begin{bmatrix} 1 & 1&  0& -1&-2\\
-t & 0& 1& 1+t& 1+2t
\end{bmatrix}\\
&=\begin{bmatrix} 1 & 1&  0& -1&-2\\
0& 0& 1& 1& 1
\end{bmatrix}.
\end{align*}
The points $V_1$ and $V_2$ have matroids
$$
\M_1:=\mbox{$\binom{[5]}{2}$} \setminus \{34,35,45\} \qquad \text{and} \qquad \M_2 :=\mbox{$\binom{[5]}{2}$} \setminus \{12\} 
$$
respectively.  The matroid polytopes $P_{\M_1}$ and $P_{\M_2}$ give a decomposition of the hypersimplex $\Delta(2,5)$ into positroid polytopes: this decomposition comes from slicing $\Delta(2,5)$ using the hyperplane $x_1+x_2 =1$ (or, equivalently $x_3+x_4+x_5 = 1$).  It follows that $X \in \Ch(2,5)_{\geq 0}$ is represented by the union of the two toric varieties $\overline{T \cdot V_1}$ and $\overline{T \cdot V_2}$ (inside $\Gr(2,5)$), which have moment polytopes $P_{\M_1}$ and $P_{\M_2}$ respectively.

The valuation $\val(f(t))$ of a polynomial (or formal power series) $f(t)$ is the degree of the lowest term in $f(t)$.  Defining $p_I:=\val(\Delta_I(V(t)))$ we obtain a vector $p_\bullet \in \R^{\binom{[5]}{2}}$, given by $p_{34}=p_{35}=p_{45}=1$ and $p_J = 0$ for $J \notin \{34,35,45\}$.  By definition, $p_\bullet$ lies the positive tropical Grassmannian.  It is also easy to check that $p_\bullet$ satisfies \eqref{eq:trop}.  For example, $0 = p_{13}+p_{25} = \min(p_{12}+p_{35},p_{15}+p_{23}) = \min(1,0) = 0$.  This induces the bijections of Theorem~\ref{thm:introbij}.

\subsection{}
Finally, let us explain some ingredients of the proof of Theorem~\ref{thm:introtop}.
We use the notion of nearly convergent functions on $\Ch(k,n)$ (the nomenclature comes from the stringy integrals of \cite{AHL1}).  These are certain $T$-invariant, subtraction-free, rational functions on the Grassmannian whose tropicalizations take nonnegative values.  The ring $\C[\Gamma]$ generated by nearly convergent functions is isomorphic to the coordinate ring of an affine open subset $X'_{P(k,n)}$ of a projective toric variety $X_{P(k,n)}$ (Proposition~\ref{prop:XPkn}) associated to some polytope $P(k,n)$.  We obtain a morphism 
$$
\varphi: \tConf(k,n) \longrightarrow X'_{P(k,n)}
$$
from an open subset $\tConf(k,n) \subset \Ch(k,n)$ of the Chow quotient to $X'_{P(k,n)}$.  We show that the restriction of $\varphi$ to the nonnegative part $\Ch(k,n)_{\geq 0}$ is a homeomorphism onto the nonnegative part $X_{P(k,n),\geq 0}$ of the toric variety, which is known to be homeomorphic to the polytope $P(k,n)$.

\smallskip
\noindent
{\bf Organization.} In \S\ref{sec:matroid}, we discuss matroids, positroids, and their matroid polytopes.  In \S\ref{sec:Gr}, we discuss the Grassmannian, configuration space, and the positroid stratification.  In \S\ref{sec:cluster}, we review cluster parametrizations of positroid cells.  In \S\ref{sec:def} and \S\ref{sec:strata}, we introduce our main object of interest: the nonnegative configuration space and its stratification by positive Chow cells.  In \S\ref{sec:Plucker} and \S\ref{sec:hypersimplex}, we study positive tropical vectors and positroid subdivisions of the hypersimplex.  In \S\ref{sec:trop}, we show that the positive Dressian and the positive tropical Grassmannian agree.  In \S\ref{sec:fan} we show that a number of fan structures on the positive Dressian coincide.  In \S\ref{sec:Gamma}, we introduce nearly convergent functions.  In \S\ref{sec:ball}, we prove that $\Ch(k,n)_{\geq 0}$ is homeomorphic to a ball.  The case $k = 2$ is studied as an example in \S\ref{sec:M0n}.  In \S\ref{sec:bridge}, we introduce and study tropical bridge operations.  \S\ref{sec:connected} contains some technical statements concerning connected positroids.
Appendix~\ref{sec:examples} contains data for the cases $(k,n) = (3,6), (3,7), (3,8)$.

\smallskip
\noindent
{\bf Acknowledgements.} We thank Song He for conversations related to this work.  We are grateful to Lauren Williams for a correction (see Remark~\ref{rem:Lauren}).  T.L. was supported
by NSF DMS-1464693, NSF DMS-1953852, and by a von Neumann Fellowship from the Institute for Advanced Study.
N.A-H. and M.S. were supported by DOE grants DE-SC0009988 and DE-SC0010010 respectively.

\begin{remark}
Many of the results in this work were announced at Amplitudes 2019 \cite{Namp}. Some of the results in \S\ref{sec:Plucker}--\S\ref{sec:trop} regarding the positive tropical Grassmannian, the positive Dressian, and positroidal subdivisions of the hypersimplex are not surprising to experts and overlap with independent recent work in \cite{Ola,Ear,Ear2,LPW,SW2}. For instance, Proposition \ref{prop:posSpeyer} is closely related to \cite[Theorem 3.8]{LPW} and Proposition \ref{prop:posdiv}(2) is \cite[Theorem 9.12]{LPW}.
\end{remark}

\section{Matroids and positroids}\label{sec:matroid}
\subsection{}
A matroid $\M \subset \binom{[n]}{k}$ of rank $k$ on $[n]$ is a nonempty collection of $k$-element subsets of $[n]$, called {\it bases}, satisfying the {\it exchange axiom}:
\begin{equation}\label{eq:exchange}
\mbox{if $I, J \in \M$ and $i \in I$ then there exists $j \in J$ such that $I \setminus \{i\} \cup \{j\} \in \M$.}
\end{equation}
The {\it uniform matroid} is the collection $\M = \binom{[n]}{k}$ of all $k$-element subsets of $[n]$. 

\subsection{} The {\it matroid polytope} $P_\M$ of a matroid in $\M$ is the convex hull of the vectors $e_I$, for $I \in \M$.  Here, $e_I = e_{i_1} + e_{i_2} + \cdots + e_{i_k}$ is the sum of $k$ basis vectors, where $I = \{i_1,\ldots,i_k\}$.  Thus the matroid polytope of the uniform matroid is the hypersimplex $\Delta(k,n)$, whose vertices are exactly the $0$-$1$ vectors with $k$ $1$-s and $(n-k)$ $0$-s.  We have the following characterization of matroid polytopes.

\begin{proposition}[\cite{GGMS}] \label{prop:GGMS}
A polytope $P \subset \R^n$ is the matroid polytope of a matroid of rank $k$ on $[n]$ if and only if its vertex set is a subset of $\{e_I \mid I \in \binom{[n]}{k}\}$, and all edges of $P$ are in the direction of $e_i - e_j$, for $i \neq j$.
\end{proposition}

\subsection{}
If $\M_1$ is a matroid on a set $S_1$ and $\M_2$ a matroid on $S_2$, then the {\it direct sum }$\M_1 \oplus \M_2$ is a matroid on the disjoint union $S_1 \sqcup S_2$, given by
$$
\M_1 \oplus \M_2 = \{I_1 \sqcup I_2 \mid I_i \in \M_i\}.
$$
We say that a matroid $\M$ of rank $k$ on $[n]$ is \emph{connected} if the matroid polytope $P_\M$ is of full dimension, that is, has dimension $n-1$.  This is equivalent to the condition that $\M$ is not a non-trivial direct sum of smaller matroids.

\subsection{}
 The {\it Bruhat partial order} on $\binom{[n]}{k}$ is defined as follows.  For two subsets $I, J \in \binom{[n]}{k}$, we write $I \leq J$ if $I = \{i_1 < i_2 <\cdots < i_k\}$, $J = \{j_1 < j_2 < \cdots < j_k\}$ and we have $i_r \leq j_r$ for $r = 1,2,\ldots,k$.  
For $I \in \binom{[n]}{k}$, the {\it Schubert matroid} $\SS_I$ is defined as 
$$
\SS_I:= \{J \in \sbinom \mid I \leq J\}
$$
and has minimal element $I$.  For $a \in [n]$, let $\leq_a$ denote the cyclically rotated order on $[n]$ with minimum $a$, which induces a partial order $\leq_a$ on $\binom{[n]}{k}$.  Let $\SS_{I,a} := \{J \in \binom{[n]}{k} \mid I \leq_a J\}$ denote the cyclically rotated Schubert matroid.

\subsection{}
A {\it $(k,n)$-Grassmann necklace} $\I = (I_1,I_2,\ldots,I_n)$ \cite{Pos} is a $n$-tuple of $k$-element subsets of $[n]$ satisfying the following condition: for each $a \in [n]$, we have
\begin{enumerate}
\item $I_{a+1} = I_a$ if $a \notin I_a$, 
\item $I_{a+1} = I_a - \{a\} \cup \{a'\}$ if $a \in I_a$,
\end{enumerate}
with indices taken modulo $n$.  A {\it positroid} $\M$ is the matroid of a totally nonnegative point in the Grassmannian, and are in bijection with Grassmann necklaces.  

\begin{proposition}[\cite{Oh, Pos}] \label{prop:Oh}
Let $\I= (I_1,I_2,\ldots,I_n)$ be a $(k,n)$-Grassmann necklace.  Then the intersection of cyclically rotated Schubert matroids
\begin{equation}\label{eq:MI}
\M_\I = \SS_{I_1,1} \cap \SS_{I_2,2} \cap \cdots \cap \SS_{I_n,n}
\end{equation}
is a positroid, and the map $\I \mapsto \M_\I$ gives a bijection between $(k,n)$-Grassmann necklaces and  positroids of rank $k$ on $[n]$.
\end{proposition}

\subsection{}
Let $\M$ be an arbitrary matroid of rank $k$ on $[n]$ and $a \in [n]$.  Then $\M$ has a minimum with respect to $\leq_a$, which is denoted $I_a(\M)$.  

\begin{proposition}
The $n$-tuple $\I(\M)=(I_1(\M),I_2(\M),\ldots,I_n(\M))$ is a $(k,n)$-Grassmann necklace.
\end{proposition}
 We call $\M_{\I(\M)}$ the {\it positroid envelope} of $\M$ \cite{Pos,KLS}.  A matroid $\M$ equals its positroid envelope if and only if $\M$ is a positroid.
 
 \subsection{}
 A {\it $(k,n)$-bounded affine permutation} is a bijection $f:\Z \to \Z$ satisfying the conditions
 \begin{enumerate}
 \item $f(i+n) = f(i) + n$ for all $i \in \Z$,
 \item $i \leq f(i) \leq i+n$ for all $i \in \Z$,
 \item $\sum_{i=1}^n (f(i)-i) = kn$. 
 \end{enumerate}
 
 Given a $(k,n)$-Grassmann necklace $\I = (I_1,I_2,\ldots,I_n)$, we define $f_\I:\Z \to \Z$ by 
$$
 f_\I(a) = \begin{cases} a & \mbox{if $a \notin I_a$} \\
 a+n & \mbox{if $a \in I_a \cap I_{a+1}$} \\
 a' & \mbox{if $I_{a+1} =  I_a - \{a\} \cup \{a'\}$ and $a' > a$}\\
 a'+n& \mbox{if $I_{a+1} =  I_a - \{a\} \cup \{a'\} $ and $a' < a$}
 \end{cases}
 $$
 for $a =1,2,\ldots,n$, and extending the domain to $\Z$ by setting $f(i+n) = f(i) + n$ for all $i \in \Z$.
 
\begin{proposition}\cite{KLS} \label{prop:KLS}
For any $(k,n)$-Grassmann necklace $\I$, the function $f_\I$ is a $(k,n)$-bounded affine permutation.  The map $\I \mapsto f_\I$ gives a bijection between $(k,n)$-Grassmann necklaces and $(k,n)$-bounded affine permutations.
\end{proposition}

Thus we have bijections between positroids of rank $k$ on $[n]$, and $(k,n)$-Grassmann necklaces, and 
$(k,n)$-bounded affine permutations.  We write $f_\M:= f_{\I(\M)}$.

\subsection{}
A polytope $P$ in $\{(x_1,\ldots,x_n) \mid x_1+x_2+\cdots+x_n = k\} \subset \R^n$ is called {\it alcoved} \cite{LP1} if it is given by the intersection of half spaces of the form
$$
H = \{(x_1,\ldots,x_n) \mid \sum_{i \in [a,b]} x_i \geq c\}
$$
where $[a,b] \subset [n]$ is a cyclic interval.  The following two results are a special case of the theory of polypositroids \cite{LP}, see also \cite[Proposition 5.5 and Corollary 5.4]{ARW}.

\begin{proposition}[\cite{LP}] \label{prop:posfacet}
Let $\M$ be a matroid with matroid polytope $P_\M$.  Then $P_\M$ is alcoved if and only if $\M$ is a positroid.
\end{proposition}
\begin{proof}
The matroid polytope $P_{\SS_{I,a}}$ of the rotated Schubert matroid $\SS_I$ is the intersection of the hypersimplex $\Delta(k,n)$ with the inequalities
$$
x_a + x_{a+1} + \cdots + x_b \geq \#(I \cap [a,b])
$$
for $i = 1,2,\ldots,n$.  By definition, this is an alcoved polytope, and so is the intersection $P_{\SS_{I_1,1}} \cap P_{\SS_{I_2,2}} \cap \cdots \cap P_{\SS_{I_n,n}}$.  Since every positroid is of the form \eqref{eq:MI}, every positroid polytope is alcoved. 

Now let $P_\M$ be the matroid polytope of an arbitrary matroid.  Then the smallest alcoved polytope $P$ containing $P_\M$ is the intersection of the rotated Schubert matroid polytopes $P_{\SS_{I_a(\M),a}}$ for $a =1,2,\ldots,n$.  Thus $P$ is the matroid polytope of the positroid envelope of $\M$.  In particular, if $\M$ is itself a positroid then $P_\M$ is alcoved.
\end{proof}

\begin{corollary}
Every face of a positroid polytope is itself a positroid polytope.
\end{corollary}
\begin{proof}
Every face of a matroid polytope (resp. alcoved polytope) is a matroid polytope (resp. alcoved polytope).
\end{proof}

A {\it noncrossing partition} $(S_1,\ldots,S_r)$ of $[n]$ is a partition of $[n]$ into disjoint sets such that there do not exist $a<b<c<d$ such that $a,c \in S_i$ and $b,d \in S_j$ for $i \neq j$.  

\begin{proposition}[{\cite[Theorem 7.6]{ARW}}]
\label{prop:disconnect} 
Let $f = f_\M$ be the bounded affine permutation associated to a positroid $\M$.  Suppose that $\M = \M_1 \oplus \M_2 \oplus \cdots \oplus \M_r$, where $\M_i$ is a positroid on the ground set $S_i \subset [n]$.  Then $(S_1,\ldots,S_r)$ form a noncrossing partition of $[n]$, and $(f_\M(S_i) \mod n) = S_i$ for $i = 1,2,\ldots,r$.  In particular, the connected components of the positroid $\M$ are themselves positroids.

Conversely, let $(S_1,\ldots,S_r)$ be a noncrossing partition of $[n]$ and $\M_i$ be a positroid on the ground set $S_i$.
Then the direct sum $\M_1 \oplus \M_2 \oplus \cdots \oplus \M_r$ is a positroid.
\end{proposition}

\section{Grassmannians and configuration spaces}\label{sec:Gr}
\subsection{}
Let $\Gr(k,n)$ denote the Grassmannian of $k$-planes in $\C^n$.  For $V \in \Gr(k,n)$ we let $\Delta_I(V)$, for $I \in \binom{[n]}{k}$ denote its {\it Pl\"ucker coordinates}.  These Pl\"ucker coordinates satisfy the Pl\"ucker relations and are defined up to a common scalar.  We refer the reader to \cite{LamCDM} for further details.  The most important relation for us will be the three-term Pl\"ucker relation
\begin{equation}\label{eq:threeterm}
\Delta_{Sac}\Delta_{Sbd} = \Delta_{Sab}\Delta_{Scd} + \Delta_{Sad} \Delta_{Sbc}
\end{equation}
where $S \subset [n]$ is of size $k-2$ and $a<b<c<d$ are not contained in $S$.

It is convenient to also work with the affine cone $\hGr(k,n)$ over $\Gr(k,n)$.  A point $V$ in $\hGr(k,n)$ is a collection of Pl\"ucker coordinates $\Delta_I(V)$ that satisfy the Pl\"ucker relations.  But now scaling Pl\"ucker coordinates give different points, and $\hGr(k,n)$ also contains a distinguished cone point $0$, where all Pl\"ucker coordinates vanish.

\subsection{}  

The {\it matroid} $\M(V)$ of $V \in \Gr(k,n)$ is defined as
$$
\M(V) := \left\{I \in \sbinom \mid \Delta_I(V) \neq 0\right\}.
$$
Similarly, one can define $\M(V)$ for $V \in \hGr(k,n) \setminus \{0\}$.  If $V = 0$ is the cone point, then $\M(V)$ is not defined.  For a matroid $\M$, define the matroid strata
$$
\G_\M:= \{V \in \Gr(k,n) \mid \M(V) = \M\}$$
For the uniform matroid, we have
$$
\oGr(k,n):=\G_{\binom{[n]}{k}} = \left\{V \in \Gr(k,n) \mid \Delta_I(V) \neq 0 \text{ for all } I \in \sbinom\right\}.
$$

\subsection{}
The torus $(\C^\times)^n = \{(t_1,t_2,\ldots,t_n) \mid t_i \in \C^\times\}$ acts on $\Gr(k,n)$ by scaling the $i$-th column of a representing matrix by $t_i$.  For $t \in \C^\times$, the element $(t,t,\ldots,t) \in (\C^\times)^n$ scales all Pl\"ucker coordinates by $t^k$, and thus the action of $(\C^\times)^n$ factors through the torus
$$
T  := (\C^\times)^n/\C^\times \simeq  (\C^\times)^{n-1}.
$$ 
An element $(t_1,\ldots,t_n) \in (\C^\times)^n$ acts on $V \in \hGr(k,n)$ by 
$$
\Delta_I((t_1,\ldots,t_n) \cdot V) = \prod_{i\in I} t_i \, \Delta_I(V).
$$
This action factors through the quotient torus
$$
\hT = (\C^\times)^{n}/(\Z/k\Z)
$$
where $\Z/k\Z = \{(\zeta,\ldots,\zeta) \mid \zeta^k = 1\}$ is a cyclic group of order $k$.  The character lattices $X(T)$ and $X(\hT)$ of $T$ and $\hT$ are naturally identified with sublattices of $\Z^n = X((\C^\times)^n)$:
\begin{align}
\begin{split}\label{eq:character}
X(T) &= \{(x_1,\ldots,x_n) \in \Z^n \mid \sum x_i = 0\} \\
X(\hT) &= \{(x_1,\ldots,x_n) \in \Z^n \mid k \text{ divides } \sum x_i \}
\end{split}.
\end{align}

\begin{lemma}\label{lem:connected}
The matroid $\M$ is connected if and only if the action of $T$ on $\G_\M$ is free.
\end{lemma}
\begin{proof}
Suppose $\M$ is connected.  We may find a vertex $e_I$ of $P_\M$ and vertices $e_{J_1},\ldots,e_{J_{n-1}}$ connected to $e_I$ via an edge of $P_\M$, so that the span of $\{e_I,e_{J_1},\ldots,e_{J_{n-1}}\}$ is linearly independent.  By Proposition~\ref{prop:GGMS}, the edges of $P_\M$ are roots, i.e., vectors of the form $e_i - e_j$.  A collection of linearly independent roots is easily seen to be unimodular, i.e., their integral span is the lattice $\{(x_1,\ldots,x_n) \in \Z^n \mid \sum_i x_i = 0\}$, which is the character lattice $X(T)$ of $T$.  Points $V \in \G_\M$ can be gauge-fixed to satisfy $\Delta_I(V) = 1$.  Under this gauge-fix, the torus $T$ acts on the Pl\"ucker coordinates $\Delta_{J_1},\ldots,\Delta_{J_{n-1}}$ with weights $e_{J_1}-e_I,\ldots,e_{J_{n-1}}-e_I$.  We have just argued that these weights span the character lattice $X(T)$, and thus $T$ must act freely on $\G_\M$.

The ``if" direction is similar.
\end{proof}

\subsection{} 
The torus $T$ acts freely on the complex manifold $\oGr(k,n)$, and thus the quotient $\oGr(k,n)/T$ is a manifold of dimension $k(n-k) - (n-1)$.  The space $\oGr(k,n)/T$ has the following alternative description.  Let $\Conf(k,n)$ be the space $(\PP^{k-1})^n/\GL(n)$ of $\GL(n)$-orbits of $n$ points in the projective space $\PP^{k-1}$.  A configuration $p = (p_1,\ldots,p_n)$ is called {\it generic} if any $r$ of the points, for $r \leq k$, affinely span a subspace of $\PP^{k-1}$ of dimension $r-1$.  We denote the space of generic configurations by $\oConf(k,n) \subset \Conf(k,n)$.  We have an isomorphism $\oGr(k,n)/T \simeq \oConf(k,n)$.

\subsection{} 
The {\it totally nonnegative Grassmannian} $\Gr(k,n)_{\geq 0}$ is the subspace of $\Gr(k,n)$ consisting of points $V \in \Gr(k,n)(\R)$ all of whose Pl\"ucker coordinates are nonnegative.  The {\it totally positive Grassmannian} $\Gr(k,n)_{>0}$ consists of points all of whose Pl\"ucker coordinates are positive.  These spaces were defined by Lusztig \cite{Lus} and Postnikov \cite{Pos}.  The totally nonnegative Grassmannian $\Gr(k,n)_{\geq 0}$ is homeomorphic to a closed ball \cite{GKL}. 

The matroid of a point $V \in \Gr(k,n)_{\geq 0}$ is called a {\it positroid}.  Positroids can be indexed by $(k,n)$-Grassmann necklaces (Proposition~\ref{prop:Oh}) and $(k,n)$-bounded affine permutations (Proposition~\ref{prop:KLS}).  The matroid of $V \in \Gr(k,n)_{>0}$ is the uniform matroid, and thus we have $\Gr(k,n)_{>0} \subset \oGr(k,n)$.  Indeed, $\Gr(k,n)_{>0}$ is a connected component of the manifold $\oGr(k,n)$, and is diffeomorphic to an open ball of dimension $k(n-k) -(n-1)$.  The image of $\Gr(k,n)_{>0}$ in $\oGr(k,n)/T \simeq \oConf(k,n)$ is called the {\it positive component} of (generic) configuration space, and denoted $\Conf(k,n)_{>0}$.

\subsection{}
For a positroid $\M$, define the {\it positroid cell}
$$
\Pi_{\M,>0} := \{V \in \Gr(k,n)_{\geq 0} \mid \M(V) = \M\}.
$$
By \cite{Pos}, $\Pi_{\M,>0}$ is homeomorphic to an open ball.  We define $\dim(\M)$ to be the dimension of this ball.  We have the disjoint union \cite{Pos}
\begin{equation}\label{eq:positroiddecomp}
\Gr(k,n)_{\geq 0} = \bigsqcup_\M \Pi_{\M,>0}
\end{equation}
as $\M$ varies over positroids of rank $k$ on $[n]$.  The closure $\Pi_{\M,\geq 0}:= \overline{\Pi_{\M,>0}}$ is a union 
$
\Pi_{\M,\geq 0} = \bigsqcup_{\M' \subseteq \M} \Pi_{\M',>0}
$
of positroid cells.  The decomposition \eqref{eq:positroiddecomp} gives $\Gr(k,n)_{\geq 0}$ the structure of a regular CW complex \cite{GKL2}. 

\subsection{}
The group $T_{>0}:= \R_{>0}^n/\R_{>0} \subset (\C^\times)^{n}/(\C^\times) = T$ acts on $\Gr(k,n)_{\geq 0}$, preserving the positroid cells $\Pi_{\M,>0}$. 
From Lemma \ref{lem:connected}, we have the following.
\begin{lemma}
Suppose that $\M$ is a connected positroid.  Then $T_{>0}$ acts freely on $\Pi_{\M,>0}$.
\end{lemma}
The quotient $\Pi_{\M,>0}/T_{>0}$ is a real manifold of dimension $\dim(\M)-(n-1)$.  In particular, if $\M$ is a connected positroid then $\dim(\M) \geq n-1$.
If $\M$ is a connected positroid and $\dim(\M) = n-1$, then we call $\M$ a {\it minimal connected positroid} (see Lemma~\ref{lem:mintree}).  In this case, $\Pi_{\M,>0}/T_{>0}$ is a single point.

\subsection{}
For a positroid $\M$, define the {\it positroid variety} $\Pi_\M$ to be the Zariski-closure of $\Pi_{\M,>0}$ inside $\Gr(k,n)$.  The variety $\Pi_\M$ is an irreducible, normal subvariety of $\Gr(k,n)$ of dimension $\dim(\M)$ \cite{KLS}.  Positroid varieties give a stratification of $\Gr(k,n)$.  We define the {\it open positroid variety} $\oPi_\M \subset \Pi_\M$ so that
$
\Pi_{\M} = \bigsqcup_{\M' \subseteq \M} \oPi_{\M'}.
$
We have $\Gr(k,n)_{\geq 0} \cap \oPi_{\M'} = \Pi_{\M',>0}$.  We caution that $\oPi_\M$ is not equal to $\G_\M$.  Instead, $\oPi_\M$ is the union of $\G_{\M'}$ over all matroids $\M'$ such that the positroid envelope of $\M'$ is equal to $\M$.

\subsection{}\label{ssec:bridgereduce}
We recall the bridge parametrizations of $\Pi_\M$ from \cite{LamCDM}, see also \cite{Kar}.  The group $\GL(n)$ acts on $\Gr(k,n)$ by right multiplication.  For $i = 1,2,\ldots, n-1$, let $x_i(t) \in \GL(n)$ be the matrix that differs from the identity in a single entry equal to $t$ in the $(i,i+1)$-th matrix entry.  For $i = n$, we let $x_n(t) \in \GL(n)$ be the matrix that differs from the identity in a single entry equal to $(-1)^{k-1} t$ in the $(n,1)$-th matrix entry.  Thus $x_i(t)$ acts on a $k \times n$ matrix $V$ by adding the $i$-th column to the $(i+1)$-th column (with a sign if $i = n$).
The action of $x_i(t)$ can be written in Pl\"ucker coordinates as (cf. \cite[Lemma 7.6]{LamCDM})
\begin{equation}\label{eq:Deltax}
\Delta_I(V \cdot x_i(t)) = \begin{cases} \Delta_I(V) + t \Delta_{I \setminus \{i+1\} \cup \{i\}}(V) &\mbox{if $i+1 \in I$ but $i \notin I$} \\
\Delta_I(V) & \mbox{otherwise.}
\end{cases}
\end{equation}
More generally, for $\gamma=(i,j)$ let $x_\gamma(t)$ be the matrix that differs from the identity in a single entry equal to $\pm t$ in the $(i,j)$-th matrix entry, taking the positive sign if $i < j$ and the sign $(-1)^{k-1}$ if $i > j$.  

The following result allows us to reduce totally nonnegative points recursively.

\begin{proposition}[\cite{LamCDM}]\label{prop:bridgedecomp}
Let $V \in \Pi_{\M,>0} \subset \Gr(k,n)_{\geq 0}$ where $k > 0$.  Then at least one of the following holds:
\begin{enumerate}
\item For some $i \in [n]$, we have $f_\M(i) = i$.  Then $V$ is in the image of the map $\kappa_i: \Gr(k,n-1)_{\geq 0} \hookrightarrow \Gr(k,n)_{\geq 0}$ obtained by adding an $i$-th column equal to 0.
\item For some $i \in [n]$, we have $f_\M(i) = i+n$.  Then $V$ is in the image of the map $\eta_i: \Gr(k-1,n-1)_{\geq 0} \hookrightarrow \Gr(k,n)_{\geq 0}$ obtained by adding an extra first row and an extra $i$-th column, placing 0-s in all the new entries except for a $1$ in the $(1,i)$-th entry, and finally multiplying the columns $1,2,\ldots,i-1$ by $(-1)^{k-1}$.
\item For some $i \in [n]$, we have $i+1 \leq f_\M(i) < f_\M(i+1) \leq i+n$.  Then $V = V' \cdot x_i(a)$ where $V' \in \Pi_{\M',>0}$ with $\M'$ the positroid satisfying $f_{\M'} = f_\M s_i$ (where $s_i$ is the simple transposition swapping $i$ and $i+1$), and 
$$
a = \frac{\Delta_{I_{i+1}}(V)}{\Delta_{I_{i+1} \setminus\{i+1\} \cup \{i\}}(V)} >0.
$$
\end{enumerate}
\end{proposition}
Proposition~\ref{prop:bridgedecomp} allows us to reduce any $V \in \Gr(k,n)_{\geq 0}$ to a positroid stratum that has dimension 0.  The recursion can be chosen to only depend on $\M = \M(V)$, with the real numbers in (3) taken to be parameters.  Parametrizations of $\R_{>0}^{\dim(\M)}  \simeq \Pi_{\M,>0}$ obtained in this way are called {\it bridge parametrizations}.  Note that $\kappa_{i+1} \circ x_i(t) = x_{i,i+2}(t) \circ \kappa_{i+1}$, so in general we need to use the matrices $x_\gamma(t)$.  Bridge parametrizations are of the form
\begin{equation}\label{eq:bridgeparam}
\R_{>0}^d \ni (t_1,t_2,\ldots,t_d) \longmapsto x_{\gamma_1}(t_1) \cdots x_{\gamma_d}(t_d) \cdot x_I \in \Pi_{\M,>0}
\end{equation}
where $d = \dim(\M)$ and $x_{i_1,\ldots,i_k} = {\rm span}(e_{i_1},\ldots,e_{i_k}) \in \Gr(k,n)^T$ is a torus fixed point.  We caution that in general $x_\gamma(t)$ for $t > 0$ does not preserve total nonnegativity.  They do preserve total nonnegativity when used in a bridge parametrization.

\section{Clusters for positroids}\label{sec:cluster}
\subsection{}\label{ssec:weaksep}
Let $\M$ be a positroid.  Recall that $\oPi_\M \subset \Gr(k,n)$ denotes the open positroid variety.  Let $\tPi_\M \subset \hGr(k,n)$ denote the cone over $\oPi_\M$.  By \cite{GL}, the coordinate rings $\C[\oPi_\M]$ and $\C[\tPi_\M]$ are isomorphic to cluster algebras.\footnote{In \cite{GL}, the ring $\C[\oPi_\M]$ is considered.  Working with $\C[\tPi_\M]$ allows us to avoid gauge-fixing a Pl\"ucker coordinate to equal to 1.}

A {\it cluster $\CC$ for $\M$} is a subset of $\M$ that indexes a seed in the cluster structure of $\C[\tPi_\M]$.  If $\M = \binom{[n]}{k}$ is the uniform matroid, then we simply call $\CC$ a {\it cluster}.  Any cluster $\CC$ for $\M$ has cardinality $|\CC| = \dim(\M) + 1$.  (We consider only clusters where the cluster variables are Pl\"ucker coordinates coming from face labels of a plabic graph.)  
Clusters for $\M$ can be described via weak separation \cite{OPS}.  We say that two subsets $I, J \subset [n]$ are {\it weakly-separated} if we cannot find $1\leq a < b < c < d \leq n$ such that $a,c \in I \setminus J$ and $b,d \in J \setminus I$ (or with $I$ and $J$ swapped).  The following result can be taken to be the definition of a cluster.

\begin{proposition}[\cite{OPS}]\label{prop:OPS}
A subset $\CC \subset \M$ is a cluster if it is pairwise weakly-separated, has size $\dim(\M) + 1$, and contains the Grassmann necklace $\I$ of $\M$.  Any pairwise weakly-separated subset of $\binom{[n]}{k}$ can be extended to a cluster.
\end{proposition}
Every $I \in \binom{[n]}{k}$ belongs to some cluster $\CC \subset \binom{[n]}{k}$, but this is not true with an arbitrary positroid $\M$ replacing $\binom{[n]}{k}$.

\subsection{} \label{ssec:mutation}
Let $S \subset [n]$ be of size $k-2$ and $a<b<c<d$ numbers not contained in $S$.
Let $\CC \in \M$ be a cluster.  If $Sac, Sab,Scd,Sad,Sbc \in \CC$ (resp. $Sbd,Sab,Scd,Sad,Sbc \in \CC$) then we can {\it mutate} $\CC$ at $Sac$ (resp. $Sbd$) to produce another cluster $\CC' \subset \M$ where $Sac$ has been replaced by $Sbd$ (resp. $Sbd$ has been replaced by $Sac$).  The Pl\"ucker variables of $\CC$ and $\CC'$ are then related by \eqref{eq:threeterm}.  

\begin{proposition}[\cite{OPS}] \label{prop:mutation}
Any two clusters $\CC,\CC' \subset \M$ (as in Proposition~\ref{prop:OPS}) are related by a sequence of mutations. 
\end{proposition} 

By a positive Laurent polynomial we mean a Laurent polynomial such that the coefficient of every monomial is nonnegative.  The following result is a special case of general positivity results of cluster algebras \cite{LS}.

\begin{proposition}\label{prop:Laurent}
For $J \in \binom{[n]}{k}$, and a cluster $\CC \subset \binom{[n]}{k}$, the Pl\"ucker variable $\Delta_J$ is a positive Laurent polynomial in $\{\Delta_I \mid I \in \CC\}$. 
\end{proposition}

For an arbitrary positroid $\M$, we have the following weaker statement.
\begin{proposition}\label{prop:Plucksub}
For $J \in \M$, and a cluster $\CC \subset \M$, the Pl\"ucker variable $\Delta_J$ is a subtraction-free rational expression in $\{\Delta_I \mid I \in \CC\}$. 
\end{proposition}
\begin{proof}
This result follows, for example, from the formulae in \cite{MS}.  It also follows from the proof of Theorem~\ref{thm:param} we give below.  Namely, in that proof we show that a formula for $\Delta_J$ in terms of $\{\Delta_I \mid I \in \CC\}$ can be obtained by iteratively applying the three-term Pl\"ucker relation \eqref{eq:threeterm}.  In other words, we iteratively substitute $\Delta_{Sac} = (\Delta_{Sab}\Delta_{Scd} + \Delta_{Sad} \Delta_{Sbc})/\Delta_{Sbd}$, without ever dividing by 0.
\end{proof}

The following conjecture is likely known to many experts.  It does not immediately follow from the identification \cite{GL} of $\C[\oPi_\M]$ with a cluster algebra because there are Pl\"ucker variables $\Delta_I$, $I \in \M$ that are not cluster variables in any cluster.
\begin{conjecture}\label{conj:Laurent}
For $J \in \M$, and a cluster $\CC \subset \M$, the Pl\"ucker variable $\Delta_J$ is a positive Laurent polynomial in $\{\Delta_I \mid I \in \CC\}$. 
\end{conjecture}

\subsection{}\label{ssec:posparam}
Let $\M$ be a positroid and set $d =\dim(\M)$.  Let $(\C^\times)^d$ be a torus with coordinate functions $x_1,\ldots,x_d$, so that the coordinate ring is $\C[x_1^{\pm 1},\ldots,x_d^{\pm 1}]$.  A rational map $(\C^\times)^d \to \Pi_\M$ is called a {\it positive parametrization} of $\Pi_\M$ if it is birational, the restriction to $\R_{>0}^d$ is a homeomorphism onto $\Gr(k,n)_{>0}$, and every Pl\"ucker coordinate is a subtraction-free rational expression $\Delta_I(\x)$ in $x_1,\ldots,x_m$.  Any choice of cluster $\CC$ gives a positive parametrization after setting one of the Pl\"ucker coordinates to 1: the map $T(\CC) = (\C^\times)^d \to \Pi_\M$ is simply the inclusion of the cluster torus $\T(\CC)$ indexed by $\CC$ into the positroid variety. This map comes from the inclusion $\C[\tPi_\M] \subset \C[\Delta_J^{\pm 1} \mid J \in \CC]$, called the {\it Laurent phenomenon}.

\subsection{}\label{ssec:gaugefix}
We now consider a simple-minded notion of ``cluster" for $\Pi_\M/T$.  Let $\CC$ be a cluster for a connected positroid $\M$.  A {\it gauge-fix} $\G = \{J_1,\ldots,J_n\} \subset \CC$ is a subset such that the integral span of $e_{J_1},\ldots,e_{J_n}$ inside $\Z^n$ is the $n$-dimensional lattice $X(\hT) = \{(x_1,\ldots,x_n) \in \Z^n \mid k \text{ divides } \sum x_i \}$.  We shall prove the following result in \S\ref{ssec:proofspan}.

\begin{lemma}\label{lem:span}
Let $\M$ be a connected positroid.  Then there exists a cluster $\CC \subset \M$ such that a gauge-fix $\G \subset \CC$ exists.
\end{lemma}

If $\G$ is a gauge-fix then the action of $\hT$ can uniquely fix $\Delta_J = 1$ for all $J \in \G$ (assuming that $\Delta_J \neq 0$ for all $J \in \G$).  We thus have a canonical identification 
\begin{equation}\label{eq:posgf}
\Pi_{\M,>0}/T_{>0} = \{V \in \Pi_{\M,>0} \mid \Delta_J(V) = 1 \mbox{ for all } J \in \G\}.
\end{equation}

\section{Nonnegative configuration space}\label{sec:def}
The goal of this section is to construct a compactification of $\Conf(k,n)_{>0}$.  
\subsection{}
Let $X \subset \PP^{k-1}$ be an irreducible subvariety of dimension $r-1$.  The degree $\deg(X)$ of $X$ is equal to the number $\#(L \cap X)$ of intersection points of $X$ with a generic hyperplane $L \subset \PP^{k-1}$ of dimension $(k-r-1)$.  Let $\ZZ(X) \subset \Gr(k-r,k)$ denote the subvariety of $(k-r-1)$-dimensional projective subspaces $L \subset \PP^{k-1}$ that intersect $X$.  It is an irreducible hypersurface of degree $d$ in $\Gr(k-r,k)$ (\cite[Proposition 2.2]{GKZ}).  Let $R_d(k-r,k)$ denote the degree $d$ component of the coordinate ring of the Grassmannian.  We define the Chow form of $X$, denoted $R_X \in R_d(k-r,k)$, to be the unique up to scalar non-zero element of $R_d(k-r,k)$ that vanishes on $\ZZ(X)$.  The variety $X$ can be recovered from $R_X$.

\subsection{}
An $(r-1)$-dimensional algebraic cycle in $\PP^{k-1}$ is a formal finite linear combination $X = \sum m_i X_i$, where $m_i$ are nonnegative integers and $X_i \subset \PP^{k-1}$ are irreducible closed subvarieties of dimension $(r-1)$.  We define $\deg(X) = \sum m_i \deg(X_i)$.

Let $C(r,d,k)$ denote the set of all $(r-1)$-dimensional algebraic cycles in $\PP^{k-1}$ of degree $d$.  Then $C(r,1,k)$ can naturally be identified with the Grassmannian $\Gr(r,k)$ of $r$-planes in $\R^k$.  The set $C(r,d,k)$ acquires the structure of an algebraic variety, called the {\it Chow variety}, via the following result of Chow and van der Waerden.

\begin{theorem}\label{thm:Chow}
The map $C(r,d,k) \to \PP(R_d(k-r,k))$ given by 
$
X \mapsto \prod_i R_{X_i}^{m_i}
$
defines an embedding of $C(r,d,k)$ as closed subvariety of $\PP(R_d(k-r,k))$.
\end{theorem}

\subsection{}
A special case of Theorem \ref{thm:Chow} is the statement that $\Gr(k,n)$ is a closed subvariety of $\PP^{\binom{[n]}{k}-1}$.  Let $X = \overline{T \cdot V} \subset \Gr(k,n)$ be a torus orbit closure in $\Gr(k,n)$, where $V \in \oGr(k,n)$.  Since $T \cdot V \simeq T$, the variety $X$ has dimension $(n-1)$.  It is a toric variety with moment polytope equal to the hypersimplex.  It follows that the degree of $X$ inside $\Gr(k,n)$ and inside $\PP^{\binom{[n]}{k}-1}$ is equal to the volume $\Vol(\Delta(k,n))$, the Eulerian number $A_{n,k}$.  We thus have a natural injection $\oConf(k,n) \hookrightarrow C(n-1,A_{n,k},\binom{n}{k})$ sending a point $V \in \oConf(k,n)$ to the algebraic cycle $\overline{T \cdot V}$.  The closure of $\oConf(k,n)$ in $C(n-1,A_{n,k},\binom{n}{k})$ is called the {\it Chow quotient of the Grassmannian}, and denoted $\bConf(k,n)$.  It is a projective algebraic variety.

We define the {\it totally nonnegative part of the Chow quotient of the Grassmannian} or {\it nonnegative configuration space}, denoted $\Ch(k,n)_{\geq 0}$, to be the closure of the image of $\Conf(k,n)_{>0}$ in $\bConf(k,n)$.  It is a compact Hausdorff topological space.

\begin{remark}
For a positroid $\M$, we can also define $\bConf(\M)_{\geq 0}$ as the closure of $\Pi_{\M,>0}/T_{>0}$ inside an appropriate Chow variety.
\end{remark}

\section{Positive Chow cells}\label{sec:strata}
A point $X \in \bConf(k,n)$ is an algebraic cycle in $\Gr(k,n)$ of dimension $(n-1)$ and degree $A_{n,k}$.  We have $X = \sum_{i=1}^s m_i X_i$ where $X_i$ are toric varieties that are torus-orbit closures for the same torus $T$, and we assume $m_i$ are positive.  Let $P_i$ denote the moment polytope of $X_i$.  By \cite{Kap}, we have the following constraints on $X$:
\begin{enumerate}
\item we have $m_i = 1$ for all $i$,
\item each $P_i = P_{\M_i}$ is a matroid polytope, and
\item the polytopes $P_1,P_2,\ldots,P_s$ form a regular polyhedral subdivision of the hypersimplex $\Delta(k,n)$.
\end{enumerate}

For $X \in \Ch(k,n)_{\geq 0}$, we strengthen this result as follows.  
\begin{proposition}\label{prop:conf}
Let $X = \sum_{i=1}^s X_i \in \Ch(k,n)_{\geq 0}$.  Then $X_i$ is a toric variety with moment polytope equal to a positroid polytope $P_i = P_{\M_i}$.  The positroid polytopes $P_1,\ldots,P_s$ form a regular polyhedral subdivision of the hypersimplex $\Delta(k,n)$.  
\end{proposition}
\begin{proof}
Let $X = \lim_{t \to 0} X(t)$ where $X(t) = \overline{T \cdot V(t)} \in \Conf(k,n)_{>0}$.  For each $i = 1,2,\ldots,s$, we have $X_i = \overline{T \cdot V_i}$ for some $V_i \in \Gr(k,n)$.  A generic point $p \in T \cdot V_i$ is thus the limit of points $p(t_1),p(t_2),\ldots$ where $p(t_j) \in T \cdot V(t_j)$, where $V(t_j) \in \Gr(k,n)_{>0}$.  Let $(\Z/2\Z)^{n-1} := \{+1,-1\}^{n}/\{+1,-1\}$ denote the group of components of $T$.  The $2^{n-1}$ points $g \cdot p$ are distinct for $g \in (\Z/2\Z)^{n-1}$.  For at least one of these points $q = g \cdot p \in X_i$, the infinite sequence $g \cdot p(t_1),g \cdot p(t_2), \ldots$ contains a subsequence $q(t'_1),q(t'_2),\ldots$ which all belong to $\Gr(k,n)_{>0}$.  It follows that $q \in \Gr(k,n)_{\geq 0}$, and thus the matroid of $X_i$ is a positroid. 
\end{proof}

Let $\D(k,n)$ denote the set of regular polyhedral subdivisions of the hypersimplex into positroid polytopes.  For $X \in \Ch(k,n)_{\geq 0}$, let $\tDelta(X)$ denote the positroid decomposition from Proposition~\ref{prop:conf}.  We have a stratification
\begin{align*}
\Ch(k,n)_{\geq 0} &= \bigsqcup_{\tDelta \in \D(k,n)} \Theta_{\tDelta, >0} \\
\Theta_{\tDelta,>0} &:= \{X \in \Ch(k,n)_{\geq 0} \mid \tDelta(X) = \tDelta\}.
\end{align*}
Define $\Theta_{\tDelta,\geq 0}$ to be closure $\overline{\Theta_{\tDelta,>0} }$ in the analytic topology.
The spaces $\Theta_{\tDelta,>0}$ and $\Theta_{\tDelta,\geq 0}$ are analogues of the open and closed positroid cells $\Pi_{\M,>0}$ and $\Pi_{\M,\geq 0}$ respectively.  Define a partial order on $\D(k,n)$ by $\tDelta' \leq \tDelta$ if $\tDelta'$ is a refinement of $\tDelta$.

\begin{theorem}\label{thm:ball}
There is a stratification-preserving homeomorphism between $\Ch(k,n)_{\geq 0}$ and a polytope $P(k,n)$ of dimension $r = k(n-k)-(n-1)$.  Each stratum $\Theta_{\tDelta,>0}$ (resp. $\Theta_{\tDelta,\geq 0}$) is non-empty and homeomorphic to an open ball (resp. closed ball) of dimension $\dim(\tDelta)$, given by \eqref{eq:dimDelta}.  The closed face $\Theta_{\tDelta,\geq 0}$ is the union of relatively open faces $\Theta_{\tDelta',>0}$ as $\tDelta'$ varies over subdivisions that refine $\tDelta$.
\end{theorem}
The proof of Theorem~\ref{thm:ball} is delayed to \S\ref{ssec:homeo}.  Theorem~\ref{thm:ball} generalizes various results concerning the topology of positroid cells \cite{Pos,GKL,GKL2}.

Finally, we define $\Theta_{\tDelta}$ to be the Zariski closure of $\Theta_{\tDelta,>0} $ inside $\bConf(k,n)$, and define $\oTheta_{\tDelta}$ to be the complement of $\{\Theta_{\tDelta'} \mid \tDelta' < \tDelta\}$ in $\Theta_{\tDelta}$.  The varieties $\oTheta_{\tDelta}$ and $\Theta_{\tDelta}$ are analogues of the open and closed positroid varieties $\oPi_\M$ and $\Pi_\M$.

\section{Positive tropical Pl\"ucker vectors}\label{sec:Plucker}
\subsection{}
Let $p_\bullet = \{p_I \mid I \in \binom{[n]}{k}\}$ be a collection of ``numbers" where $p_I \in \R \cup \{\infty\}$, and not all $p_I$ are equal to infinity.  The support of $p_\bullet$ is the collection $\Supp(p) = \{I \mid p_I <\infty\} \subset \binom{[n]}{k}$.  We say that $p_\bullet$ satisfies the {\it tropical Pl\"ucker relations} \cite{Spe} if for every $S$ of size $k-2$ and $a<b<c<d$ not contained in $S$, the minimum of the three quantities
$$\{p_{Sac} + p_{Sbd}, p_{Sab} + p_{Scd} , p_{Sad} + p_{Sbc}\}$$ 
is attained twice.  We say that $p_\bullet$ is a {\it tropical Pl\"ucker vector} if it satisfies the tropical Pl\"ucker relations and, in addition, the support of $p_\bullet$ is a matroid.  We say that $p_\bullet$ satisfies the {\it positive tropical Pl\"ucker relations} if for every $S$ and $a<b<c<d$ not contained in $S$, the equation \eqref{eq:trop} holds.  A tropical Pl\"ucker vector $p_\bullet$ that satisfies the positive tropical Pl\"ucker relations is called a {\it positive tropical Pl\"ucker vector}.  A tropical Pl\"ucker vector is called {\it integral} if $p_I \in \Z\cup \{\infty\}$ and {\it rational} if $p_I \in \Q \cup \{\infty\}$.

\begin{remark}\label{rem:Lauren}
In an earlier version of this work, we erroneously asserted that any vector $p_\bullet$ satisfying the tropical Pl\"ucker relations had support $\Supp(p_\bullet)$ given by a matroid.  We thank Lauren Williams for the correction.  See \cite{OPS2} for some related discussion.
\end{remark} 

\subsection{}
 {\it Tropicalization} takes a subtraction-free rational expression to a piecewise-linear expression under the substitution
\begin{equation}\label{eq:tropsubs}
(+,\times,\div) \longmapsto (\min, +, -).
\end{equation}
For example, the rational function $\dfrac{x^3+y+1}{2xy+y^2}$ tropicalizes to the piecewise-linear function $\min(3X,Y)-\min(X+Y,2Y)$.
The equation \eqref{eq:trop} is obtained from \eqref{eq:threeterm} by applying \eqref{eq:tropsubs}, and sending the variables $\Delta_I$ to the variables $p_I$.  

\subsection{}
For a positroid $\M$, let $\Pluck(\M)_{>0}$ denote the set of positive tropical Pl\"ucker vectors with support $\M$.  Let $\Pluck(\M)_{>0}(\Z)$ (resp. $\Pluck(\M)_{>0}(\Q)$) denote those vectors that are integral (resp. rational).
We write $\Pluck(k,n)_{>0}$ when $\M$ is the uniform matroid and call $p_\bullet \in \Pluck(k,n)_{>0}$ a {\it finite} positive tropical Pl\"ucker vector.  We let $\Pluck(k,n)_{\geq 0}$ denote the set of all positive tropical Pl\"ucker vectors, the union of $\Pluck(\M)_{>0}$ over all positroids $\M$ (see Proposition \ref{prop:posSpeyer}).

In \cite{HJJS}, the set of finite tropical Pl\"ucker vectors is called the {\it Dressian}, and following their terminology, we call $\Pluck(k,n)_{>0}$ the {\it positive Dressian}. 
\subsection{}

If $p_\bullet$ satisfies is a positive tropical Pl\"ucker vector, then in addition to \eqref{eq:exchange}, $\M = \Supp(p_\bullet)$ satisfies the following positive 3-term exchange relation:
\begin{equation}\label{eq:3ex}
(Sab,Scd \in \M) \text{ or } (Sad, Sbc \in \M) \implies (Sac,Sbd \in \M)
\end{equation}
for $a<b<c<d$ not contained in $S \subset [n]$.

\begin{proposition}\label{prop:posSpeyer}
If a matroid $\M$ satisfies the positive 3-term exchange relation then it is a positroid.  Thus if $p_\bullet$ is a positive tropical Pl\"ucker vector then $\M=\Supp(p_\bullet)$ is a positroid.
\end{proposition}
\begin{proof}
We establish this result by induction on $n$.  Suppose that $\M$ is disconnected, so $\M = \M_1 \oplus \M_2$ on disjoint ground sets $S_1$ and $S_2$, such that $S_1 \cup S_2 =[n]$.  If $S_1,S_2$ are cyclic intervals, then $\M_i$ satisfies positive 3-term exchange relation within $S_i$.  Thus by induction $\M$ is the direct sum of two positroids on disjoint cyclic intervals, and is thus a positroid by Proposition~\ref{prop:disconnect}.

Otherwise, we can find direct summands $\M',\M''$ of $\M$ which are two connected matroids on subsets $A \sqcup C,B \sqcup D$ that are crossing, i.e. $A,B,C,D$ are cyclic intervals occurring in cyclic order.  Since $\M'$ is connected, there are bases $I_1,I_2$ of $\M'$ such that $|I_1 \cap A| \neq |I_2 \cap A|$.  By repeated application of the basis exchange axiom, we see that $\M'$ contains two bases $I',J'$ such that $J' = I' \cup\{c\} -\{a\}$ with $a \in A$ and $c \in C$.  Similarly,  we have bases $I'',J''$ of $\M''$ such that $J'' = I'' \cup\{d\} -\{b\}$ with $b \in A$ and $d \in C$.  We can thus find bases $Tab,Tcd$ of $\M$, while $Tac, Tbd$ are not bases, a contradiction.

We now assume that $\M$ is a connected matroid.  If it is not a positroid, then by Lemma \ref{prop:posfacet}, its matroid polytope $P_\M$ has a facet cut out by an equation $H = \{\sum_{s\in S} x_s = r\}$, where $S= S_1 \sqcup S_2$ is cyclically disconnected.  Let $\M'$ be the matroid whose matroid polytope is $P_{\M'} = H \cap P_\M$.  Now $\M' = \M_1 \oplus \M_2$ where $\M_1$ has ground set $S$ and $\M_2$ has ground set $[n] \setminus S$.  The assumption that $\dim(P_{\M'}) = \dim(P_\M) - 1$ together with $|S|, |[n]\setminus S| \geq 2$ implies that both $\M_1$ and $\M_2$ are non-trivial.  Since the ground sets $S$ and $[n] \setminus S$ are crossing, our earlier argument implies that $\M'$ cannot satisfy the positive 3-term exchange relation.  After a cyclic relabeling we may assume that $Tab,Tcd \in \M'$, while at least one of $Tac$ and $Tbd$ is not in $\M'$.  However, both $Tac,Tbd$ are in $\M$, so this is only possible if both $e_{Tac}$ and $e_{Tbd}$ do not lie on the hyperplane $H$.  Indeed, the two vertices $e_{Tac}$ and $e_{Tbd}$ lie on opposite sides of $H$, and this contradicts the assumption that $H$ is a facet.
\end{proof}

\subsection{}
Whereas the space of tropical Pl\"ucker vectors has a very complicated polyhedral structure \cite{SS,HJJS}, the situation is much simpler for positive tropical Pl\"ucker vectors.  

\begin{theorem}\label{thm:param}
Let $\M$ be a positroid and $\CC$ be a cluster for $\M$.  The maps
\begin{align}\label{eq:cluster}
\Pluck(\M)_{>0} &\longrightarrow \R^{\CC}, \qquad &p_\bullet &\longmapsto (p_I \mid I \in \CC) \\
\Pluck(\M)_{>0}(\Z) &\longrightarrow \Z^{\CC}, \qquad &p_\bullet &\longmapsto (p_I \mid I \in \CC)
\end{align}
are bijections.
\end{theorem}

\begin{proof}
Fix a cluster $\CC \subset \M$.  We show that $p_\bullet \in \Pluck(\M)_{>0}$ is determined by the values of $p_I$, $I \in \CC$.  Recall from \S\ref{ssec:weaksep} that $I,J$ are {\it weakly separated} if there does not exist cyclically ordered $a<b<c<d$ such that $a,c \in I \setminus J$ and $b,d \in J \setminus I$.  Let $(I_1,\ldots,I_n)$ denote the Grassmann necklace of $\M$.  For each $J \in \M$, define two integers $w(J)$ and $d(J)$ by
\begin{align*}
w(J)&:= \#\{a \mid (I_a, J) \text{ are not weakly separated } \} \\
d(J) &:= \begin{cases} \min_{a : \; (I_a, J) \text{ are not weakly separated}}(|I_a\setminus J|) & \mbox{if $w(J) > 0$} \\
0 & \mbox{otherwise.}
\end{cases}
\end{align*}
We show that $p_J$ is determined by $p_I$, $I \in \CC$ by induction first on $w(J)$, then on $d(J)$.  If $w(J) = 0$ then $J$ is weakly separated with $(I_1,\ldots,I_n)$ and by Proposition~\ref{prop:OPS} $J$ belongs to some cluster $\CC$ for $\M$.  By Proposition~\ref{prop:mutation}, all clusters are related by mutation, so we can express $p_J$ in terms of $p_I$, $I \in \CC$ by repeated application of the positive tropical Pl\"ucker relations.  %This formula is a linear function on $A'$.  
Thus all $J \in \M$ satisfying $w(J)=0$ is determined by $p_I$, $I \in \CC$.

Now suppose $w = w(J) > 0$.  Then we have $d = d(J) \geq 2$.  We suppose by induction that the result has been proven for all $J'$ with $w(J') < w$, or $w(J') =w$ and $d(J') < d$.  Choose $a \in [n]$ so that $(I_a,J)$ is not weakly separated and $I_a \setminus J = \{i_1,\ldots,i_d\}$ and $J \setminus I_a = \{j_1,\ldots,j_d\}$.  Since $J \geq_a I_a$, there is a unique noncrossing matching on these $2d$ points, which after reindexing we assume to be $\{(i_1,j_1),(i_2,j_2),\ldots,(i_d,j_d)\}$ where $i_r <_a j_r$ for all $r$.  Here, $<_a$ denotes the cyclic rotation of the total order where $a$ is minimal.  Since $(I_a,J)$ is not weakly separated, we can find $(i,j)$ and $(i',j')$ in this matching so that $i <_a j <_a i' <_a j'$ and there are no elements of $(I_a \setminus J) \cup (J \setminus I_a)$ in the open cyclic intervals $(i,j)$ and $(i',j')$.  Let $S = J \setminus \{j,j'\}$.  Define
$$
K_0 = Sii' \qquad K_1 = S ij \qquad K_2= Si'j' \qquad K_3 = Sij' \qquad K_4 = Sji'
$$
so that we have a positive tropical Pl\"ucker relation
$$
p_{J} + p_{K_0}= \min( p_{K_1} + p_{K_2},  p_{K_3} + p_{K_4}).
$$
We make the following claims: 
\begin{enumerate}
\item
if $(I_b,J)$ is weakly separated, then $(I_b,K_t)$ is weakly separated, for $t=0,1,2,3,4$,
\item $K_0,K_3,K_4 \in \M$, and 
\item for each $t =0,1,2,3,4$, one of the following holds: $K_t \notin \M$, or $w(J) > w(K_t)$, or ($w(J) = w(K_t)$ and $d(K_t) < d(J)$).
\end{enumerate}

\noindent {\bf Proof of (1).} First note that if $I_b \leq_b L$ and $(I_b,L)$ are weakly separated, then for some $c$ we have $I_b \setminus L \subset [b,c-1]$ and $L \setminus I_b \subset [c,b-1]$.  In particular, for $L = I_a$, we have 
\begin{equation}\label{eq:Iab}
I_b \setminus I_a \subset [b,a-1] \qquad \text{and} \qquad I_a \setminus I_b \subset [a,b-1]
\end{equation}
and for $L = J$, we have for some $c$,
\begin{equation}\label{eq:IbJ}
I_b \setminus J \subset [b,c-1] \qquad \text{and} \qquad J \setminus I_b \subset [c,b-1]
\end{equation}

We say that $(x,y) \in [n]^2$ {\it crosses} $(u,v) \in [n]^2$ if all of $x,y,u,v$ are distinct and the two line segments $\overline{xy}$ and $\overline{uv}$ cross when $1,2,\ldots,n$ are arranged in order around a circle.  Suppose that $(I_b,K_t)$ is not weakly separated.  Then we have $x,y \in I_b \setminus K_t$ and $u,v \in K_t \setminus I_b$ such that $(x,y)$ crosses $(u,v)$.

Case (a): we have $i,j,i',j' \in [b,a-1]$.  Then $i,i' \in I_b$ so neither $u$ or $v$ is equal to $i,i'$, and so the claim follows from $(I_b,J)$ being weakly separated.

Case (b): we have $i \in [a,b-1]$ and $j,i',j' \in [b,a-1]$.  We may assume that $u = i$ and since $i' \in I_b$, we have $v \in J \setminus I_b$.  Neither $x$ nor $y$ lies in $[i+1,b-1] \subset (i,j)$ since they would have to belong to $I_a$ and thus also to $J$.  Thus $v \in [b,i-1]$ and we have $(\{x,y\} \cap [v+1,i-1]) \neq \emptyset$ which contradicts \eqref{eq:IbJ}.

Case (c): we have $i,j \in [a,b-1]$ and $i',j' \in [b,a-1]$.  We may assume that $u =i$ and since $i' \in I_b$, we have $v \in J \setminus I_b$.  Neither $x$ nor $y$ lies in $(i,j)$ since they would have to belong to $I_a$ and thus also to $J$.  But $j \in J \setminus I_a$ and by \eqref{eq:Iab}, $j \notin I_b$.  If $(x,y)$ crosses $(i,v)$ then it must also cross $(j,v)$, contradicting $(I_b,J)$ being weakly separated.

Case (d): we have $i,j,i' \in [a,b-1]$ and $j' \in [b,a-1]$.  If one of $u,v$ is equal to $i$ and the other is in $J \setminus I_b$ then the argument is the same as for Case (c).  Also, if $(x,y)$ crosses $(i,i')$ then $(x,y)$ crosses $(j,i')$ as well since \eqref{eq:Iab} implies that $(\{x,y\} \cap (i,j) )= \emptyset$.  We may thus assume that $u = i'$ and $v \in J \setminus I_b$.  Similarly, $(\{x,y\} \cap [i+1,b-1]) = \emptyset$ so $v \in [b,i-1]$ and we have $(\{x,y\} \cap [v+1,i-1]) \neq \emptyset$ which contradicts \eqref{eq:IbJ}.

Case (e): we have $i,j,i',j' \in [a,b-1]$.  By \eqref{eq:Iab}, we have $(\{x,y\} \cap (i,j)) = \emptyset = (\{x,y\} \cap (i',j'))$.  At least one of $u,v$ is equal to $i,i'$.  Replacing $i$ by $j$ or $i'$ by $j'$ does not change whether $(x,y)$ crosses $(u,v)$.  Thus $(x,y)$ crosses something of the form $(j,r)$ or $(j',r)$ where $r \in J \setminus I_b$.  This contradicts $(I_b,J)$ being weakly separated. 

\noindent {\bf Proof of (2).}
We use the description of $\M$ as an intersection of Schubert matroids (Proposition~\ref{prop:Oh}).  We prove that $K_3,K_4 \in \M$ (and this immediately implies $K_0 \in \M$).  Indeed, we show that swapping any $i<_a j$ where the open interval $(i,j)$ contains nothing in $(I_a \setminus J) \cup (J \setminus I_a)$ works.  So let $J'$ be the result of such a swap.  We need to show that $J' \geq_b I_b$ for all $b$, and it suffices to show this for $b \in J'$.  Let $L = (i,j) \cap I_a = (i,j) \cap J = (i,j) \cap I_a \cap J$.

Case (a): Suppose $b \leq_a i$.  The claim follows from $J \geq_b I_b$ together with $\{i\} \cup L \subset I_b$, which holds since $\{i\} \cup L \subset I_a$, using \eqref{eq:Iab}.

Case (b): Suppose $b >_a j >_a i$.  We have that $I_b$ is the disjoint union of sets $A,B,C$ such that $A \leq_b J \cap [b,i)$ and $B\leq_b  L \cup \{j\}$ and $C \leq_b J \cap (j,b-1]$.  We claim that $B \leq_b L \cup \{i\}$.  This follows from $j \notin I_a \implies j \notin I_b \implies j \notin B$.

Case (c): Suppose $b >_a i$ but $b \leq_a j$.  The claim follows from $J' \geq_b  J \geq_b I_b$.

\noindent {\bf Proof of (3).}  Suppose that $K_t \in \M$.  Then from (1) we have $w(K_t) \leq w(J)$, and equality can only happen if $(I_a,K_t)$ is not weakly separated.  But by construction we have $|I_a \setminus K_t| > |I_a \setminus J|$, so if $w(K_t) = w(J)$ we have $d(K_t) < d(J)$.

We have proved all the claims (1),(2),(3).  By induction, all of the $p_{K_t}$ are determined by $p_I$, $I \in \CC$.
The formula $p_{J}= \min( p_{K_1} + p_{K_2},  p_{K_3} + p_{K_4}) - p_{K_0}$ shows that $p_J$ is also determined by $p_I$, $I \in \CC$.  By induction, $p_\bullet \in \Pluck(\M)_{>0}$ is determined by $p_I$, $I \in \CC$.  Furthermore, it is clear that if $p_I \in \Z$ for $I \in \CC$ then $p_\bullet \in \Pluck(\M)_{>0}(\Z)$.  The theorem is proven.
\end{proof}

Another proof of Theorem \ref{thm:param} is given after Theorem \ref{thm:bridge}.

Let $X^\vee(\hT) \subset \R^n$ denote the lattice generated by $\Z^n$ and the vector $(1/k,1/k,\ldots,1/k) \in \R^n$.  The lattice $X^\vee(\hT)$ is dual to the lattice $X(\hT) \subset \Z^n$ of \eqref{eq:character}.  Restricting the action \eqref{eq:equiv} of $\R^n$ on $\Pluck(\M)_{>0}$, we obtain an action of  $X^\vee(\hT)$ on $\Pluck(\M)_{>0}(\Z)$.  The orbits of this action are denoted $\Pluck(\M)_{>0}/\!\!\sim$ and $\Pluck(\M)_{>0}(\Z)/\!\!\sim$ respectively.  Let us now parametrize $\Pluck(\M)_{>0}/\!\!\sim$ and $\Pluck(\M)_{>0}(\Z)/\!\!\sim$.  Recall from \S\ref{ssec:gaugefix} the notion of a gauge-fix $\G\subset \CC$.

\begin{corollary}\label{cor:param}
Let $\M$ be a connected positroid, $\CC$ be a cluster for $\M$ and $\G \subset \CC$ be a gauge-fix.  The maps
\begin{align*}
\Pluck(\M)_{>0}/\!\!\sim&\longrightarrow \R^{\CC \setminus \G}, \qquad &p_\bullet &\longmapsto (p'_I \mid I \in \CC \setminus \G) \\
\Pluck(\M)_{>0}(\Z)/\!\!\sim &\longrightarrow \Z^{\CC\setminus \G}, \qquad &p_\bullet &\longmapsto (p'_I \mid I \in \CC \setminus \G)
\end{align*}
are bijections, where $p'_\bullet \sim p_\bullet$ is the unique vector in the equivalence class of $p_\bullet$ satisfying $p'_J = 0$ for $J \in \G$.
\end{corollary}
\begin{proof}
We prove the $\Z$ case.  The statement follows from Theorem \ref{thm:param} and the following claim: given $p_\bullet \in \Pluck(\M)_{>0}(\Z)$ and integers $(c_J \mid J \in \G) \in \Z^{\G}$, there is a unique $p'_\bullet \sim p_\bullet$ such that $p'_J = c_J$ for all $J \in \G$.  This claim follows from the following statement: for any $(c_J \mid J \in \G) \in \Z^{\G}$, there exists $\a \in X^\vee(\hT)$ such that $\a \cdot e_J = c_J$ for all $J \in \G$, which in turn follows from the definition of gauge-fix in \S\ref{ssec:gaugefix}.
\end{proof}

\section{Subdivisions of the hypersimplex}\label{sec:hypersimplex}
\subsection{}
\def\tP{{\tilde P}}
Let $P$ be a polytope.  A {\it subdivision} of $P$ is a collection $\tP = \{Q\}$ of polytopes $Q$ (called {\it faces} of $\tP$) such that 
\begin{enumerate}
\item each face of $Q \in \tP$ is in $\tP$,
\item the intersection $Q \cap Q'$ for $Q,Q' \in \tP$ is a face of both $Q$ and $Q'$,
\item the union of all $Q \in \tP$ is equal to $P$.
\end{enumerate}
We typically give a subdivision by only listing the polytopes $Q$ of maximal dimension.

\subsection{}
Given any $p_\bullet \in \R^{\binom{n}{k}}$ we obtain a subdivision of the hypersimplex as follows.  We lift each vertex $e_I \in \Delta(k,n)$ to the point $e'_I = (e_I, p_I)$ in one higher dimension.  Then we project the lower faces of the convex hull $\Conv(e'_I)$ back into $\Delta(k,n)$.  These faces will give us a polyhedral subdivision of $\Delta(k,n)$ denoted $\tDelta(p_\bullet)$, and these subdivisions are called {\it regular}.  More generally, for $p_\bullet \in (\R \cup \{\infty\})^{\binom{n}{k}}$ (not all equal to $\infty$), we obtain a regular subdivision $\tDelta(p_\bullet)$ of the polytope $P_{\Supp(p_\bullet)} := \Conv(e_I \mid I \in \Supp(p_\bullet))$.  By lifting some vertices $e_I$ to $\infty$, the resulting convex hull $\Conf(e'_I)$ becomes a polyhedron with many ``vertical" faces.  The projection of the lower faces of $\Conv(e'_I)$ will only cover $P_{\Supp(p_\bullet)}$.

Let us describe the faces of $\tDelta(p_\bullet)$ more explicitly.  Faces $F$ of $\tDelta(p_\bullet)$ are convex polytopes whose vertices are a subset of the vectors $\{e_I \mid I \in \binom{[n]}{k}\}$.  Abusing notation, we will also consider $F$ as a subset of $\binom{[n]}{k}$.  We call $p_\bullet$ and $p'_\bullet$ {\it equivalent}, and write $p_\bullet \sim p'_\bullet$ if there exists a vector $\a = (a_1,\ldots,a_n) \in \R^n$ such that for all $I \in \binom{[n]}{k}$,
\begin{equation}\label{eq:equiv}
p'_I  = p_I + \sum_{i \in I} a_i.
\end{equation}
We write $p'_\bullet = \a \cdot p_\bullet$.
If $p_\bullet \sim p'_\bullet$, then we have $\tDelta(p_\bullet) = \tDelta(p'_\bullet)$.  The faces of $\tDelta(p_\bullet)$ are the bottom faces
\begin{equation}\label{eq:tDelta}
F = F(p'_\bullet) := \{I \mid p'_I = \min(p'_\bullet)\} \subseteq \sbinom
\end{equation}
for $p'_\bullet \sim p_\bullet$.  Here, $\min(p'_\bullet) \in \R$ is the minimum value in the vector $p'_\bullet$.  By \eqref{eq:equiv}, we can always assume that $p'_\bullet \geq 0$ and $\min(p'_\bullet) = 0$.  Note that if $p'_\bullet$ is chosen generically, we expect $F(p'_\bullet)$ to be a single vertex of the hypersimplex.

\begin{lemma}\label{lem:facezero}
Suppose that $F$ is a full-dimensional face of $\tDelta(p_\bullet)$.  Then there is a unique $p'_\bullet \sim p_\bullet$ such that $p'_I =0$ for $I \in F$.  This $p'_\bullet$ satisfies $p'_J >0$ for $J \notin F$.
\end{lemma}
\begin{proof}
If $F$ is full-dimensional, then the vectors $\{e_I \mid I \in F\}$ span $\R^n$.  Thus at most one $p'_\bullet$ equivalent to $p_\bullet$ satisfies the condition $p'_I =0$ for $I \in F$.   The existence and the last conclusion follows from the assumption that $F$ is a face of $\tDelta(p_\bullet)$.
\end{proof}

The following result is immediate.
\begin{lemma}
The condition that $p_\bullet$ is a tropical Pl\"ucker vector (resp. positive tropical Pl\"ucker vector) is a property of the equivalence class of $p_\bullet$.
\end{lemma}

\subsection{}

Suppose that $p_\bullet$ is a tropical Pl\"ucker vector.  Then $\tDelta(p_\bullet)$ is a regular subdivision of the matroid polytope $P_{\Supp(p_\bullet)}$.  Suppose that $p_\bullet$ is a positive tropical Pl\"ucker vector.  Then by Proposition \ref{prop:posSpeyer}, $\tDelta(p_\bullet)$ is a regular subdivision of the positroid polytope $P_{\Supp(p_\bullet)}$.  Part (1) of the following result is proved in \cite{Spe}; see also \cite[Corollary 13]{OPS2}.
\begin{proposition}\label{prop:posdiv}
Suppose that $p_\bullet$ has support equal to a matroid $\M$.
\begin{enumerate}
\item
$p_\bullet$ satisfies the tropical Pl\"ucker relations if and only if each face of the regular subdivision $\tDelta(p_\bullet)$ is a matroid polytope.   
\item
$p_\bullet$ satisfies the positive tropical Pl\"ucker relations if and only if each face of the regular subdivision $\tDelta(p_\bullet)$ is a positroid polytope.   
\end{enumerate}
\end{proposition}
\begin{proof}
We prove (2).

Suppose that $\tDelta(p_\bullet)$ is a positroid subdivision.  Let us fix $S \subset [n]$ of size $k-2$ and $a<b<c<d$ not contained in $S$.  The intersection of $\Delta(k,n)$ with the hyperplanes $\{x_i = 0 \mid i \notin Sabcd\}$ and $\{x_i = 1 \mid i \in S\}$ is a face $F = F(S;abcd)$ of $\Delta(k,n)$.  This face $F$ is an octahedron with six vertices $e_{Sab}, e_{Sac}, e_{Sad}, e_{Sbc}, e_{Sbd}, e_{Scd}$.  The intersection of the positroid subdivision $\tDelta(p_\bullet)$ with $F$ gives a subdivision of $\tF$ (or a subpolytope of $\tF$), which must be a positroid subdivision.  Thus it suffices to observe that the positive tropical Pl\"ucker relation \eqref{eq:trop} holds for the six ``numbers" $p_{Sab}, p_{Sac}, p_{Sad}, p_{Sbc}, p_{Sbd}, p_{Scd}$ if they induce a positroid subdivision of a subpositroid of $\tF$.  This is a straightforward case-by-case analysis.

Suppose that $p_\bullet$ satisfies the positive tropical Pl\"ucker relation.  Let $F$ be a face of the subdivision $\tDelta(p_\bullet)$.  Then we can find $p'_\bullet \sim p_\bullet$ satisfying $p'_\bullet \geq 0$ and $\min(p'_\bullet) = 0$ so that $F = \{I \mid p'_I = 0\}$.  Let $q_\bullet$ be defined by
$$
q_I = \begin{cases} 0 & \mbox{if $p'_I = 0$} \\
\infty & \mbox{if $p'_I > 0$.}
\end{cases}
$$
By (1), we know that $F = \Supp(q_\bullet)$ is a matroid.  By Proposition \ref{prop:posSpeyer}, to show that $F$ is a positroid, it suffices to show that $q_\bullet$ is a positive tropical Pl\"ucker vector.  Consider any $S$ and $a<b<c<d$ as in \eqref{eq:trop}.  If both sides of \eqref{eq:trop} are equal to 0 (resp. positive) for $p'_\bullet$, then both sides of \eqref{eq:trop} are equal to 0 (resp. equal to $\infty$) for $q_\bullet$.  Thus \eqref{eq:trop} holds for $q_\bullet$ if it holds for $p'_\bullet$.
\end{proof}

\section{Positive tropical Grassmannian}\label{sec:trop}
\subsection{}\label{sec:real}
 Let $\RR = \bigcup_{n=1}^\infty \R((t^{1/n})) $ denote the field of Puiseux series over $\R$.  We define $\val: \RR \to \R \cup \{\infty\}$ by  $\val(0) = \infty$ and $\val(x(t)) = r$ if the lowest term of $x(t)$ is equal to $\alpha t^r$ where $\alpha \in \R^\times$.  We define $\RR_{>0} \subset \RR$ to be the semifield consisting of Puiseux series $x(t)$ that are non-zero and such that coefficient of the lowest term is a positive real number.  We let $\RR_{\geq 0} :=\RR_{>0} \cup \{0\}$.

A point $V(t)$ in the Grassmannian $\Gr(k,n)(\RR)$ is, as usual, determined by its Pl\"ucker coordinates $\Delta_I(V) \in \RR$, which satisfy the Pl\"ucker relations and are defined up to a common scalar.
We let $ \Gr(k,n)(\RR_{\geq 0}) \subset \Gr(k,n)(\RR)$ (resp. $ \Gr(k,n)(\RR_{> 0}) \subset \Gr(k,n)(\RR)$) denote the subset of points $V(t)$ whose Pl\"ucker coordinates lie in $\RR_{\geq 0}$ (resp. $\RR_{>0}$).  Similarly, we define $\hGr(k,n)(\RR)$, $\hGr(k,n)(\RR_{\geq 0})$, and $\hGr(k,n)(\RR_{>0})$.  

\begin{lemma}\label{lem:valtrop}
For $V(t) \in \hGr(k,n)(\RR_{\geq 0})$, the vector $p_I = \val(\Delta_I(V(t)))$ is a positive tropical Pl\"ucker vector.
\end{lemma}
\begin{proof}
The support of $p_\bullet$ is a matroid $\M$.  The three-term positive tropical Pl\"ucker relations follows from tropicalizing \eqref{eq:threeterm}.  Thus $p_\bullet$ is a positive tropical Pl\"ucker vector.
\end{proof}

For a positroid $\M$, we also define 
$$
\Pi_\M(\RR_{>0}):= \{V(t) \in \Gr(k,n)(\RR_{\geq 0}) \mid \Delta_I(V(t)) \in \RR_{>0} \text{ if and only if } I \in \M\}.
$$
%\begin{proposition}
%Let $V  \in \Gr(k,n)(\R((t))^+)$.  Define $p_\bullet = \val(V)$ by $p_I = \val(\Delta_I(V))$.  Then $p_\bullet$ is a positive tropical Pl\"ucker vector.  Furthermore, $V \in \Pi_\M$ for a positroid $\M$, and $\Supp(p_\bullet) = \M$.
%\end{proposition}
By Proposition~\ref{prop:posSpeyer} and Lemma~\ref{lem:valtrop}, we have the disjoint union
\begin{equation}\label{eq:RRdecomp}
\Gr(k,n)(\RR_{\geq 0}) = \bigsqcup_\M \Pi_\M(\RR_{>0})
\end{equation}
as $\M$ varies over positroids of rank $k$ on $[n]$, an analogue of \eqref{eq:positroiddecomp}.

\subsection{}
We now connect $\Gr(k,n)(\RR_{>0})$ to the nonnegative part $\Ch(k,n)_{\geq 0}$ of the Chow quotient.  Let $\gamma(t)$ be a curve in $\Gr(k,n)$ that is analytic near $t = 0$.  Thus each Pl\"ucker coordinate $\Delta_I(\gamma(t))$ has a Taylor expansion $g_I(t) \in \R[[t]]$ that converges near $t = 0$.  Suppose that for some $s > 0$, we have that $\gamma((0,s)) \in \Gr(k,n)_{>0}$.  Then for every $I$, we have that $g_I(t) \neq 0$ and the coefficient of the lowest term of $g_I(t)$ must be positive.  In particular, $\gamma(t)$ ``agrees" with a point $V(t) \in \Gr(k,n)(\RR_{>0})$.

Furthermore, $\gamma((0,s))$ determines a curve in $\Conf(k,n)_{>0} =\Gr(k,n)_{>0}/T_{>0}$, and thus $\lim_{t \to 0} \gamma(t)$ is a point in $X = \sum_i X_i \in \Ch(k,n)_{\geq 0}$.  The hypersimplex decomposition $\tDelta(X)$ is given by $\tDelta(X) = \tDelta(\val(V(t)))$.  Let $d(\a) = {\rm diag}(t^{a_1},\ldots,t^{a_n}) \in \GL(n)(\RR)$ be an $\RR$-valued point in the torus acting on $\Gr(k,n)$.  Then $p_\bullet = \val(V)$ and $p'_\bullet = \val(V \cdot d(\a))$ are related by \eqref{eq:equiv}.  Let $V'(t) = V(t) \cdot d(\a)$, and suppose that $p'_\bullet = \val(V'(t)) \geq 0$ and at least one $p_I$ is equal to 0.  Then $V_0 = \lim_{t \to 0} V'(t)$ lies in $\Pi_{\M,>0}$ for some positroid $\M$.  The point $V_0$ belongs to (at least) one of the toric varieties $X_i = \overline{T \cdot V_i}$, and the positroid polytope $P_\M$ is one of the faces of $\tDelta(X)$.  To obtain the maximal faces, one has to pick the vector $\a$ carefully, just as \eqref{eq:tDelta} typically gives lower-dimensional faces.

We refer the reader to \cite{KT} for further details from this perspective.
%
%
%Let $\tDelta \in \D(k,n)$ be a regular subdivision of the hypersimplex by positroid polytopes.  By Proposition \ref{prop:posdiv}(2), we have $\tDelta = \tDelta(p_\bullet)$ for some $p_\bullet \in \Pluck(k,n)_{>0}$, and we may assume that we have chosen $p_\bullet \in \Pluck(k,n)_{>0}(\Z)$.
%By Theorem \ref{thm:main}, $p_\bullet = \val(\Delta_I(V))$ for $V \in \Gr(k,n)(\RR_{>0})$.  The point $V$ is gives a ``Puiseux" curve $\Spec(\RR) \to \Gr(k,n)$ in the real Grassmannian.  In fact, by Theorem~\ref{thm:bridge}, we see that $V$ can be chosen so that $\Delta_I(V)$ are polynomials i.e. $V \in \Gr(k,n)(\R[t])$.  So we have an analytic curve $\gamma: [0,c) \to \Gr(k,n)$ and for $\epsilon$ sufficiently small we must have $\gamma(\epsilon) \in \Gr(k,n)_{>0}$ since the lowest coefficient of $\Delta_I(V)$ is positive for all $I$.  It follows that $\lim_{t \to 0} \overline{T \cdot \gamma(t)}$ defines a point $X \in \Ch(k,n)_{\geq 0}$.  As discussed in \S\ref{sec:strata}, $X$ is the union of toric varieties corresponding to the subdivision $\tDelta$.

\subsection{}
Following \cite{SW}, we define the {\it positive tropical Grassmannian}
$$
\Trop_{> 0} \Gr(k,n) :=\overline{\val(\hGr(k,n)(\RR_{> 0}))} \subset \Pluck(k,n)_{>0}
$$
as the closure of $\val$ on $\hGr(k,n)(\RR_{> 0})$, and the {\it nonnegative tropical Grassmannian}
$$
\Trop_{\geq 0} \Gr(k,n) := \overline{\val ( \hGr(k,n)(\RR_{\geq 0})  \setminus \{0\})} \subset \Pluck(k,n)_{\geq 0}
$$
where $\{0\}$ denotes the cone point of $\hGr(k,n)(\RR)$ (by convention, this cone point does not lie in $\hGr(k,n)(\RR_{>0})$).  For a positroid $\M$, we also define $\Trop_{>0} \Pi_{\M} :=  \overline{\val (\hPi_\M(\RR_{>0}))}$.  By \eqref{eq:RRdecomp}, we have 
\begin{equation}\label{eq:Tropecomp}
\Trop_{\geq 0} \Gr(k,n) = \bigsqcup_\M \Trop_{>0} \Pi_{\M} .
\end{equation}

%Let $\R((t))^+$ be the set of formal Laurent series with real coefficients which are either 0, or whose lowest coefficient is positive.  This is a semifield: it is closed under addition, multiplication, and division.  We have a valuation map $\val: \R((t))^+ \to \Z \cup \{0\}$, where $\val(0) = \infty$ and $\val(\sum_{i=m}^\infty a_i t^i) = m$ if $a_m >0$. 

\subsection{}
We call a positive tropical Pl\"ucker vector $p_\bullet$ {\it realizable} if it arises as $\val(\Delta_I(V))$ for some $V \in  \Gr(k,n)(\RR_{\geq 0})$.  It is not immediately clear that every $p_\bullet \in \Pluck(k,n)_{\geq 0}$ is realizable because a priori $\val(\Delta_I(V))$ satisfies relations beyond \eqref{eq:trop}: for example such a relation exists for any (not necessarily three-term) Pl\"ucker relation for $\Gr(k,n)$.

\begin{theorem}\label{thm:main}
Every rational positive tropical Pl\"ucker vector is realizable.  
Thus 
$$
\Trop_{\geq 0}\Gr(k,n) = \Pluck(k,n)_{\geq 0}, \;\;\;  \Trop_{> 0}\Gr(k,n) = \Pluck(k,n)_{>0}, \;\;\;\Trop_{>0} \Pi_{\M} = \Pluck(\M)_{>0}.
$$
%If $\Supp(p_\bullet) = \M$, then $p_\bullet = \val(\Delta_I(V))$ for $V \in \Pi_\M(\R((t))^+)$.
%A positive tropical Pl\"ucker vector with support $\M$ is realizable by a point in the totally positive positroid cell $\Pi_\M$ with entries in $\R((t))^+$.
\end{theorem}
\begin{proof}
Let $\M$ be a positroid, $\CC \subset \M$ be a cluster, and $d = \dim(\M)$.  By \S\ref{ssec:posparam}, we have a positive parametrization $T(\CC) = (\C^\times)^d \hookrightarrow \Pi_\M$.  This map induces a map $\iota_\CC: \RR_{>0}^d \to \Pi_\M(\RR)$ where if $f=(f_I(t) \mid I \in \CC) \in \RR_{>0}^d$, we have $\Delta_I(\iota_\CC((f)) = f_I(t)$.  By Proposition~\ref{prop:Plucksub}, the image of $\iota_\CC$ lies in $\Pi_\M(\RR_{>0})$.  Let $\nu_\CC: \Pluck(\M)_{>0} \to \R^\CC$ denote the projection $p_\bullet \mapsto (p_I \mid I \in \CC)$ of Theorem~\ref{thm:param}.
Then the composition $\nu_\CC \circ \val \circ \iota_\CC$ sends $\RR_{>0}^d$ surjectively onto $\Q^d$.  By Theorem~\ref{thm:param} we have $\nu_\CC: \Pluck(M)_{>0}(\Q) \to \Q^d$ is an isomorphism, we see that every rational positive tropical Pl\"ucker vector is realizable.  Taking the closure in $\R^d$, we obtain $\Trop_{>0} \Pi_{\M} = \Pluck(\M)_{>0}$.
\end{proof}

\subsection{}
We give another proof of Theorem \ref{thm:main} using bridge decompositions, which we believe is of independent interest.  Recall from \eqref{eq:bridgeparam} the notion of bridge parametrizations of $\Pi_{\M,>0}$.  In \S\ref{sec:bridge}, we define tropical bridges $T_i(a)$, $a\in \Z$ (and more generally $T_\gamma(t)$ for $\gamma =(i,j)$) acting on the space of positive tropical Pl\"ucker vectors, with the following property: for $a(t) \in \RR_{>0}$ and $V(t) \in \Gr(k,n)(\RR_{\geq 0})$
$$
\val( V(t) \cdot x_i(a(t))) = T_i(\val(a(t))) \cdot \val(V(t)).
$$
In particular, if $p_\bullet$ is representable, it follows immediately that $T_i(a) \cdot p_\bullet$ is representable for any $a \in \Q$.

\begin{theorem}\label{thm:bridge}
For any bridge parametrization \eqref{eq:bridgeparam} of $ \Pi_\M(\R_{>0})$, we have a tropical bridge parametrization of $\Pluck(\M)_{>0}$, given by
%\begin{align*}
$$
%\R^d &\simeq \Pluck(\M)_{>0}, \qquad &(z_1,z_2,\ldots,z_d) &\longmapsto \Trop(x_{\gamma_1}(t_1)) \cdots \Trop(x_{\gamma_d}(t_d) )\cdot p(I)_\bullet \\
\R^{d+1} \simeq \Pluck(\M)_{>0}, \qquad (z_0,z_1,z_2,\ldots,z_d) \longmapsto 
%\Trop(x_\gamma(\z))\cdot p(I,z_0)_\bullet :=
T_\gamma(\z) \cdot p(I,z_0)_\bullet := T_{\gamma_1}(z_1) \cdots T_{\gamma_d}(z_d) \cdot p(I,z_0)_\bullet
$$
%\end{align*}
where $d = \dim(\M)$ and $p(I)_\bullet \in \Pluck(\{I\})_{>0}$ is defined by $p(I,z_0)_I = z_0$ and $p(I,z_0)_J =\infty$ for $J \neq I$.  The tropical bridge parametrization maps $\Z^{d+1}$ (resp. $\Q^{d+1}$) isomorphically onto $\Pluck(\M)_{>0}(\Z)$ (resp. $\Pluck(\M)_{>0}(\Q)$).
\end{theorem}
The proof of Theorem \ref{thm:bridge} is given in \S\ref{sec:bridge}.  Theorem \ref{thm:main} follows immediately from Theorem \ref{thm:bridge}.   We give another proof of Theorem~\ref{thm:param}.
%The tropical totally positive Grassmannian is defined in \cite{SW}.  The above result says that positive tropical Pl\"ucker vectors with support equal to the uniform matroid is identical to the tropical totally positive Grassmannian.

\begin{proof}[Second proof of Theorem \ref{thm:param}]
%Since all clusters of $\M$ are related by mutation, it suffices to prove Theorem \ref{thm:param} for one choice of cluster $C$.  
Any two bridge parametrizations of $\Pi_{\M,>0}$ are related by an invertible rational transformation that is subtraction-free in both directions, see \cite{Pos}.  %(Note that this transformation is not necessarily given by positive Laurent polynomials, so they are not equivalent in the sense of \S\ref{ssec:posparam}.)
For particular choices of bridge parametrizations and cluster parametrizations, \cite{MS} gives explicit invertible subtraction-free rational expressions between them; for one direction we have \cite[Theorem 3.3]{MS} and for the other combine \cite[Theorem 7.1]{MS} with \cite[Proposition 7.10]{MS}.
Tropicalizing these subtraction-free rational transformations (using \eqref{eq:tropsubs}), we conclude that $\z \in \R^d$ and $(p_I \mid I \in \CC) \in \R^\CC$ are related by an invertible piece-wise linear transformation, and this holds even integrally.  Thus the bijections of \eqref{eq:cluster} follow from Theorem \ref{thm:bridge}.
%
%\fixit{Figure out the precise condition we want}, a bridge decomposition $(\gamma_1,\ldots,\gamma_d)$ and a cluster $C$ for $\M$ can be chosen so that the Pl\"ucker variables $\{\Delta_I \mid I \in C\}$ and the bridge parameters $t_1,\ldots,t_d$ are related by a monomial transformation.  It follows that $\z \in \R^d$ and $(p_I \mid I \in C) \in \R^C$ are related by a {\it linear} transformation.  So the fan 
\end{proof}

\section{Fan structure}\label{sec:fan}

\subsection{}
A (polyhedral) {\it fan} $\F = \{C\}$ in a vector space $\R^d$ is a finite collection of closed polyhedral cones $C \subset \R^d$ satisfying the conditions: 
\begin{enumerate}
\item
If $C, C' \in \F$ then $C \cap C' \in \F$.
\item
For $C \in \F$, every face $C' \subset C$ is in $\F$.
\end{enumerate} 
We say that $\F$ is {\it complete} if the union of the cones in $\F$ is equal to $\R^d$.  In the case that $\F$ is complete, or if $\F$ is pure of some dimension $d'$ (i.e., all maximal cones have dimension $d'$), we call the maximal cones {\it chambers}.

%Let $\Pluck_+(k,n) \subset \R^{\binom{n}{k}}$ denote the space of positive tropical Pl\"ucker vectors: the subspace of (finite) vectors satisfying \eqref{eq:trop}.  For each positroid $\M$, we let $\Pluck(\M)_{>0}(k,n)$ denote the subspace of vectors satisfying \eqref{eq:trop} with support equal to $\M$. We would like to compare three different polyhedral fan structures on $\Pluck(k,n)_{>0}$ and $\Pluck(\M)_{>0}(k,n)$.
%\subsection{Three fan structures}
\subsection{}
The spaces $\Pluck(\M)_{>0}$ and $\Pluck(k,n)_{>0}$ have a number of different {\it fan structures} that we now define.  All of these fan structures are compatible with equivalence: if $p_\bullet\sim p'_\bullet$ then they belong to the same cone.  Thus we can also think of these as complete fan structures on the vector spaces $\Pluck(\M)_{>0}/\!\!\sim$ and $\Pluck(k,n)_{>0}/\!\!\sim$.  

If we consider $\Pluck(\M)_{>0}$ as a subset of $\R^{\binom{[n]}{k}}$, then these fan structures are collections of cones of dimension less than or equal to $\dim(\M)+1 = |\CC|$.  If we consider $\Pluck(\M)_{>0} \simeq \R^{\CC}$ as a vector space (using Theorem \ref{thm:param}), then we have complete fan structures on $\R^{\CC}$.

\subsection{}
We begin by defining the {\it secondary fan structure}.  Two vectors $p_\bullet, q_\bullet \in \Pluck(\M)_{>0}$ belong to the same (relatively open) cone of the secondary fan structure if and only if they induce the same positroid polytope subdivision $\tDelta$ of $P_\M$ i.e. if $\tDelta(p_\bullet) = \tDelta(q_\bullet)$.  For a regular subdivision $\tDelta$ of $P_\M$ into positroid polytopes, we denote by $C(\tDelta) \subset \Pluck(\M)_{>0}$ the corresponding (closed) cone.  The faces of a $C(\tDelta) \subset \Pluck(\M)_{>0}$ are the cones $C(\tDelta')$ as $\tDelta'$ varies over all subdivisions that refine $\tDelta$.  

We define an integer $\dim(\tDelta)$ for a regular polyhedral subdivision $\tDelta$ of $P_\M$ into positroid polytopes: 
\begin{equation}
\label{eq:dimDelta}
\dim(\tDelta):= \dim(\M)+1 - \dim(C(\tDelta)).
\end{equation}
Thus if $\tDelta_0$ is the trivial subdivision of $\Delta(k,n)$, then $\dim(\tDelta_0) = k(n-k)-(n-1)$.

\subsection{}
For every $S \subset [n]$ and $a<b<c<d$ not contained in $S$, we have the three term Pl\"ucker fan whose two chambers are given by
\begin{align*}
p_{Sac} + p_{Sbd} &= p_{Sab} + p_{Scd} \leq  p_{Sad} + p_{Sbc}\\
p_{Sac} + p_{Sbd} &=  p_{Sad} + p_{Sbc} \leq  p_{Sab} + p_{Scd} 
\end{align*}
This fan has exactly three cones: the two chambers above and a third cone where $p_{Sac} + p_{Sbd} =  p_{Sad} + p_{Sbc} =  p_{Sab} + p_{Scd}$, the intersection of the two chambers.  The {\it Pl\"ucker fan structure} on $\Pluck(\M)_{>0}$ is the common refinement of all the three term Pl\"ucker fans.  In other words, two vectors $p_\bullet$ and $q_\bullet$ belong to the same relatively open cone of the Pl\"ucker fan structure if exactly which of the three quantities $p_{Sac} + p_{Sbd}, p_{Sad} + p_{Sbc}, p_{Sab} + p_{Scd}$ is minimal is the same for $p_\bullet$ and for $q_\bullet$.  

\subsection{}
The {\it positive fan structure} is the fan whose cones are the images of the domains of linearity for a positive parametrization by a cluster.  To be precise, pick a cluster $\CC$ for $\M$.  By Theorem \ref{thm:param} we identify $p_\bullet \in \Pluck(\M)_{>0}$ with a point  in $\R^{\CC}$.  For each $J \in \M$, the function $p_J$ is a piecewise-linear function on $\R^{\CC}$.  The common domains of linearity (see \S\ref{ssec:domain} for further discussion) for all the functions $p_J$, $J \in \M$ is the positive fan structure of $\Pluck(\M)_{>0}$ with respect to the cluster $\CC$.

\begin{proposition}
The positive fan structures on $\Pluck(\M)_{>0}$ induced by two different clusters $\CC$ and $\CC'$ are naturally identified via the isomorphisms of Theorem \ref{thm:param}.
\end{proposition}
\begin{proof}
In each chamber $A$ of the positive fan structure on $\R^{\CC}$, all the functions $p_J$, $J\in \M$ are linear.  The composition and inverse of linear functions is linear, so under the piecewise-linear isomorphism $\R^{\CC} \to \R^{\CC'}$ the cone $A$ will be sent inside some chamber $A'$ of the positive fan structure on $\R^{\CC'}$, and reversing the roles of $\CC$ and $\CC'$ we see that $A$ and $A'$ are isomorphic.
\end{proof}
We may thus speak of {\it the} positive fan structure on $\Pluck(\M)_{>0}$.  

\subsection{}
We shall show that the secondary fan structure, the Pl\"ucker fan structure, and the positive fan structure coincide.
\begin{proposition}\label{prop:pure}
With any of the three fan structures, $\Pluck(\M)_{>0}$ is a polyhedral fan pure of dimension $\dim(\M)+1$.
\end{proposition}
\begin{proof}
Any complete finite fan structure on $\R^{\CC}$ is pure of dimension $|\CC|$, i.e., $\R^{\CC}$ is the union of the closed cones of dimension exactly $|\CC|$.
By Theorem \ref{thm:param}, the set $\Pluck(\M)_{>0}$ is the image of $\R^{\CC}$ under a continuous piecewise-linear map.  Thus $\Pluck(\M)_{>0}$ is a polyhedral fan pure of dimension $\dim(\M)+1$.
%
%Thus every cone is in the closure of a cell of dimension $\R^{\dim(\M)}$ and this statement holds for any finite fan structure we put on $\Pluck(\M)_{>0}$. %\fixit{More...}
\end{proof}

The secondary fan structure and the Pl\"ucker fan structure are shown to coincide in \cite{OPS2}. 
\begin{theorem}\label{thm:coincide}
The three fan structures on $\Pluck(\M)_{>0}$ coincide.
\end{theorem}
\begin{proof}
We show that the Pl\"ucker fan and the positive fan agree.  By Proposition \ref{prop:pure}, it suffices to compare cones of maximal dimension $\dim(\Pi_\M)$.  Let $A \subset \R^{\CC}$ be such a cone in the positive fan structure for a cluster $\CC$ for $\M$.  Then all the functions $p_J$ are linear functions on $A$: we have $p_J = \sum_I \alpha^A_{J,I} p_I$, where $\alpha^A_{J,I} \in \Z$.  Substituting these expressions into $(p_{Sac} + p_{Sbd}) - (p_{Sab} + p_{Scd})$ or $(p_{Sac} + p_{Sbd}) - (p_{Sad} + p_{Sbc})$ we obtain two linear functions.  For each point of $A$, at least one of these linear functions vanishes.  Since $A$ has maximal dimension, one of the two linear functions is identically 0 on $A$; it follows that $A$ is completely contained in some chamber of the Pl\"ucker fan.  We have shown that the positive fan structure is a refinement of the Pl\"ucker fan structure.

Next let $A'$ be a chamber of the Pl\"ucker fan structure.  The proof of Theorem~\ref{thm:param} shows that each $p_J$ can be written as a piecewise-linear function of $p_I$, $I \in \CC$, by iteratively applying the equality \eqref{eq:trop}.  Within the chamber $A'$, the RHS of \eqref{eq:trop} is a linear function instead of a piecewise-linear function.  It follows that the restriction of $p_J$ to $A'$ is a linear function.  We have shown that the Pl\"ucker fan structure is a refinement of the positive fan structure.
\end{proof}

%\begin{remark}
%Another fan structure on $\Pluck(\M)_{>0}$ can be obtained from a bridge decomposition.  Fix a bridge parametrization of $\Pi_\M$ as in \S\ref{sec:real}.  We define a fan on $\R^d$ as the common domains of linearity of all the functions $p_J$, $J \in \M$, where $p_J = p_J(\Trop(x_\gamma(\z))\cdot p(I,z_0)_\bullet)$.  It is clear that this bridge fan structure is a refinement of the positive fan structure.  It follows from the proof of Theorem~\ref{thm:bridge} that this bridge fan structure agrees with the positive fan structure, see \fixit{reference}.  (For certain bridge decompositions, there is a monomial transformation between the bridge parametrization and the Pl\"ucker coordinates for some cluster of $\M$.  In this case, the equality of the two fan structures is immediate.)
%\end{remark}

%By \cite{MS}, the bridge parameters and the Pl\"ucker coordinates in a cluster are related by an invertible monomial transformation, after composing with an appropriate {\it twist map}.  

\begin{remark}
In \cite{SW} a particular fan structure on $\Trop_{>0} \Gr(k,n)$ is studied.  The positive parametrization in \cite{SW} is related by an invertible monomial transformation to a gauge-fix of the cluster parametrization for the cluster
$$
\CC = \{ \{1,2\ldots,i-1\} \cup \{i+j-k,\ldots,j\} \mid (i,j) \in [1,k] \times [k+1,n]\}.
$$
The tropicalization of a monomial transformation is a linear map, so the fan structure of \cite{SW} agrees with the one in Theorem~\ref{thm:coincide}.
\end{remark}

\begin{remark}
The twist map $\eta_\M$ \cite{MS} is a subtraction-free birational map that induces a birational isomorphism $\eta_\M: \Pi_\M \simeq \Pi_\M$. It also induces a piecewise-linear isomorphism $\Trop(\eta_\M) : \Pluck(\M)_{>0} \simeq\Pluck(\M)_{>0}$.  It is likely that this map preserves the fan structure i.e. $\Trop(\eta_\M)$ restricts to a linear map on each cone of the fan $\Pluck(\M)_{>0}$.
\end{remark}

\subsection{}
Theorem~\ref{thm:coincide} allows us to parametrize $\D(k,n)$ with collections of planar trees, see also \cite[Section 4]{HJJS} and \cite{BC,CGUZ}.    For a subset $S \subset [n]$, let $H_{S=1}$ denote the subspace given by intersecting $\{x_s = 1\}$ for $s \in S$.  Note that the face $H_{S=1} \cap \Delta(k,n)$ is isomorphic to $\Delta(k-|S|,n-|S|)$.

Let $\tDelta \in \D(k,n)$.  Intersecting $\tDelta$ with the face $H_{S=1}$ for $|S|=k-2$ gives a subdivision of $\Delta(2,[n] \setminus S)$, i.e. an element $\tDelta_S \in \D(2,[n] \setminus S)$.  It is well known, and explained in \S\ref{sec:M0n}, that $\D(2,[n] \setminus S)$ is in bijection with the set of planar trees $\{T_{[n] \setminus S}\}$ with leaves labeled cyclically by $[n] \setminus S$.  We obtain a map
\begin{equation}\label{eq:planartree}
\tDelta \longmapsto \T=\{T_{[n] \setminus S} \mid S \in \mbox{$\binom{[n]}{k-2}$}\}
\end{equation}
sending a hypersimplex subdivision to a collection of planar trees.

\begin{corollary}
The map \eqref{eq:planartree} is injective.
\end{corollary}
\begin{proof}
Suppose $\tDelta = \tDelta(p_\bullet)$ where $p_\bullet \in \Pluck(k,n)_{>0}$.  To determine $\tDelta$, by Theorem~\ref{thm:coincide} we need to know which cone of the Pl\"ucker fan structure $p_\bullet$ lies in.  Thus for every $S \subset [n]$ and $a<b<c<d$ not contained in $S$, we need to know which of the following situations we are in:\begin{align*}
p_{Sac} + p_{Sbd} &= p_{Sab} + p_{Scd} <  p_{Sad} + p_{Sbc}\\
p_{Sac} + p_{Sbd} &= p_{Sab} + p_{Scd} =  p_{Sad} + p_{Sbc}\\
p_{Sac} + p_{Sbd} &=  p_{Sad} + p_{Sbc} < p_{Sab} + p_{Scd}.
\end{align*}
Again, by Theorem~\ref{thm:coincide}, this is determined by $\tDelta_S \in \D(2,[n] \setminus S)$ which in turn is determined by $T_S$.
\end{proof}
Note however that not every collection $\{T_{[n] \setminus S} \mid S \in \binom{[n]}{k-2}\}
$ of planar trees are in the image of \eqref{eq:planartree}.
%The choice of weight vector $p_\bullet \in \Pluck_+(k,n)$ gives a metric structure on the trees $\T = \{T_S\}$, where $d_S(i,j) = -p_{Sij}$.  These metric trees satisfy the property
%\begin{equation}\label{eq:metric}
%d_S(i,j) = d_{S'}(i',j')
%\end{equation}
%whenever $S \sqcup \{i,j\} = S' \sqcup \{i',j'\}$.

\section{Nearly convergent functions}\label{sec:Gamma}
For more details on the material of this section we refer the reader to \cite{AHL1}.
\subsection{}\label{ssec:domain}
Let $f(\x) = p(\x)/q(\x) \in \R(\x):=\R(x_1,\ldots,x_r)$ be a subtraction-free rational function i.e., both $p(\x)$ and $q(\x)$ are polynomials with positive coefficients.  The piecewise-linear function $\Trop(f)$ on $\R^r$ is obtained by the substitution \eqref{eq:tropsubs}.

Suppose now that $f(\x) \in \R_{\geq 0}[x_1^{\pm 1}, \ldots, x_r^{\pm 1}]$ is a positive Laurent polynomial.    Let $\N[f(\x)]$ denote the Newton polytope of $f(\x)$ inside $\R^r$.  This is the lattice polytope given by the convex hull of $\vv \in \Z^r$, as $\vv$ varies over lattice points such that $\alpha \x^\vv$ is a monomial appearing in $f(\x)$.  The following result is well-known.
\begin{lemma}\label{lem:Newton}
For a positive Laurent polynomial $f(\x)$, the domains of linearity of $\Trop(f)$ consist of the chambers of a complete fan $\F$, and $\F$ is equal to the inner normal fan of the polytope $\N[f(\x)]$.
\end{lemma}

The normal fan of the Minkowski sum of two polytopes is the common refinement of the two normal fans.  
Thus if $f_1(\x),f_2(\x),\ldots,f_a(\x)$ are positive Laurent polynomials, then the common domains of linearity give a complete fan that is equal to the inner normal fan of the Minkowski sum $\N[f_1(\x)]+\cdots +\N[f_a(\x)]$.

\subsection{}
A subtraction-free rational function $f(\x) \in  \R(\x)$ is called {\it nearly convergent} if the function $\Trop(f(\x))$ takes nonnegative values on the whole of $\R^r$.  

\begin{lemma}\label{lem:nearlyconverge}
A  subtraction-free rational function $f(\x) = p(\x)/q(\x)$ is nearly convergent if and only if $\N[p(\x)] \subset \N[q(\x)]$.
\end{lemma}
\begin{proof}
Write $p(\x) = \sum_\vv \alpha_\vv \x^\vv$.  Then  
$$\Trop(f(\x)) = \Trop(\sum_\vv \alpha_\vv \x^\vv/q(\x) ) = \min_\vv(\Trop(\x^\vv/q(\x)).$$
Thus $\Trop(f(\x))$ is nonnegative if and only if $\Trop(\x^\vv/q(\x))$ is nonnegative for all $\vv$ such that a monomial $\alpha_\vv \x^\vv$ occurs in $p(\x)$.  Thus it suffices to establish the claim for the case $f(\x) = \x^\vv/q(\x)$, i.e., $\Trop(f(\x))$ is nonnegative if and only if $\vv \in \N[q(\x)]$.  This follows from the following observation: a point $\vv \in \Z^r$ is outside a polytope $P$ if and only if there exists a vector $\a \in \R^r$ such that $\a \cdot \vv < \min(\a \cdot \u \mid \u \mbox{ a vertex of $P$})$.
\end{proof}

\subsection{}\label{ssec:topgauge}
Set $d = k(n-k)$ and $r = d-(n-1)$.  Fix $\CC \subset \binom{[n]}{k}$ a cluster and $\G \subset \CC$ a gauge-fix as in \S\ref{ssec:gaugefix}.  Let $\CC \setminus \G = \{J_1,J_2,\ldots,J_r\}$ and denote $x_i:= \Delta_{J_i}$.  Let $\T(\CC,\G) \simeq (\C^\times)^{r}$ be the subtorus of $\T(\CC) := \Spec(\C[\Delta_I \mid I \in \CC])$ satisfying $\Delta_J = 1$ for all $J \in \G$.  Let $\Gr(k,n)_\G \subset \Gr(k,n)$ be the subspace where $\Delta_J = 1$ for all $J \in \G$.  We have a rational map $\pi_{\CC,\G}: \T(\CC,\G) \to \Gr(k,n)_\G$ induced by the positive parametrization $\T(\CC) \to \hGr(k,n)$ in \S\ref{ssec:posparam}.  We identify $\Gr(k,n)_\G$ birationally with $\Conf(k,n)$.  The following result follows from \eqref{eq:posgf} and Proposition~\ref{prop:Laurent}. 

\begin{proposition}\label{prop:picg}
The map $\pi_{\CC,\G}$ is birational, the restriction to $\R_{>0}^{d-(n-1)}$ is a homeomorphism onto $\Conf(k,n)_{>0}$, and every Pl\"ucker coordinate $\Delta_I$ pulls back under $\pi_{\CC,\G}$ to a positive Laurent polynomial $\Delta_I(\x)$ in $x_1,\ldots,x_r$.
\end{proposition}

\subsection{}

 A Laurent monomial 
\begin{equation}\label{eq:M}
M= \prod_{I\in \binom{[n]}{k}} \Delta_I^{a_I}, \qquad a_I \in \Z
\end{equation}
in the Pl\"ucker coordinates is $T$-invariant if it has weight 0, i.e., $\sum_I a_I e_I = 0$.  
For fixed $(\CC,\G)$, each Laurent monomial $M$ \eqref{eq:M} pulls back to a subtraction-free rational function $M(\x)$ by Proposition~\ref{prop:picg}.  We say that $M$ is nearly convergent with respect to $(\CC,\G)$ if $M(\x)$ is nearly converegent.  

\begin{lemma}
Let $M$ be a $T$-invariant monomial and let $(\CC,\G)$ and $(\CC',\G')$ be two choices of cluster and gauge-fix.  Then $M$ is nearly convergent with respect to $(\CC,\G)$ if and only if it is nearly convergent with respect to $(\CC',\G')$.
\end{lemma}
\begin{proof}
Pulling $M$ back to $\T(\CC)$ we have the notion of $M$ being nearly convergent with respect to $\CC$.  The $T$-invariance of $M$ implies that $M$ is nearly convergent with respect to $\CC$ if and only if it is nearly convergent with respect to some $(\CC,\G)$.  The positive parametrizations of two clusters $\CC$ and $\CC'$ are related by invertible subtraction-free rational transformations, and this implies that the notion of nearly convergent does not depend on cluster.
\end{proof}

\subsection{}\label{ssec:Px}
Let $P(\x) = P(x_1,\ldots,x_r) = \prod_I \Delta_I(x_1,\ldots,x_r)$ denote the product of all the Pl\"ucker variables, considered as a Laurent polynomial in $x_1,\ldots,x_r$.  Let $P(k,n) = \N[P(\x)] \subset \R^r$ be the Newton polytope of $P(\x)$.  Thus $P(k,n)$ is the Minkowski sum of the Newton polytopes $\N[\Delta_I(\x)]$.  By Theorem~\ref{thm:coincide} (specifically, the equivalence of the secondary fan structure and the positive fan structure) and Lemma~\ref{lem:Newton}, we have the following.
\begin{proposition}\label{prop:Ft}
There is a bijection $F \mapsto \tDelta(F)$ between the set of faces $\{F \subset P(k,n)\}$ of $P(k,n)$ and the set $\D(k,n)$ of regular subdivisions of the hypersimplex into positroid polytopes.
\end{proposition}

\subsection{}\label{ssec:Gamma}

Let
$$
\Gamma=\{M \mid \mbox{$M$ is $T$-invariant and nearly convergent}\}
$$ 
denote the (finitely-generated) monoid of nearly convergent, $T$-invariant, Laurent monomials in $\Delta_I$.  Let $\C[\Gamma] \subset \C(\Gr(k,n))^T$ be the subring of $T$-invariant rational functions on the Grassmannian generated by $M \in \Gamma$.  We will also identify $\C[\Gamma]$ with the subring $\C[M(\x) \mid M \in \Gamma] \subset \C(\x)$ of rational functions on the torus $\T(\CC,\G)$.

The following result follows from Lemma~\ref{lem:nearlyconverge} and the fact that each variable $x_i$ is the image in $\C(\x)$ of some $T$-invariant monomial $M$.  See also \cite[Section 10]{AHL1}.
\begin{lemma}\label{lem:vP}
The ring $\C[\Gamma] \subset \C(\x)$ is spanned by the rational functions $\x^\vv/P(\x)^\ell$ for an integer $\ell \geq 1$ and $\vv \in \ell \,P(k,n) \cap \Z^r$.
\end{lemma}

\begin{lemma}\label{lem:bound}
Let $f \in \C[\Gamma]$.  Then there exists a constant $D\in \R$ such that $|f(\x)| \leq D$ for all $\x \in \R_{>0}^r$. 
\end{lemma}
\begin{proof}
By Lemma~\ref{lem:vP}, it is enough to show that $\x^\vv/P(\x)^\ell$ is bounded above on $\R_{>0}^r$, where $\vv$ is a lattice point in the polytope $\ell P(k,n)$.  We may write $c \vv = \sum_i c_i \vv_i$ where $\vv_i$ are the vertices of $\ell\,P(k,n)$ and $c, c_i$ are nonnegative integers satisfying $c>0$ and $c = \sum_i c_i$.  Then $\x^{c\vv} = \prod_i (\x^{\vv_i})^{c_i} \leq (\sum_i \x^{\vv_i})^c$ and it follows that $\x^\vv/P(\x)^\ell \leq 1$ on $\R_{>0}^r$.
\end{proof}

\subsection{}
For a lattice polytope $P \subset \R^r$, one has a proper normal toric variety $X_{P}$ which depends only on the normal fan of $P$ \cite{Ful}.  The toric variety $X_P$ contains a dense torus $\T = (\C^\times)^r \subset X_P$.  The subspace $X_{P,>0}:=\R_{>0}^r \subset X_P(\R)$ is called the {\it positive part of $X_P$} and its closure $X_{P,\geq 0}$ is called the {\it nonnegative part of $X_P$}.  The torus orbits of $\T$ on $X_P$ stratify $X_P$ and $X_{P,\geq 0}$, and the strata are in bijection with faces of $P$.  We have a stratification-preserving homeomorphism between $X_{P, \geq 0}$ and the polytope $P$ with its face stratification.  For a face $F \subset P$, let $X_F$ denote the corresponding closed stratum, which is itself isomorphic to the toric variety for the polytope $F$.  Thus the closed (resp. relatively open) faces of $X_{P,\geq 0}$ are exactly the $X_{F,\geq 0}$ (resp. $X_{F,>0}$) as $F$ varies over faces of $P$.

\begin{proposition}[see \cite{AHL1}]\label{prop:XPkn}\
 \begin{enumerate}
\item
The variety $\Spec(\C[\Gamma])$ is isomorphic to an affine open subset $X'_{P(k,n)}$ of the projective (and normal) toric variety $X_{P(k,n)}$ associated to the normal fan of $P(k,n)$.  
\item
The affine open subset $X'_{P(k,n)}$ contains the nonnegative part $X_{P(k,n),\geq 0}$.
\item
Let $\vv \in  \ell \,P(k,n) \cap \Z^r$ and $F \subset P(k,n)$ be a face.  The function $\x^\vv/P(\x)^\ell \in \C[\Gamma]$ vanishes identically on $X_{F,\geq 0}$ if and only if $\vv \notin \ell F$.  The subspace $X_{F,\geq 0}$ is cut out of $X'_{P(k,n)}$ by the vanishing of such functions.
\end{enumerate}
\end{proposition}
\begin{proof}
(1) is proven in \cite[Section 10]{AHL1}; we sketch the main idea.  Recall that a full-dimensional lattice polytope $Q$ is called {\it very ample}  if for sufficiently large integers $s > 0$, every lattice point in $rQ$ is a sum of $s$ (not necessarily distinct) lattice points in $Q$.  Given $Q$, it is known that some dilate $\ell Q$ of $Q$ is very ample.  We then have a projective embedding of $X_Q$ given by the closure of the image of the map 
\begin{equation}\label{eq:closure}
\x \mapsto [\x^{\vv_0}: \x^{\vv_1}: \cdots : \x^{\vv_m}] \in \PP^m
\end{equation}
where $\ell Q \cap \Z^r = \{\vv_0,\vv_1,\ldots,\vv_m\}$ and $\x \in (\C^\times)^r$.  

We apply this construction to $X_{P(k,n)}$, supposing that $\ell P(k,n)$ is very ample.  Let $P(\x)^\ell = \sum_{i=0}^m c_i \x^{\vv_m}$, where $\ell P(k,n) \cap \Z^r = \{\vv_0,\ldots,\vv_m\}$.  Consider the linear section $H = \{c_0y_0 + \cdots + c_m y_m = 0\} \subset \PP^m$, where $y_i$ are the homogeneous coordinates on $\PP^m$.  The complement $X'_{P(k,n)} := X_{P(k,n)} \setminus H$ is an open subset of $X_{P(k,n)}$ that is an affine variety.  The coordinate ring $\C[X'_{P(k,n)}]$ is generated by the images of $y_i/(c_0y_0 + \cdots + c_m y_m)$, that is, $\x^{\vv_i}/P(\x)^\ell$.  By Lemma~\ref{lem:vP}, this gives the isomorphism $\Spec(\C[\Gamma])\simeq X'_{P(k,n)}$, establishing (1).

Since all the coefficients of $P(\x)$ are positive, the linear section does not intersect $X_{P(k,n),\geq 0}$, giving (2).  

The torus orbit closure $X_F$ is cut out from $X_P$ by the vanishing of the homogeneous coordinates $\{\x^{\vv} \mid \vv \notin F \cap \Z^r\}$ in the embedding \eqref{eq:closure}.  This gives (3).
\end{proof}

\section{Topology of nonnegative configuration space}\label{sec:ball}
%\fixit{Define Zariski closure of $\Theta_{\tDelta}$ Estimate for dimension}
%\subsection{Minkowski sum}
%Let $(\C^\times)^{r} \to \Gr(k,n)/T$ be a positive parameterization and let $x_1,\ldots, x_r$ be the natural coordinate functions on $(\C^\times)^{r}$.  We assume that $\Conf(k,n)$ has been identified birationally with a subspace of $\Gr(k,n)$ where some subset of the Pl\"ucker variables have been set to 1.  Then each Pl\"ucker variable $\Delta_I$ is a Laurent polynomial $\Delta_I(x_1,\ldots,x_r)$, and for simplicity we assume these polynomials have positive coefficients, and in addition we assume that $x_1,\ldots,x_r$ are among this list of Laurent polynomials.
%
%For a positive Laurent polynomial $f(x_1,\ldots,x_r)$, we let $\Trop(f)$ denote the piecewise-linear function on $\R^r$ obtained by the substitution \eqref{eq:tropsubs}.  The maximal domains of linearity of $\Trop(f)$ forms a complete fan on $\R^r$.  This fan is the normal fan to the Newton polytope $\N[f]$ of $f$.  Given a collection $f_1,f_2,\ldots,f_a$ of Laurent polynomials, the common refinement of the domains of linearity of the piecewise-linear functions $\Trop(f_1),\Trop(f_2),\ldots,\Trop(f_a)$ is the normal fan to the Minkowski sum of the Newton polytopes $\N[f_i]$.
%
%Let $P(\x) = P(x_1,\ldots,x_r) = \prod_I \Delta_I(x_1,\ldots,x_r)$ denote the product of all the Pl\"ucker variables, considered as a Laurent polynomial in $x_1,\ldots,x_r$.  Let $P(k,n) = \N[P(\x)] \subset \R^r$ be the Newton polytope.  Thus $P(k,n)$ is the Minkowski sum of the Newton polytopes $\N[\Delta_I(x_1,\ldots,x_r)]$.  

\subsection{} We continue the notation of \S\ref{ssec:topgauge}.  In particular, a cluster $\CC \subset \binom{[n]}{k}$ and a gauge-fix $\G \subset \CC$ are fixed.  Since $\Conf(k,n) \simeq \oGr(k,n)/T$, each $T$-invariant monomial extends to a rational function on $\bConf(k,n)$.  Let $\tConf(k,n) \subset \bConf(k,n)$ denote the locus where all nearly convergent monomials are regular i.e. where the polar locus of each $M \in \Gamma$ has been removed.  Then we have a natural morphism 
\begin{equation}\label{eq:varphi}
\varphi: \tConf(k,n) \longrightarrow \Spec(\C[\Gamma]) = X'_{P(k,n)}.
\end{equation}
It follows from Lemma~\ref{lem:bound} that $\Ch(k,n)_{\geq 0} \subset \tConf(k,n)$.  Recall the bijection $F \mapsto \tDelta(F)$ in Proposition~\ref{prop:Ft}.

\begin{proposition}\label{prop:strata}
Let $F \subset P(k,n)$ be a face.  If $X \in \Theta_{\tDelta(F)}(k,n)$, then $\varphi(X) \in X_{F,> 0} \subset  X'_{P(k,n)}$.
\end{proposition}
\begin{proof}
Identify $\Pluck(k,n)_{>0}/\sim$ with a complete fan on $\R^r$.  Via Proposition~\ref{prop:Ft}, the cones $C(F) = C(\tDelta(F))$ are indexed by either faces $F$ of $P(k,n)$ or by $\tDelta(F) \in \D(k,n)$.  Let $\gamma(t)$ be an (analytic) curve in $\Gr(k,n)_{>0}/T_{>0}$ such that $\lim_{t \to 0} \gamma(t) = X \in \Theta_{\tDelta(F)}(k,n)$.  Then $\gamma(t)$ gives rise to a formal curve, and thus a Puiseux curve $V(t) \in \Gr(k,n)(\RR_{>0})$.  The positive tropical Pl\"ucker vector $\val(V(t))$ lies in the cone $C(\tDelta(F))$.

Now let $\vv \in (\ell P(k,n) \cap \Z^r) \setminus (\ell F \cap \Z^r)$.  Let $f(\x) =  \x^\vv/P(\x)^\ell$.  The function $\Trop( \x^\vv/P(\x)^\ell)$ is nonnegative on $\R^r$ and strictly positive on the cone $C(F)$ in the fan $\Pluck(k,n)_{>0}$.  Considering $f(\x)$ as a $T$-invariant rational function on $\Gr(k,n)$, we conclude that $\val(f(V(t)))>0$.  Thus $f(\x)$ vanishes at $X$ and by Proposition~\ref{prop:XPkn}(3), $\varphi(X) \in X_{F,\geq 0}$.  The same argument shows that we must actually have $\varphi(X) \in X_{F,> 0}$.
\end{proof}

\subsection{}\label{ssec:homeo}
We can now prove Theorem~\ref{thm:ball}.

\begin{theorem}\label{thm:homeo}
There is a stratification-preserving homeomorphism $\Ch(k,n)_{\geq 0} \simeq P(k,n)$.  In particular, each stratum $\Theta_{\tDelta, \geq 0}$ (resp. $\Theta_{\tDelta, > 0}$) is homeomorphic to a closed ball (resp. open ball) of dimension $\dim(\tDelta)$.
\end{theorem}

We shall show that the map $\varphi$ restricts to a stratified bijection $\varphi: \Ch(k,n)_{\geq 0} \to X_{P(k,n),\geq 0}$.  Since $\Ch(k,n)_{\geq 0}$ is compact and $X_{P(k,n),\geq 0}$ is Hausdorff, this claim would establish Theorem \ref{thm:homeo}.  Proposition~\ref{prop:strata} shows that $\varphi$ is stratification-preserving.  It is easy to see that $\varphi:\Conf_{\geq 0} \to X_{P(k,n),\geq 0}$ is surjective.  If $p \in X_{P(k,n),\geq 0}$ then $p = \lim_{t \to 0} p(t)$ where $p(t) \in X_{P(k,n),> 0}$ for $t >0$.  The curve $p(t)$ can be lifted to a curve $X(t) \in \Conf(k,n)_{>0}$, such that $\varphi(X(t)) = p(t)$ for $t>0$.  Then $\varphi(\lim_{t \to 0} X(t) ) = p$.  

It remains to argue that the map $\varphi$ is injective.  A point $X = \sum_{i=1}^r X_i \in \Ch(k,n)_{\geq 0}$ is determined by the torus orbit closures $X_i = \overline{T \cdot V_i}$ where $V_i \in \Pi_{\M_i, >0}/T$ and $P_{\M_1},\ldots,P_{\M_r}$ are positroid polytopes that form a regular polyhedral subdivision of the hypersimplex.  The point $X$ is uniquely determined by the collection of points $\{V_1,\ldots,V_r\}$.  Thus, the injectivity of $\varphi$ follows from the following proposition which completes the proof of Theorem~\ref{thm:homeo}.

\begin{proposition}\label{prop:TV}
Let $X \in  \sum_{i=1}^r X_i  \Theta_{\tDelta,\geq 0}$ with $X_i = \overline{T \cdot V_i}$ where $V_i \in \Pi_{\M_i, >0}/T_{>0}$.  Then $V_1,\ldots,V_r$ is determined by the values $f(V_i)$ for varying $f \in \C[\Gamma]$.
\end{proposition}
\begin{proof}
Note that for $f \in \C[\Gamma]$, we have $f(V_i) = f(V_j)$ whenever $f$ makes sense on both $V_i$ and $V_j$.
%The values $M(V)$ for $M \in \Gamma$ determine $f(V)$ for all $f \in \C[\Gamma]$.
Let us first consider the special case where $\M=\binom{[n]}{k}$, and suppose $V \in \Gr(k,n)_{>0}/T_{>0}$.  By Lemma~\ref{lem:vP}, the function $\x^\vv/P(\x)^\ell$ is nearly convergent when $\vv \in \ell P(k,n)$.  For sufficiently large $\ell$, we can find lattice points $\vv$ and $\vv+e_i$ inside $\ell P(k,n)$, and the ratio of $\frac{\x^\vv}{P(\x)^\ell}$ and $ \frac{\x^{\vv+e_i}}{P(\x)^\ell}$ is equal to $x_i$, so the value of $x_i(V)$ is determined by $f(V)$, for $f \in \Gamma$.  (Note that the ratios $\frac{\x^\vv}{P(\x)^\ell}$ are always positive because all Pl\"ucker variables are positive on $V$.)  By Proposition~\ref{prop:picg}, $V \in \Gr(k,n)_{>0}/T_{>0}$ is determined by the positive parameters $x_1,\ldots,x_r \in \R_{>0}$, and thus we have recovered the point $V$.

Now suppose that $V \in \Pi_{\M, >0}/T_{>0}$ where $\M$ is an arbitrary connected positroid.  Applying Lemma~\ref{lem:span} we find a cluster $\CC' \subset \M$ and a gauge-fix $\G \subset \CC'$.  By Proposition~\ref{prop:OPS}, we can find a cluster $\CC$ of $\binom{[n]}{k}$ that contains $\CC'$.  We use $(\CC,\G)$ as our positive parametrization of $\Conf(k,n)_{>0}$, and we suppose that $x_1,\ldots,x_b$ belong to $\CC' \setminus \G$ while $x_{b+1},\ldots,x_r$ belong to $\CC \setminus \CC'$.  Here, $b = \dim(\M)-(n-1)$.  To determine $V$, we need to determine $x_1(V),\ldots,x_b(V)$.  The same argument as in the previous paragraph applies, except we need to ensure that the ratios $\frac{\x^\vv}{P(\x)^\ell}$ and $ \frac{\x^{\vv+e_i}}{P(\x)^\ell}$ used do not vanish on $\Pi_{\M,>0}/T_{>0}$. As in the proof of Proposition~\ref{prop:strata}, the function $x^\vv/P(\x)^\ell$ vanishes on $\Theta_{\tDelta}$ exactly when $\vv$ does not belong to the face $\ell F$ of $\ell P(k,n)$.

%
%
%Using Lemma \ref{lem:span}, we may find a positive parametrization for $(x_1 = \Delta_{I_1},\ldots,x_r = \Delta_{I_r})$ for $\Gr(k,n)/T$ such that $(x_1,\ldots,x_d)$ give a positive parametrization for $\Pi_\M/T$, where $d = \dim(\M)-n+1$.  To determine $V \in \Pi_{\M,>0}/T$, we need to determine $x_1(V),\ldots,x_d(V)$.  The same argument as in the previous paragraph applies, except we need to ensure that the ratios $\frac{\x^\v}{P(\x)^\ell}$ and $ \frac{\x^{\v+e_i}}{P(\x)^\ell}$ used do not vanish on $\Pi_\M/T$.  Let $F \subset P(k,n)$ be the face such that $\tDelta = \tDelta(F)$.  As in the proof of Proposition~\ref{prop:strata}, the function $x^\v/P(\x)^\ell$ vanishes on $\Theta_{\tDelta}$ exactly when $\v$ does not belong to the face $\ell F$ of $\ell P(k,n)$.

We claim that for each $i = 1,2,\ldots,b$, the face $\ell F$ contains lattice points $\vv$ and $\vv+e_i$ for some $\vv$.  In other words, if $\y = (y_1,\ldots,y_r) \in \R^r$ is a vector in the normal cone $C(F)$ to $F$ then $y_i = 0$ for $i = 1,2,\ldots,b$.  Let $p_\bullet$ be a positive tropical Pl\"ucker vector such that $\tDelta(p_\bullet) = \tDelta$.  By Lemma~\ref{lem:span}, there is a unique $p'_\bullet$ equivalent to $p_\bullet$ such that $p'_I = 0$ for all $I \in \G$, and by Lemma~\ref{lem:facezero}, it must be the case that $p'_J = 0$ for all $J \in \M$, since $P_\M$ is one of the full-dimensional pieces in $\tDelta$.  By Corollary~\ref{cor:param}, setting $y_i = p'_{I_i}$ for $I_i \in \CC$ gives a vector in $C(F)$, and any vector in the relative interior of $C(F)$ is obtained in this way.  All these vectors satisfy $y_1 = y_2 = \cdots = y_b = 0$ since $p'_J =0$ for all $J \in \M$.
\end{proof}

Let us illustrate the proof of Proposition~\ref{prop:TV} with an example for $(k,n) = (3,6)$. Take the positive parametrization 
\begin{equation}\label{eq:36param1}
(z_1,z_2,z_3,z_4) \in (\C^\times)^4 \mapsto
\begin{bmatrix}
 0 & 0 & -1 & -1 & -1 & -1 \\
 0 & 1 & 0 & -1 & -1-z_1 & -1-z_1z_2-z_2 \\
 1 & 0 & 0 & 1 & 1+z_1+z_3 & 1+z_1+z_2+z_4+\frac{z_3 z_3}{z_1}+\frac{z_4}{z_1}
   \\
\end{bmatrix}
\end{equation}
This is the positive parametrization associated to the gauge-fix and cluster 
$$\G:=\{123,124,125,126,134,234\} \text{ and } \CC := \G \cup \{z_1=145,z_2=156,z_3=345,z_4=456\}.
$$  Consider the two-piece hypersimplex decomposition $\tDelta$ (see \S\ref{ssec:36}) that slices $\Delta(3,6)$ with the hyperplane $x_4+x_5 +x_6= 2$.  On one side we have the positroid polytope for $\M = \binom{[6]}{3}\setminus \{4,5,6\}$, which has a cluster
$$
\CC \supset \CC' := \{123,124,125,126,134,234,145,156,345\} \supset \G.
$$
Now, the normal cone $C(F) = C(\tDelta)$ is a ray and it is generated by the vector $\y = r'_6=(0,0,0,1)$.  (This $r'_6$ is the image of $r_6$ in \S\ref{ssec:36} under the linear transformation that tropicalizes the monomial transformation $(x_1,x_2,x_3,x_4) \to (z_1,z_3/z_1, z_2/z_1, z_4/(z_2z_3))$.)  Agreeing with the proof of Proposition~\ref{prop:TV}, we have $y_1=y_2=y_3=0$.
\subsection{}
We end this section with the following question.

\begin{question}
Is the morphism $\varphi$ of \eqref{eq:varphi} an isomorphism of algebraic varieties?
\end{question}
We note that $\Spec(\C[\Gamma])$ is a normal variety, while the Chow quotient $\bConf(k,n)$ has rather complicated geometry (c.f \cite{KT,Laf}).

\section{\texorpdfstring{$\M_{0,n}$ and the case $k=2$}{M0n and the case k=2}}\label{sec:M0n}
\subsection{}
It is well-known \cite{Kap} that we have $\Conf(2,n) = \oConf(2,n) = \M_{0,n}$ and $\bConf(2,n) = \bM_{0,n}$, the moduli space of $n$ points on $\PP^1$ and its Deligne-Knudsen-Mumford compactification respectively.  The space $\M_{0,n}(\R)$ consists of $n$ points $z_1,\ldots,z_n$ on a circle $S^1$.  It is a smooth open manifold with $(n{-}1)!/2$ connected components, each of which is diffeomorphic to an open ball of dimension $n{-}3$.  Each connected component is given as the subspace where the $n$ points are in a fixed dihedral ordering.  Fixing such an ordering $z_1 < z_2 < \cdots < z_n$ (up to dihedral symmetries) we obtain the positive component $(\M_{0,n})_{>0} \subset \M_{0,n}(\R)$.  It is well-known that the closure $\Conf(2,n)_{\geq 0} = (\M_{0,n})_{\geq 0}$ of $(\M_{0,n})_{>0}$ in $\M_{0,n}(\R)$ is homeomorphic as a stratified space to the associahedron $\A_n$, agreeing with Theorem \ref{thm:homeo}.  The affine variety $\tM_{0,n} := \tConf(2,n)$ sits in between $\M_{0,n}$ and $\bM_{0,n}$.  It can be obtained from $\bM_{0,n}$ by removing all boundary divisors of $\bM_{0,n}$ that do not intersect $ (\M_{0,n})_{\geq 0}$ in codimension one, see \cite{Brown}.

\subsection{}
Let us now spell out our combinatorial constructions in this case.  A positroid $\M$ of rank $2$ on $[n]$ is given by a collection of conditions of the following form:
\begin{enumerate}
\item For some $i \in [n]$, we have $i \in I$ for all $I \in \M$.
\item For some $j \in [n]$, we have $j \notin I$ for all $I \in \M$.
\item For some cyclic interval $[a,b]$, we have $|[a,b] \cap I| \leq 1$ for all $I \in \M$.
\end{enumerate}
For $\M$ to be connected we must have none of the conditions of the form (1) or (2).  Such a connected positroid $\M$ is then determined by a decomposition $[n] = \bigsqcup_{i=1}^r [a_i,b_i]$ of $[n]$ into at least three cyclic intervals such that $\M$ is given by
$$
\M([a_i,b_i]) = \left\{I \in \mbox{$\binom{[n]}{2}$} \mid |I \cap [a_i,b_i]| \leq 1 \text{ for } i = 1,2,\ldots,r\right\}.
$$
We have the formula $\dim(\M) = n+r-4$.  For example, if $\M$ is the uniform matroid then $r = n$ and $\dim(\M) = 2n-4 = 2(n-2)$.  If $r = 3$ then $\dim(M) = n-1$ and $\M$ is a minimal connected positroid.  

For a connected positroid $\M = \M([a_i,b_i])$, it is not difficult to see that there is a canonical isomorphism $\oPi_\M/T \simeq \M_{0,r}$ that sends $\Pi_{\M,>0}/T_{>0}$ to $(\M_{0,r})_{>0}$, where $r$ is the number of cyclic intervals.  If $r = 3$, then $\M_{0,3}$ is a point, as expected.
\subsection{}
The faces of the associahedron $\A_n$ are labeled by planar trees $T$ with $n$ cyclically ordered leaves $1,2,\ldots,n$, with internal vertices having degree at least $3$.  Some such trees for $n =5$ are drawn below:
$$
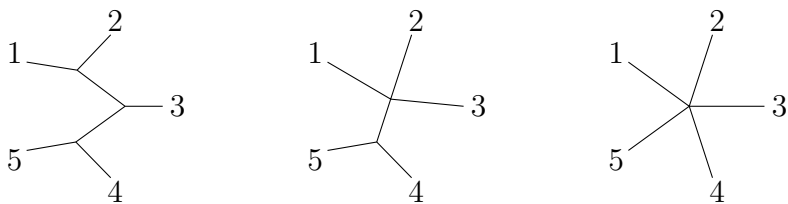
\begin{tikzpicture}
\draw (0:1.2) node {$3$};
\draw (72:1.2) node {$2$};
\draw (144:1.2) node {$1$};
\draw (216:1.2) node {$5$};
\draw (288:1.2) node {$4$};
\draw (72:1)--(106:0.5)--(144:1);
\draw (216:1)--(252:0.5)--(288:1);
\draw (106:0.5)--(0:0.5)--(0:1);
\draw (252:0.5)--(0:0.5);
\begin{scope}[shift={(4,0)}]
\draw (0:1.2) node {$3$};
\draw (72:1.2) node {$2$};
\draw (144:1.2) node {$1$};
\draw (216:1.2) node {$5$};
\draw (288:1.2) node {$4$};
\draw (72:1)--(72:0.1)--(144:1);
\draw (72:0.1)--(0:1);
\draw (216:1)--(252:0.5)--(288:1);
\draw (252:0.5)--(72:0.1);
\end{scope}
\begin{scope}[shift={(8,0)}]
\draw (0:1.2) node {$3$};
\draw (72:1.2) node {$2$};
\draw (144:1.2) node {$1$};
\draw (216:1.2) node {$5$};
\draw (288:1.2) node {$4$};
\draw (72:1)--(0:0)--(144:1);
\draw (0:1)--(0:0)--(216:1);
\draw (288:1)--(0:0);
\end{scope}
\end{tikzpicture}
$$

Let $\Vert(T)$ denote the set of internal vertices of $T$.  Let $T$ be such a planar tree and $v \in \Vert(T)$ be an internal vertex of degree $r \geq 3$.  Removing $v$ from $T$ we obtain a forest with $r$ components, and this decomposes $[n]$ into $r$ cyclic intervals $[a_1,b_1],\ldots,[a_r,b_r]$, giving a positroid $\M(v)$.  Thus $I = \{i,j\}$ is a basis of the positroid $\M(v)$ if and only if the unique path from $i$ to $j$ in $T$ passes through $v$.

\begin{proposition}[\cite{Kap}]
The map
$$
T \longmapsto \tDelta_T:=\{P_{\M(v)} \mid v \in \Vert(T)\}
$$
gives a bijection between planar trees with $n$ leaves and regular subdivisions of the hypersimplex $\Delta(2,n)$ into positroid polytopes.
\end{proposition}

These hypersimplex decompositions can be obtained by intersecting $\Delta(k,n)$ with the hyperplanes $H_e:=\{x_a+x_{a+1}+\cdots+x_b = 1\}$, one for each each internal $e$ of $T$, where for an internal edge $e$, we set $[a,b]$ to be the cyclic interval of leaves on one side of $e$.  The positroid polytope $P_{\M(v)}$ has as interior facets the $H_e$ where $v$ is incident to $e$.  If $T$ is the star with a single internal vertex $v$ and no edges then $\tDelta_T = \{P_{\M(v)} =\Delta(2,n)\}$ is the trivial decomposition.  

\subsection{}
Now let $p_\bullet \in \Pluck(2,n)_{>0}$ be a (finite) positive tropical Pl\"ucker vector.  We assign a planar tree $T(p_\bullet)$ to $p_\bullet$ as follows.  For $1 \leq a<b<c<d \leq n$, we consider which of the following three possibilities holds:
\begin{enumerate}
\item
$p_{ac} + p_{bd} = p_{ab} + p_{cd} < p_{ad} + p_{bc}$,
\item
$p_{ac} + p_{bd} = p_{ad} + p_{bc}<  p_{ab} + p_{cd}$,
\item
$p_{ac} + p_{bd} = p_{ab} + p_{cd} = p_{ad} + p_{bc}$.
\end{enumerate}
The tree $T(p_\bullet)$ has the property that 
\begin{enumerate}
\item the shortest path from leaf $a$ to leaf $d$ does not intersect the shortest path from leaf $b$ to leaf $c$,
\item the shortest path from leaf $a$ to leaf $b$ does not intersect the shortest path from leaf $c$ to leaf $d$,
\item
there is an internal vertex $v$ such that any shortest path between two of the vertices $a,b,c,d$ passes through $v$,
\end{enumerate}
respectively.
\begin{proposition}
The map $p_\bullet \mapsto T(p_\bullet)$ defines a fan structure on $\Pluck(2,n)$ that agrees with the ones in \S\ref{sec:fan}.
\end{proposition}

\subsection{}
Let $T$ be a planar tree and $\tDelta_T$ be the corresponding hypersimplex subdivision.  Suppose that $\tDelta_T = \{P_{\M(v)}\}$.  We have a homeomorphism 
\begin{equation}\label{eq:Ch}
\Theta_{\tDelta_T,>0} \simeq \prod_{v \in \Vert(T)} \Pi_{\M(v),>0}/T_{>0}
\end{equation}
and $\dim(\tDelta_T) = \dim(\Theta_{\tDelta_T,>0}) = n{-}2{-}|\Vert(T)|$.  This agrees with $\Pi_{\M(v),>0}/T_{>0} \simeq \R^{\deg(v) - 3}$, since $\sum_{v \in \Vert(T)} (\deg(v){-} 3) = n{-}2{-}|\Vert(T)|$.  For $k > 2$, we do not know an easy estimate for $\dim(\tDelta)$, nor does the ``factorization" in \eqref{eq:Ch} hold.  We will study of the geometry of $\Theta_{\tDelta,>0}$ in future work \cite{ALS2}.

\subsection{}
We shall use the following positive parametrization of $\Conf(2,n)$:
\begin{equation}\label{eq:Markparam}
\begin{bmatrix} 0 & 1&  1& 1&1& \cdots & 1 \\
-1 & 0& 1& 1+x_1 & 1+x_1+x_1x_2 & \cdots & 1+x_1+x_1x_2+ \cdots +x_1x_2\cdots x_{n-3}
\end{bmatrix}
\end{equation}
In this positive parametrization, the Pl\"ucker coordinates for $12,13,\ldots,1n$ and $23$ have been gauge-fixed to 1, while the positive parameters $(x_1,\ldots,x_{n-3})$ are given by $x_i = \Delta_{i+2,i+3}/\Delta_{i+1,i+2}$.  Thus, our positive parameterization is related to the one for
$$
(\CC,\G) = (\{12,13,\ldots,1n,23,34,\ldots,(n-1)n\},\{12,13,\ldots,1n,23\})
$$
by a monomial transformation.  Our reasons for choosing this parametrization will be explained in \cite{ALS2}.  An explicit description of the fan structure in these positive coordinates is given in \cite{SW}.

Let us take $n = 5$.  Then the non-monomial Pl\"ucker variables are $\Delta_{24} = 1+x_1$, $\Delta_{25} = 1+x_1+x_1x_2$, and $\Delta_{35} = x_1+x_1x_2$.  The common domains of linearity of $\Trop(\Delta_{24}) = \min(0,X_1)$, $\Trop(\Delta_{25}) = \min(0,X_1,X_1+X_2)$ and $\Trop(\Delta_{35}) = X_1 + \min(0,X_2)$ give the following fan:
%\begin{figure}
%\begin{center}
$$
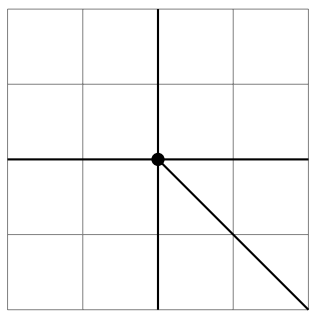
\begin{tikzpicture}
\draw [help lines, step=1cm] (-2,-2) grid (2,2);
\node[fill=black,circle, inner sep=0pt,minimum size=5pt] at (0,0) {};
\draw[thick] (0,0) --(2,0);
\draw[thick] (0,0) --(-2,0);
\draw[thick] (0,0) --(0,-2);
\draw[thick] (0,0) --(0,2);
\draw[thick] (0,0) --(2,-2);
%\draw[thick] (0,0) --(-2,-2);
%\draw[thick] (0,0) --(-2,2);
\end{tikzpicture}
$$
%\caption{The normal fan of the heptagon.}
%\label{fig:heptagon}
%\end{center}
%\end{figure}
As an example, let us consider the integer vector $(1,-1)$ lying on the southeast pointing ray, and substitute $(x_1,x_2) = (t,1/t)$ into \eqref{eq:Markparam} to obtain
$$
V(t) = \begin{bmatrix} 0 & 1&  1& 1&1\\
-1 & 0& 1& 1+ t & 2+t  
\end{bmatrix}
$$
The tropical Pl\"ucker vector $p_\bullet =\val( \Delta_\bullet(V(t)))$ is given by $p_{34} = 1$ and $p_J = 0$ for $J \neq 34$.  Taking $\a = (1/2,1/2,-1/2,-1/2,1/2)$ in \eqref{eq:equiv}, we see that $p_\bullet \sim p'_\bullet$, where $p'_{12}=p'_{15}=p'_{25} = 1$ while $p_J =0$ for $J \notin \{12,15,25\}$.
%
%$$
%p_I = \begin{cases} 1 & \mbox{if $I = 34$} \\ 
%0 &\mbox{otherwise}
%\end{cases}
%$$
Thus the corresponding hypersimplex subdivision $\tDelta(p_\bullet)$ consists of two positroid polytopes $P_{\M_1}$ and $P_{\M_2}$ where
$$
\M_1 = \mbox{$\binom{[5]}{2}$} \setminus \{34\} \qquad \text{and} \qquad \M_2=\{13,14,23,24,34,35,45\}
$$
separated by the hyperplane $x_3+x_4 = 1$.  The planar tree is
$$
 
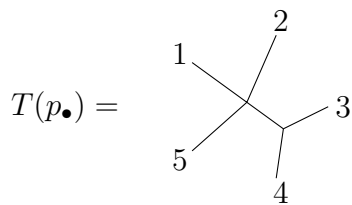
\begin{tikzpicture}
\node at (180:2.5) {$T(p_\bullet)=$};
\draw (0:1.2) node {$3$};
\draw (72:1.2) node {$2$};
\draw (144:1.2) node {$1$};
\draw (216:1.2) node {$5$};
\draw (288:1.2) node {$4$};
\draw (72:1)--(144:0.1)--(144:1);
\draw (144:0.1)--(216:1);
\draw (288:1)--(-36:0.5)--(0:1);
\draw (-36:0.5)--(144:0.1);
\end{tikzpicture}
$$

%has two internal vertices $v_1$ (joined to $1,2,v_2,5$) and $v_2$ (joined to $v_1,3,4$).

\subsection{}
Let us fix a cluster $\CC\subset \binom{[n]}{2}$.  In the case $k=2$, the polytope $P(2,n)$ has the special feature that it is simple, and thus every cone of the normal fan $\F$ is a simplicial cone.  Let us denote the (first integer point on the) rays of $\F$ by $r_{ij}$, corresponding to the tree with a single interior edge separating leaves $i+1,\ldots,j$ from $j+1,\ldots,i-1,i$, where $(i,j)$ varies over the diagonals of a polygon with vertices $1,2,\ldots,n$.  Also write $p^{ij}_\bullet$ for the tropical Pl\"ucker vector that maps to $r_{ij}$ under Theorem~\ref{thm:param}.  As explained in \cite{AHL1}, a consequence of the simplicial-ness of $\F$ is that the ring $\C[\Gamma]$ has some particularly nice generators.
For $(i,j)$ a diagonal of the polygon with vertices $1,2,\ldots,n$, define the rational function
$$
u_{ij} = \frac{\Delta_{i,j+1} \Delta_{i+1,j} }{\Delta_{ij} \Delta_{i+1,j+1}} = \frac{(z_i-z_{j+1})(z_{i+1}-z_j)}{(z_i-z_j)(z_{i+1}-z_{j+1})}
$$
on $\Conf(2,n)$, which can be interpreted as a cross ratio of the four points $z_i,z_{i+1},z_j,z_{j+1}$ on $\PP^1$.  The functions $u_{ij}$ have the following special property.

\begin{proposition}[\cite{AHL1,AHL2}] The ring $\C[\Gamma]$ of \S\ref{sec:Gamma} is the subring of $\C(\Gr(k,n))$ generated by the $u_{ij}$.  For two diagonals $(i,j)$ and $(i',j')$, we have $$\Trop(u_{ij})(r_{i'j'}) = \Trop(u_{ij})(p^{i'j'}_\bullet) = \delta_{ij,i'j'}.$$
\end{proposition}
Note that $\Trop(u_{ij}) = p_{i,j+1} + p_{i+1,j} - p_{ij} - p_{i+1,j+1}$, so it is easy to see that it takes nonnegative values on any tropical Pl\"ucker vector $p_\bullet$, and thus $u_{ij}$ is nearly convergent. It is not hard to see that all the inequalities from \eqref{eq:trop} are positive linear combinations of $\Trop(u_{ij})$ for various $i,j$.

Let  $\G \subset \CC$ be a gauge-fix.  Setting $\Delta_I = 1$ for $I \in \G$, it is not difficult to see that there is an invertible monomial transformation between the $n(n{-}3)/2$ functions $u_{ij}$ and the functions $\Delta_J$ with $J \notin \CC \setminus \G$.  The fan structure on $\Pluck(2,n)_{>0}$ is given by intersecting the cones $\Trop(u_{ij}) = 0$ and $\Trop(u_{ij}) > 0$ as $i,j$ vary.  

For $n= 5$ and the parametrization \eqref{eq:Markparam}, the $u_{ij}$ functions are
\begin{align*}
u_{13} &= \frac{1}{1+x_1},&  u_{14}&= \frac{1 + x_1}{1 + x_1 + x_1 x_2},& u_{24}&=\frac{x_1 (1 + x_1 + x_1 x_2)}{(1 + x_1) (x_1 + x_1 x_2)} \\
 u_{25}&= \frac{x_1 + x_1 x_2}{1 + x_1 + x_1 x_2}, &u_{35}&=\frac{x_1 x_2}{x_1 + x_1 x_2}
\end{align*}
which are easily seen to belong to $\Gamma$ using Lemma~\ref{lem:nearlyconverge}.

\section{Tropical bridge reduction}
\label{sec:bridge}

We will frequently use the following easy result without mention.
 \begin{lemma}
 We have $I \leq J$ if and only if $(I \setminus J) \leq (J \setminus I)$.
 \end{lemma}
\subsection{}
Let $q_\bullet$ be a positive tropical Pl\"ucker vector.  Pick $a \in \Z$ and $i \neq j \in [n]$.  We define the tropical bridge $T_\gamma(a) = T_{i,j}(a)$ acting on $\R^{\binom{[n]}{k}}$ by $q'_\bullet = T_\gamma(a)  \cdot q_\bullet$ where
$$
q'_I = \begin{cases} \min(q_I, q_{I \setminus \{j\} \cup \{i\}} + a)& \mbox{if $j \in I$ but $i \notin I$} \\
q_I & \mbox{otherwise.}
\end{cases}
$$
This formula is the tropicalization of \eqref{eq:Deltax}.  When $j=i+1$, we write $T_i(a):= T_{i,i+1}(a)$.
\begin{proposition}\label{prop:bridge}
If $q_\bullet \in \Pluck(k,n)_{\geq 0}$ then $T_i(a) \cdot q_\bullet \in \Pluck(k,n)_{\geq 0}$.
\end{proposition}
\begin{proof}
By Theorem~\ref{thm:main}, every $q_\bullet \in \Pluck(k,n)_{\geq 0}$ is representable by $V \in \Gr(k,n)(\RR_{\geq 0})$.  The claim then follows immediately from \eqref{eq:Deltax}.  
\end{proof}

\begin{remark}
Our proof below gives proofs of Theorem~\ref{thm:main} and of Proposition~\ref{prop:bridge} that are independent of our earlier proof of Theorem~\ref{thm:main}.  We will only apply Proposition~\ref{prop:bridge} to $q_\bullet \in \Pluck(k,n)_{\geq 0}$ that we separately know to be representable.  And once Theorem \ref{thm:main} is established, Proposition~\ref{prop:bridge} follows for arbitrary $q_\bullet \in \Pluck(k,n)_{\geq 0}$.
\end{remark}

\subsection{}
We shall show that bridge reduction (Proposition~\ref{prop:bridgedecomp}) holds for positive tropical Pl\"ucker vectors.  

\begin{proposition}\label{prop:tropreduce}
Let $p_\bullet \in \Pluck(\M)_{>0}$ where $\M$ is a rank $k >0$ positroid on $[n]$.  Then we show that at least one of the following holds:
\begin{enumerate}
\item For some $i \in [n]$, we have $f_\M(i) = i$.  Define $\epsilon_i: [n-1] \to [n]$ by $\epsilon_i(a) = a$ if $a <i$ and $\epsilon_i(a) = a+1$ if $a \geq i$.  Then $p_\bullet$ is in the image of the map $\epsilon_i:\Pluck(k,n-1)_{\geq 0} \hookrightarrow \Pluck(k,n)_{\geq 0}$ given by 
$$
\epsilon_i(q_\bullet)_I = \begin{cases} \infty & \mbox{if $i \in I$} \\
q_J & \mbox{if $I = \epsilon_i(J)$.}
\end{cases}
$$ 
\item For some $i \in [n]$, we have $f_\M(i) = i+n$.  Define $\epsilon_i: [n-1] \to [n]$ by $\epsilon_i(a) = a$ if $a <i$ and $\epsilon_i(a) = a+1$ if $a \geq i$.  Then $p_\bullet$ is in the image of the map $\varepsilon_i:\Pluck(k-1,n-1)_{\geq 0} \hookrightarrow \Pluck(k,n)_{\geq 0}$ given by
$$
\varepsilon_i(q_\bullet)_I = \begin{cases} \infty & \mbox{if $i \notin I$} \\
q_J & \mbox{if $I = \epsilon_i(J) \cup\{i\}$.}
\end{cases}
$$ 
\item For some $i \in [n]$, we have $i+1 \leq f_\M(i) < f_\M(i+1) \leq i+n$.  Then $p_\bullet = T_i(a) \cdot q_\bullet$ where $q_\bullet \in \Pluck(\M')_{>0}$ with $\M'$ the positroid satisfying $f_{\M'} = f_\M s_i$, and 
$$
a = p_{I_{i+1}} - p_{I_{i+1} \setminus\{i+1\} \cup \{i\}}.
$$
\end{enumerate}
\end{proposition}

It is easy to see that $f_\M$ satisfies at least one of the three stated conditions in Proposition~\ref{prop:tropreduce}.  In Case (1), if $f_\M(i) = i$ then $I \notin \M$ for all $I$ containing $i$.  Thus $p_I = \infty$ whenever $i \in I$.  It is clear that $p_\bullet$ is in the image of $\epsilon_i$, and that the image of $\epsilon_i$ lies inside $\Pluck(k,n)_{\geq 0}$.  Case (2) is similar.

\subsection{}
We now consider Case (3) of Proposition~\ref{prop:tropreduce}.  Let $f = f_\M$.  We suppose that $i \in [n]$ satisifies $i+1 \leq f_\M(i) < f_\M(i+1) \leq i+n$.  To simplify the notation, we assume that $i = 1$ so we have $2 \leq f(1) < f(2) \leq 1+n$.   

First consider the case $f(1) = 2$.  Define $q_\bullet$ by
$$
q_I = \begin{cases} \infty & \mbox{if $2 \in I$} \\
p_I & \mbox{if $2 \notin I$}
\end{cases}
$$
which clearly satisfies the positive tropical Pl\"ucker relations.  
\begin{lemma}\label{lem:f12}
We have $p_\bullet = T_i(a) \cdot q_\bullet$ where $a = p_{I_2} - p_{I_{2} \setminus\{2\} \cup \{1\}}$. 
\end{lemma}
\begin{proof}
By induction on $\M$, we may assume that $q_\bullet$ is representable and thus $p'_\bullet:= T_i(a) \cdot q_\bullet \in \Pluck(\M)_{>0}$ by Proposition \ref{prop:bridge}.  The assumption $f(1) = 2$ implies that $I \notin \M$ for any $I$ containing $\{1,2\}$.  It suffices to show that $p'_I = p_I$ for $2 \in I$ and $I \in \M$.  Let $I_2 = 2J := J \sqcup \{2\}$ where $J \subset [3,n]$.  We claim that 
\begin{equation}\label{eq:12K}
p_{2K} = p_{1K} + a.
\end{equation}
for $K \subset [3,n]$ such that $2K \in \M$.  To show this we proceed by (downward) induction on $|K \cap J|$, the case $|K\cap J| = k-1$ being tautological.  If $|K \cap J| < k-1$ then by the exchange relation there exists $K'$ with $2K' \in \M$ such that $|K' \cap J| = |K \cap J| + 1$ and $K' = K \setminus \{a\} \cup \{b\}$.  Setting $L = K \setminus \{a\}$, we have by \eqref{eq:trop} the equality
$$
p_{1La}+p_{2Lb} = \min(p_{12L}+p_{Lab}, p_{1Lb}+p_{2La}) = p_{1Lb}+p_{2La}
$$
if $a < b$, using $p_{12L} =\infty$.  The same equality $p_{1La}+p_{2Lb} = p_{1Lb}+p_{2La}$ holds if $a>b$.  Thus by induction we have $p_{2K}-p_{1K} = p_{2K'}-p_{1K'} = a$, establishing \eqref{eq:12K}.

By definition $p'_{2K} = \min(q_{2K}, q_{1K}+a) = p_{1K}+a$.  We conclude that $p_\bullet = p'_\bullet$.
\end{proof} 

\subsection{}
For the remainder of the proof we assume that $f(1) = j > 2$, and we set $f(2) = \ell> j$.  (It is possible for $\ell$ to equal $1+n$.)
Let the Grassmann necklace of $\M$ be $(I_1,I_2,\ldots,I_n)$.  
Then $I_1 = 12I$, $I_2 = 2I j $, $I_3 = I j \ell$.  Define $\M'$ to be the positroid with bounded affine permutation given by $f'(1) = \ell$ and $f'(2) = j$ and $f'(a) = f(a)$ for $a \in [3,n]$.  We define $q_\bullet$ by
$$
q_J = \begin{cases} \mbox{recursion given below} & \mbox{if $1 \notin J$ and $2 \in J$} \\
p_J & \mbox{otherwise.}
\end{cases}
$$
It follows immediately that instances of \eqref{eq:trop} for $q_\bullet$ immediately hold whenever all subsets involved either contain $1$ or do not contain $2$.

We now give the formula for $q_{K2}$, $K \subset [3,n]$ recursively.  At every step, we will check that instances of \eqref{eq:trop} of the form 
\begin{equation}\label{eq:ind}
q_{L2a} + q_{L1b} = \min(q_{L12} + q_{Lba}, q_{L1a}+ q_{L2b})
\end{equation}
hold, where $3 \leq b < a \leq n$ and $K = La \subset [3,n]$.  We say that \eqref{eq:ind} is associated to the pair $\{L2a,L1b\}$ of subsets, which uniquely determines \eqref{eq:ind}.

\noindent
(a) If $K2 \notin \M$, for example if $K < Ij$ in dominance order, then we already have $p_{K2} = \infty$.  We then set $q_{K2} = \infty$ as well.  \eqref{eq:ind} will hold since both sides were already equal to $\infty$ for $p_\bullet$ (and thus also for $q_\bullet$).

\noindent
(b) If $K = I j$, then we set $q_{I2j} = \infty$.  The equation $\eqref{eq:ind}$ holds because the LHS is already $\infty$ for $p_\bullet$.  To see this, observe that for the pair $\{I2j, I1b\}$ with $b < j$, we have that $I1 b  \not \geq_3 Ij\ell = I_3$; for the pair $\{I2j = L'2aj, L'1bj\}$ where $b < a$, either $b < j$ and $L'1bj  \not \geq_3 L'ajk = I_3$, or $b > j$ and $L'1bj \not \geq_1 L'12a = I_1$.  In both cases $I1b$ (resp. $L'1bj$) does not belong to $\M$.

\noindent
(c) If $K2 \in \M$ and $K1 \notin \M$, then we set $q_{K2} = p_{K2}$.  We call the subset $K2$ of {\it parallel type}.  Setting $K = La$, we see that \eqref{eq:ind} holds because 
$$
p_{L2a} + p_{L1b} = \min(p_{L12}+p_{Lab}, p_{L1a}+ p_{L2b}) = p_{L12}+p_{Lab}
$$
and all the terms on both sides are unchanged in $q_\bullet$.

\noindent
(d) Suppose finally that $K2, K1 \in \M$, and $K2 \neq I2j$.  We call such minors {\it general type}.  We inductively assume that the value of $q_{K'2}$ has been defined for all $K'2 < K2$.
Suppose that we have found $2 < b < a \leq n$ such that 
\begin{equation}\label{eq:pair}
\mbox{$K = La$ and $L1b \in \M$.}
\end{equation}
(The existence of such a pair $(a,b)$ follows from an exchange relation argument similar to the proof of Lemma~\ref{lem:f12}, see \cite[Lemma 7.11]{LamCDM}.) We then define
\begin{equation}\label{eq:ab}
q_{K2} =q_{L2a} := \min(q_{L12} + q_{Lba}, q_{L1a}+q_{L2b})- q_{L1b}.
\end{equation}
Since $L2b< L2a$, we may assume that $q_{L2b}$ has already been defined.  

\begin{lemma}
Equation \eqref{eq:ab} well-defines the value of $q_{K2}$, regardless of the choice of $a$ and $b$.
\end{lemma}
\begin{proof}
Suppose we have two pairs $(b < a)$ and $(b' < a')$ such that \eqref{eq:pair} holds.  

Case 1: If $a = a'$, we assume that $b' < b$, and compute as follows.
\begin{align*}
&q_{L2a} + q_{L1b} + q_{L1b'} \\
&= \min(q_{L12} + q_{Lba} +q_{L1b'}, q_{L1a}+q_{L2b} + q_{L1b'})  & \mbox{using \eqref{eq:ab} for $(b < a)$} \\
&=  \min(q_{L12} + q_{Lba} +q_{L1b'}, q_{L1a}+q_{L12} + q_{Lb'b}, q_{L1a}+q_{L1b} + q_{L2b'}) \\
&=\min(q_{L12} + q_{Lb'a} +q_{L1b}, q_{L1a}+q_{L2b'} + q_{L1b})
\end{align*}
Noting that $q_{L1b} + q_{L1b'} < \infty$, we conclude that \eqref{eq:ab} gives the same result using $(b<a)$ and $(b'<a)$.  

Case 2: We have $b = b'$.  Similar to Case 1.

Case 3:
The four numbers $a,b,a',b'$ are distinct.  Set $K = Maa'$, so by assumption $M1ab', M1ba' \in \M$.  Let us assume that $b < a < a'$.  If $b< b' < a$, then $M1ba,M1b'a' \in \M$ as well by \eqref{eq:trop}, and we can reduce to Case 1.  In the other cases, we may conclude from \eqref{eq:trop} that $M1bb', M1aa' \in \M$.
%In all other cases, we either reduce to Case 1, or we can conclude that $M1bb', M1aa' \in \M$.

We now assume that $M1bb',M1aa' \in \M$ and first consider the case $b < a < b' < a'$.  To simplify notation, we omit $M$, and suppose that $(b,a,b',a') = (3,4,5,6)$.  So $145,136,135,146 \in \M$.  In the following, we underline the terms where \eqref{eq:trop} has been applied.  Using \eqref{eq:ab} for the pair $(246,145)$ we have
\begin{align*}
&\underline{q_{246}} + \underline{q_{145}} + q_{136}  \\
&= \min(\underline{q_{124}} + q_{456} + \underline{q_{136}}, q_{146}+q_{245} + q_{136}) \\
&= \min(q_{123} + q_{456} + q_{146}, q_{456} + q_{126} + q_{134}, q_{146}+q_{245} + q_{136})
\end{align*}
Using \eqref{eq:ab} for the pair $(246,136)$ we have
\begin{align*}
&\underline{q_{246}} + q_{145} + \underline{q_{136}}  \\
&= \min(q_{126} + \underline{q_{346}} + \underline{q_{145}}, q_{146}+q_{236} + q_{145}) \\
&= \min(q_{126} + q_{134} + q_{456},q_{126} + q_{146} + q_{345} ,q_{146}+q_{236} + q_{145}).
\end{align*}
The quantity $q_{126}+q_{134}+q_{456}$ appears in both minima, so it suffices to show that
$$
\min(q_{123}+q_{456},q_{245}+q_{136}) = \min(q_{126}+q_{345},q_{236}+q_{145}).
$$
Adding $q_{135}$ (which is $<\infty$), we compute
\begin{align*}
&\min(q_{123}+q_{456}+q_{135},\underline{q_{245}}+q_{136}+\underline{q_{135}}) \\
&=\min(q_{123}+q_{456}+q_{135},\underline{q_{125}}+q_{345}+\underline{q_{136}},q_{145}+q_{235}+q_{136}) \\
&=\min(q_{123}+q_{456}+q_{135},q_{345}+q_{123}+q_{156},q_{345}+q_{126}+q_{135},q_{145}+q_{235}+q_{136}) 
\end{align*}
and
\begin{align*}
&\min(q_{126}+q_{345}+q_{135},\underline{q_{236}}+q_{145}+\underline{q_{135}}) \\
&=\min(q_{126}+q_{345}+q_{135},\underline{q_{145}}+q_{123}+\underline{q_{356}},q_{145}+q_{136}+q_{235}) \\
&=\min(q_{126}+q_{345}+q_{135},q_{123}+q_{135}+q_{456},q_{123}+q_{156}+q_{345},q_{145}+q_{136}+q_{235}).
\end{align*}
This shows that \eqref{eq:ab} for the pair $(246,145)$ gives the same result as for the pair $(246,136)$.
Note that we have used \eqref{eq:ab} for $K2 = 245$ and $K2 = 236$, which may assume to hold since they are both less than $246$ in dominance order.  

The other case $b' < b < a < a'$ is similar.   
\end{proof}

We have now completely defined the vector $q_\bullet$.  
\begin{lemma}\label{lem:qbullet}
We have $q_\bullet \in \Pluck(\M')_{>0}$.
\end{lemma}  

We first show that $q_\bullet$ has the correct support.

\begin{lemma}\label{lem:supp} 
We have $\Supp(q_\bullet) = \M'$.
\end{lemma}
\begin{proof}
We have $\Supp(p_\bullet) = \M$.
Let $J \in \M \setminus \M'$ be such that $p_J < \infty$.  We must show that $q_J = \infty$.  Note that $J = K2$ must be of general type.  Apply \eqref{eq:ab}.  If $q_J < \infty$, then by induction at least one of the two pairs $\{L12,Lba\}$ and $\{L1a,L2b\}$ must be contained in $\M'$, while $J = K2$ is not.  This contradicts the fact that $\M'$ satisfies the 3-term positive exchange relation.  A similar argument shows that for $J \in \M'$ we have $q_J < \infty$.
\end{proof}

By Lemma \ref{lem:supp}, to prove Lemma \ref{lem:qbullet} only need to check relations where the LHS of \eqref{eq:trop} is finite i.e. both terms on the LHS are indexed by elements in $\M'$.  

\begin{lemma}\label{lem:notgeneral}
Suppose that $J,K \in \M'$ appear on the LHS of \eqref{eq:trop}, and both $J$ and $K$ are not of general type.  Then all four subsets on the RHS of \eqref{eq:trop} are also not of general type.
\end{lemma}
\begin{proof}
Let $X$ be any point in $\Pi_{\M,>0}$ and $a>0$ be the unique value such that $X' = x_i(-a) \cdot X \in \Pi_{\M',>0}$ as in \cite{LamCDM}.  We have $\Delta_J(X) = \Delta_{J}(X')$ and $\Delta_K(X) = \Delta_K(X')$.  But we also have $\Delta_I(X') \leq \Delta_I(X)$ for all $I$.  It follows that on the RHS of the three-term Pl\"ucker relation for the matrix $X$, all subsets $S$ that appear must satisfy $\Delta_S(X) = \Delta_S(X')$.  In particular, all subsets $S$ that appear in the relation are not of general type.
\end{proof}

Thus \eqref{eq:trop} holds whenever the LHS involves $q_{T2}$ where $T2$ is not of general type.
It thus suffices to consider instances of \eqref{eq:trop} involving $q_{T2}$ where $T2$ is of general type.

\noindent
(a) We have already verified all relations of the form \eqref{eq:ind}.

\noindent
(b) Let us verify relations of the form
\begin{equation}\label{eq:caseb}
q_{S2b} + q_{Sac} = \min(q_{S2a} + q_{Sbc} , q_{S2c} + q_{Sab})
\end{equation}
where $2 < a < b < c \leq n$ and $S2b$ is of general type.  Thus $S1b \in \M$ and we add $q_{S1b} < \infty$ to both sides.  The LHS becomes
\begin{align*}
&\min(q_{S2b}+q_{S1a} + q_{Sbc},q_{S2b}+q_{S1c}+q_{Sab}) \\
&= \min(q_{S12}+q_{Sab}+q_{Sbc},q_{S1b}+q_{S2a}+q_{Sbc},q_{S2b}+q_{S1c}+q_{Sab})
\end{align*}
which is equal to (RHS of \eqref{eq:caseb} $+q_{S1b}$), using only instances of \eqref{eq:trop} for $q_\bullet$ that we know must hold.

\noindent
(c) Let us verify relations of the form
$$
q_{S2ac} + q_{S12b} = \min(q_{S12a} + q_{S2bc} , q_{S12c} + q_{S2ab})
$$
where $2 < a < b < c \leq n$ and $S2ac$ is of general type.  Thus $S1ac \in\M$ and we add $q_{S1ac} < \infty$ to both sides.  The calculation is similar to (b).

\noindent
(d) Let us verify relations of the form
\begin{equation}\label{eq:verify}
q_{S2ac} + q_{S2bd} = \min(q_{S2ab} + q_{S2cd} , q_{S2ad} + q_{S2bc})
\end{equation}
where $2<a<b<c<d \leq n$ are disjoint from $S2$ and both $S2ac, S2bd$ are of general type.  By assumption, $S1ac, S1bd \in \M$, so it follows that either both $S1ad, S1bc \in \M$ or $S1ab, S1cd \in \M$.

In the former case, add $q_{S1ad}+q_{S1ac} < \infty$ to both sides.  Apply previously established instances of \eqref{eq:trop} successively to the pairs $(2bd,1ad)$, $(abd,1ac)$, $(1ab,2ac)$, $(1ac,1bd)$, $(1ab,2ad)$, $(2ac,abd)$ on the LHS, and $(1ad,2cd)$, $(1ac,2bc)$, $(12c,2ad)$, $(1ac,2ad)$ on the RHS.  The resulting formulae are the same, proving \eqref{eq:verify}.

In the latter case, add $q_{S1ab} + q_{S1bd} < \infty$ to both sides.  Apply previously established instances of \eqref{eq:trop} successively to the pairs $(2ac,1ab)$, $(1ac,1bd)$, $(1bc,2bd)$, $(12b,1ad)$ on the LHS, and $(1ab,2ad)$, $(2cd,1bd)$, $(2bc,abd)$ on the RHS.  The resulting formulae are the same, proving \eqref{eq:verify}.

\noindent
(e) Let us verify the relation \eqref{eq:verify} where one of $S2ac, S2bd$ is of general type, and the other one is of parallel type.  By an argument similar to the proof of Lemma \ref{lem:notgeneral}, at least one of $S2ab,S2cd,S2bc,S2ad$ is of general type.  

We have $S2bd,S2ac \in \M$.  If one of $S1ad$, $S1cd$, or $S1bc$ is in $\M$ then \eqref{eq:3ex} implies that $S1ac \in \M$.  So if $S2ac$ is of parallel type, we may assume that $S1bd \in \M$.  Thus $q_{S1ab}+q_{S1bd} <\infty$ and the argument reduces to that in Case (d).  

Otherwise, we have $S2ac$ of general type.  If $S1cd \in \M$ then \eqref{eq:3ex} implies that $S1bd \in \M$, a contradiction.  If $S1ad \in \M$ then again the argument reduces to that in Case (d).  The remaining cases are very similar, and can be simplified by using the assumption that $q_{S1bd}, q_{S1ad}, q_{S1cd}$ are equal to $\infty$.

We have now completed the proof of Lemma~\ref{lem:qbullet}.  Finally, let us check that indeed $p_\bullet \to q_\bullet$ is a tropical bridge reduction.

\begin{lemma}
Let $a = p_{I2j} - p_{I1j}$.  Then $p_\bullet = T_i(a) \cdot q_\bullet$.
\end{lemma}
\begin{proof}
Let $p'_\bullet =T_i(a) \cdot q_\bullet$.  By induction on $\M$, we may suppose that $q_\bullet$ is representable, so Proposition \ref{prop:bridge} applies and $p'_\bullet$ is a positive tropical Pl\"ucker vector.
We have that $\Supp(p'_\bullet) = \M$ and $p'_J = p_J$ except when $J = K2 > I_2$ is of general type.  But the recursion we used to define $q_{K2}$ can also be applied to $p'_\bullet$ (resp. $p_\bullet$), so we conclude that the rest of the values of $p_\bullet$ and $p'_\bullet$ agree.
\end{proof}

This completes the proof of Proposition~\ref{prop:tropreduce}.

\subsection{Proof of Theorem \ref{thm:bridge}}
Since $\Pluck(\M)_{>0}$ has the structure of a rational polyhedral complex, the statement over $\Z$ implies the other statements.  If $\M = \{I\}$ is a singleton, then $\Pluck(\{I\})_{>0}(\Z)$ consists of the vectors $p(I,z)_\bullet$, $z \in \Z$ given by
$$
p(I,z)_J = \begin{cases} z & \mbox{if $I = J$,} \\
\infty & \mbox{otherwise.}
\end{cases}
$$
Thus the result holds for $\M = \{I\}$.  Otherwise, by using the bridge reduction moves, we have $f_\M = f_{\M'} (i,j)$ for $\dim(\M') = \dim(\M) -1$ and some transposition $\gamma = (i,j)$.  The effect of using the embeddings $\epsilon_a:\Pluck(k,n-1)_{\geq 0} \hookrightarrow \Pluck(k,n)_{\geq 0}$ and $\varepsilon_a:\Pluck(k-1,n-1)_{\geq 0} \hookrightarrow \Pluck(k,n)_{\geq 0}$ is that $(i,i+1)$ is changed to $(i,j)$ (extra numbers are inserted in between $i$ and $i+1$).  We have shown that each $p_\bullet \in \Pluck(\M)_{>0}(\Z)$ is equal to $T_\gamma(a) \cdot q_\bullet$ for a uniquely specified $q_\bullet \in \Pluck(\M')_{>0}(\Z)$ and $a \in \Z$.  Conversely, by induction every point $q_\bullet \in \Pluck(\M')_{>0}(\Z)$ is representable so $T_\gamma(a) \cdot q_\bullet \in \Pluck(\M)_{>0}(\Z)$ by Proposition \ref{prop:bridge}.  We conclude that we have a bijection $\Z \times \Pluck(\M')_{>0}(\Z) \simeq \Pluck(\M)_{>0}(\Z)$.  The stated tropical bridge parametrization follows.

\section{Connected positroids}\label{sec:connected}

\subsection{}
We collect some facts concerning connected positroids here.  

\begin{lemma}\label{lem:connectbridge}
Let $\M$ be a connected positroid with bounded affine permutation $f = f_\M$.  Then at least one of the following holds: 
\begin{enumerate}
\item $f(i) \in \{i,i+1,i+n-1,i+n\}$ for some $i \in [n]$;
\item there exists $i \in [n]$ so that $i < f(i) < f(i+1) < i+n$ such that $f' = f s_i$ is a bounded affine permutation of a connected positroid.
\end{enumerate}
\end{lemma}
\begin{proof}
Suppose that $f = f_\M$ is a counterexample.  Since no $i \in [n]$ satisfies (1),  we must have some $i \in [n]$ such that $i < f(i) < f(i+1) <i+n$.  We have $f'(j) = f(j)$ if $j \neq i,i+1$ and $f'(i) = f(i+1)$ and $f'(i+1) = f(i)$.  Let $\M'$ satisfy $f_{\M'} = f'$.  Write $\pi,\pi':[n] \to [n]$ for the permutations that are reductions of $f,f'$ modulo $n$.  Since (2) fails, $\M'$ is not connected.  Then by Lemma~\ref{prop:disconnect}, we must have a decomposition $[n] = A \cup B$ into disjoint cyclic intervals $A=A[i],B=B[i]$ so that 
$$
\pi'(A) = A \qquad \text{and} \qquad \pi'(B) = B.
$$
But $\M$ is connected so after renaming $A$ and $B$ we must have $i,\pi(i+1) \in A=[j+1,i]$ and $i+1,\pi(i) \in B = [i+1,j]$ where we assume that $j$ has been chosen so that $|B|$ is minimal.  

We now assume that out of all $i \in [n]$ satisfying $i < f(i) < f(i+1) <i+n$ we have chosen $i$ so that $|B[i]|$ is minimal. 
Suppose that $f(r) <f(r+1)$ for some $r,r+1 \in [i+2,j]$.  Since $\pi(r),\pi(r+1) \in B$, the minimality assumption on $|B|$ implies that $\pi(r+1), r , r+1, \pi(r)$ are in order within (the totally ordered cyclic interval) $B$.  Set $f''=fs_r$ and let $\pi''$ be the reduction of $f''$ modulo $n$.  Now let $B[r] = [r+1,j']$, so that $\pi''(B[r]) = B[r]$.  Since $|B[r]| \geq |B|$ by assumption, we must have $B[r] \cap A \neq \emptyset$.  If $\pi(i+1) \in (B[r]\cap A)$, we would contradict $\pi''(B[r]) = B[r]$, since $i+1 \notin B[r]$.  Thus $i \notin B[r]$, and we must have $\pi''(A \cap B[r]) \subset A$ as well, so that $\pi''(A \cap B[r]) = (A \cap B[r])$.  This contradicts our assumption that $B[r] = [r+1,j']$ is chosen so that $|B[r]|$ is minimal.

Thus we must have $f(i+2) > f(i+3) > \cdots > f(j)$, and all these values lie inside $B$ modulo $n$.  It is easy to see that this is impossible (since (1) is never satisfied) and we have arrived at a contradiction.\end{proof}

\subsection{}

Let $\CC \subset \M$ be a cluster for a positroid $\M$.  The plabic graph construction \cite{Pos,OPS} implies that there exists a bipartite graph $G(\CC)$ embedded into the disk, whose faces are in bijection with $\CC$ (write $F(J)$ for the face indexed by $J \in \CC$).  The graph $G(\CC)$ is essentially unique if we assume that interior vertices have degree $> 2$, except for vertices connected to the boundary vertices.  (To fix conventions, let us use ``target-labels" for the faces, in agreement with \cite{OPS}.)  The graph $G(\CC)$ has the following properties:
\begin{enumerate}
\item every connected component of $G(\CC)$ is connected to the boundary, and the connected components of $G(\CC)$ are naturally in bijection with the connected components of $\M$;
\item for each interior face $F(J)$, we have $e_{J_1}+e_{J_3}+\cdots+e_{J_{2r-1}} = e_{J_2}+e_{J_4}+\cdots+e_{J_{2r}}$, where $F(J_1),F(J_2),\ldots,F(J_{2r})$ are the faces edge adjacent to $F(J)$ arranged in cyclic order.
\end{enumerate}

\begin{lemma}\label{lem:mintree}
A connected positroid $\M$ is minimal if and only if $G(\CC)$ is a tree, for some cluster $\CC$ of $\M$.  In this case, $\M$ has a unique cluster $\CC \subset \M$ and we denote by $T_\M:=G(\CC)$ this tree.
\end{lemma}

The positroid polytope $P_\M$ of a minimal connected positroid $\M$ can be described as follows.  We assume that $T_\M$ is chosen so that it is a bipartite tree all of whose vertices have degree $> 2$.  Then the facets of $P_\M$ are in bijection with the edges of $T_\M$.  There are $n$ edges of $T_\M$ that connect to boundary points.  These correspond to external facets: these are facets of $P_\M$ that are supported on the same hyperplane $\{x_i = 0\}$ or $\{x_i = 1\}$ as a facet of the hypersimplex (the color of the internal vertex determines which of $\{x_i = 0\}$ or $\{x_i = 1\}$ occurs).  Interior  edges $e$ of $T_\M$ correspond to internal facets of $P_\M$.  The edge $e$ divides $T_\M$ into two components, and thus $[n] = [a,b-1] \cup [b,a-1]$ into two cyclic intervals.  The internal facet of $P_\M$ is supported on a hyperplane of the form $\sum_{i\in[a,b-1]} x_i = c$ for some $c$.

\subsection{Proof of Lemma~\ref{lem:span}}\label{ssec:proofspan}
We first argue that $\sp_\Z\{e_J \mid J \in \CC\} = X(\hT)$ for any cluster $\CC$.  Since $\M$ is connected, Lemma~\ref{lem:connected} says that the action of $T$ on $\oPi_\M$ is faithful and thus the action of $\hT$ on $\tPi_\M$ is faithful.  On the other hand, if the integral span of $\{e_J \mid J \in \CC\}$ is strictly smaller than $X(\hT)$, then there is a nontrivial subtorus $\{1\} \subsetneq S \subset \hT$ that acts on the rational function field $\C(\Delta_{J} \mid J \in \CC)$ as the identity.  But we have an inclusion $\C[\tPi_\M] \subset \C(\Delta_{J} \mid J \in \CC)$ coming from the cluster structure on $\tPi_\M$, and thus $S$ must act trivially on $\tPi_\M$, a contradiction.

Suppose that $\M$ is a connected positroid.  Then $i<f_\M(i)<i+n$ for all $i$.  By Lemma~\ref{lem:connectbridge}, we must be in one of the following situations:
\begin{enumerate}
\item for some $i$, we have $f_\M(i) = i+1$ or $f_\M(i) = i+n-1$;
\item for some $i$, we have $i < f_\M(i) < f_\M(i+1) < i+n$ and $ f s_i = f_{\M'}$ where $\M'$ is connected.
\end{enumerate}
Suppose (1) holds.  We assume $f_\M(i) = i+n-1$; the other case is similar.  For any cluster $\CC \subset \M$ the graph $G(\CC)$ the vertices $i-1,i$ are connected to the same white interior vertex.  Thus there is $J \in \CC$ such that $i-1 \in J$ but $i-1\notin J'$ for all $J '\in \CC':=\CC \setminus \{J\}$.  The set $\CC'$ is a cluster for some connected positroid $\M'$ on $[n] \setminus \{i-1\}$ (not depending on $\CC$) of the same rank as $\M$ and all clusters $\CC'$ for $\M'$ occur in this way.  By induction, there exists a gauge-fix and cluster $\G' \subset \CC'$ for $\M'$ i.e. $|\G'| = n-1$ and 
$$\sp_\Z\{e_J \mid J \in \CC'\} = \{(x_1,\ldots,x_n) \in \Z^n \mid x_{i-1} = 0 \text{ and } k \text{ divides } \sum x_i \}.$$
It follows that $\G := (\G' \cup \{J\} )\subset \CC:=(\CC' \cup \{J\})$ is a gauge-fix.

Suppose (2) holds.  Let $\G' \subset \CC' \subset \M'$ be a gauge-fix and a cluster for $\M'$.  It is possible to add an edge to the planar bipartite graph $G(\CC')$ to obtain the planar bipartite graph $G(\CC)$ for some cluster $\CC$ of $\M$, see \cite[Section 7.4]{LamCDM}.  We have $\CC = \CC' \cup \{J\}$ where $J$ is the face label of the new face in $G(\CC)$.
Thus $\G' \subset \CC$ is a gauge-fix.

\appendix
\section{\texorpdfstring{The cases $(3,6)$, $(3,7)$, and $(3,8)$}{The cases (3,6), (3,7), and (3,8)}}\label{sec:examples}
\def\ndim{{\rm ndim}}
For a positroid subdivision of the hypersimplex $\tDelta = \{P_{\M_1},\ldots,P_{\M_s}\}$, we define a ``naive dimension":
$$
\ndim(\tDelta) = \sum_{i=1}^s \dim(\Pi_{\M_i,>0}/T_{>0})= \sum_{i=1}^s(\dim(\M_i) - (n-1)).
$$
The naive dimension $\ndim(\tDelta)$ is always greater than or equal to $\dim(\tDelta)$.  In the case $k = 2$, we have $\dim(\tDelta) = \ndim(\tDelta)$, but for $k > 2$ this no longer holds.
We would have $\ndim(\tDelta) = \dim(\tDelta)$ if the positive Chow cell $\Theta_{\tDelta,>0}$ were isomorphic to the product $\prod_i \Pi_{\M_i,>0}/T_{>0}$.  A less naive dimension estimate would take into account $\dim(\M)$ for lower-dimensional faces of $\tDelta$.

Our positive parametrizations essentially agree with those of \cite{SW}.
\subsection{$(k,n) = (3,6)$}\label{ssec:36}
We use the following positive parametrization of $\Gr(3,6)/T$, which is related to \eqref{eq:36param1} by a monomial transformation
$$
\begin{bmatrix}
 0 & 0 & -1 & -1 & -1 & -1 \\
 0 & 1 & 0 & -1 & -1-x_1 & -1-x_1-x_1 x_3 \\
 1 & 0 & 0 & 1 & 1+x_1+x_1 x_2 & 1+x_1+x_1x_2+x_1x_2 x_3 +x_1x_3 +x_1x_2 x_3 x_4 \\
\end{bmatrix}
$$
The polytope $P(3,6)$ (see Proposition~\ref{prop:Ft}) has $f$-vector $(1, 48, 98, 66, 16, 1)$.
We list the sixteen rays of the inner normal fan $\F$ of the polytope $P(3,6)$:
\begin{align*}
r_1&=(-1, 0, 0, 0)
&r_2&=(-1, 0, 0, 1)
&r_3&=(0, -1, 0, 0)
&r_4&=(0, 0, -1, 0) \\
r_5&=(0, 0, 0, -1)
&r_6&=(0, 0, 0, 1)
&r_7&=(0, 0, 1, -1)
&r_8&=(0, 0, 1, 0) \\
r_9&=(0, 1, 0, -1)
&r_{10}&=(0, 1, 0, 0)
&r_{11}&=(0, 1, 1, -1)
&r_{12}&=(1, -1, -1, 0) \\
r_{13}&=(1, -1, 0, 0)
&r_{14}&=(1, 0, -1, 0)
&r_{15}&=(1, 0, 0, -1)
&r_{16}&=(1, 0, 0, 0)
\end{align*}
A comparison of this fan with the cluster fan of $D_4$ was given in \cite{BCL}.
Let $p^{(i)}_\bullet$ be the positive tropical Pl\"ucker vector associated to $r_i$ via Theorem~\ref{thm:param}.
Let $\tDelta_i = \Delta(p^{(i)}_\bullet)$.  These positroid subdivisions form four families under the action of the cyclic group, and we explicitly describe one in each family:
\begin{enumerate}
\item (Rays 1, 5, 8, 10, 13, 14) The two-piece decomposition obtained by slicing $\Delta(3,6)$ with the hyperplane $x_1+x_2 = 1$.  The dimensions $\dim(\M)$ of the participating positroids are $\{7,7\}$.  Thus $\ndim(\tDelta) = 4$ but $\dim(\tDelta) = 3$.
\item (Rays 3, 4, 6, 7, 9, 16) The two-piece decomposition obtained by slicing $\Delta(3,6)$ with the hyperplane $x_1+x_2 +x_3= 1$.  The dimensions $\dim(\M)$ of the participating positroids are $\{5,8\}$.  Thus $\ndim(\tDelta) = 3 =\dim(\tDelta)$.
\item (Rays 2, 15) The three-piece decomposition obtained as follows.  Consider the projection $\pi:\R^6 \to \R^3$ given by $(x_1,x_2,x_3,x_4,x_5,x_6) \mapsto (x_1+x_2,x_3+x_4,x_5+x_6) = (y_1,y_2,y_3)$.  The image $\pi(\Delta(3,6))$ can be divided into three pieces $\{y_1 \leq 1, y_1+y_2 \leq 2\}, \{y_2 \leq 1, y_2+y_3 \leq 2\}, \{y_3 \leq 1, y_1+y_3 \leq 2\}$.  The three pieces of $\tDelta$ are the preimages under $\pi$.  The dimensions $\dim(\M)$ of the participating positroids are $\{6,6,6\}$.  Thus $\ndim(\tDelta) = 3 =\dim(\tDelta)$.
\item (Rays 11,12) The three-piece decomposition that are the preimages under $\pi$ of $\{y_2 \leq 1, y_1+y_2 \leq 2\}, \{y_3 \leq 1, y_2+y_3 \leq 2\}, \{y_1 \leq 1, y_1+y_3 \leq 2\}$.  The dimensions $\dim(\M)$ of the participating positroids are $\{6,6,6\}$.  Thus $\ndim(\tDelta) = 3 =\dim(\tDelta)$.
\end{enumerate}

\subsection{$(k,n) = (3,7)$}
We use the following positive parametrization of $\Gr(3,7)/T$:
{\tiny $$
\begin{bmatrix}
 0 & 0 & 1 \\
 0 & 1 & 0 \\
 -1 & 0 & 0 \\
 -1 & -1 & 1 \\
 -1 & -1-x_1 & 1+x_1+x_1 x_2 \\
 -1 & -1-x_1-x_1 x_3 & 1+x_1+x_1 x_2+x_1 x_3+x_1 x_2 x_3+x_1 x_2 x_3 x_4
   \\
 -1 & {-}1{-}x_1{-}x_1 x_3{-}x_1 x_3 x_5 & 1{+}x_1{+}x_1 x_2{+}x_1 x_3{+}x_1 x_2 x_3{+}x_1
   x_2 x_3 x_4{+}x_1 x_3 x_5{+}x_1 x_2 x_3 x_5{+}x_1 x_2 x_3 x_4 x_5{+}x_1 x_2
   x_3 x_4 x_5 x_6 \\
\end{bmatrix}
$$}

 The polytope $P(3,7)$ has $f$-vector $(1, 693, 2163, 2583, 1463, 392, 42, 1)$.
 We list the 42 rays of the inner normal fan $\F$ of the polytope $P(3,7)$:
\begin{align*}
r_{1}&=(-1, 0, 0, 0, 0, 0)
&r_{2}&=(-1, 0, 0, 1, 0, -1)
&r_{3}&=(-1, 0, 0, 1, 0, 0) \\
r_{4}&=(-1, 0, 0, 1, 1, -1)
&r_{5}&=(0, -1, 0, 0, 0, 0)
&r_{6}&=(0, 0, -1, 0, 0, 0) \\
r_{7}&=(0, 0, -1, 0, 0, 1)
&r_{8}&=(0, 0, -1, 1, 0, 0)
&r_{9}&=(0, 0, 0, -1, 0, 0)\\
r_{10}&=(0, 0, 0, 0, -1, 0)
&r_{11}&=(0, 0, 0, 0, 0, -1)
&r_{12}&=(0, 0, 0, 0, 0, 1)\\
r_{13}&=(0, 0, 0, 0, 1, -1)
&r_{14}&=(0, 0, 0, 0, 1, 0)
&r_{15}&=(0, 0, 0, 1, 0, -1)\\
r_{16}&=(0, 0, 0, 1, 0, 0)
&r_{17}&=(0, 0, 0, 1, 1, -1)
&r_{18}&=(0, 0, 1, -1, -1, 0)\\
r_{19}&=(0, 0, 1, -1, 0, 0)
&r_{20}&=(0, 0, 1, 0, -1, 0)
&r_{21}&=(0, 0, 1, 0, 0, -1)\\
r_{22}&=(0, 0, 1, 0, 0, 0)
&r_{23}&=(0, 1, 0, -1, 0, 0)
&r_{24}&=(0, 1, 0, 0, 0, -1)\\
r_{25}&=(0, 1, 0, 0, 0, 0)
&r_{26}&=(0, 1, 0, 0, 1, -1)
&r_{27}&=(0, 1, 1, -1, -1, 0)\\
r_{28}&=(0, 1, 1, -1, 0, -1)
&r_{29}&=(0, 1, 1, -1, 0, 0)
&r_{30}&=(0, 1, 1, 0, 0, -1)\\
r_{31}&=(1, -1, -1, 0, 0, 0)
&r_{32}&=(1, -1, -1, 0, 0, 1)
&r_{33}&=(1, -1, 0, 0, -1, 0)\\
r_{34}&=(1, -1, 0, 0, 0, 0)
&r_{35}&=(1, 0, -1, 0, 0, 0)
&r_{36}&=(1, 0, -1, 0, 0, 1)\\
r_{37}&=(1, 0, 0, -1, -1, 0)
&r_{38}&=(1, 0, 0, -1, 0, 0)
&r_{39}&=(1, 0, 0, 0, -1, 0)\\
r_{40}&=(1, 0, 0, 0, 0, -1)
&r_{41}&=(1, 0, 0, 0, 0, 0)
&r_{42}&=(1, 0, 1, -1, -1, 0)
\end{align*}
These rays and their corresponding positroid subdivisions form six families under the action of the cyclic group, and we explicitly describe one in each family:
\begin{enumerate}
\item (Rays 1, 11, 14, 20, 25, 34, 35) The two-piece decomposition obtained by slicing $\Delta(3,7)$ with the hyperplane $x_1+x_2 = 1$.  The dimensions $\dim(\M)$ of the participating positroids are $\{9,10\}$.  Thus $\ndim(\tDelta) = 7$ but $\dim(\tDelta) = 5$.
\item (Rays 6, 9, 16, 19, 22, 24, 39) The two-piece decomposition obtained by slicing $\Delta(3,7)$ with the hyperplane $x_1+x_2+x_3= 1$.  The dimensions $\dim(\M)$ of the participating positroids are $\{8,10\}$.  Thus $\ndim(\tDelta) = 6$ but $\dim(\tDelta) = 5$.
\item (Rays 5, 10, 12, 13, 15, 23, 41) The two-piece decomposition obtained by slicing $\Delta(3,7)$ with the hyperplane $x_1+x_2+x_3+x_4= 1$.  The dimensions $\dim(\M)$ of the participating positroids are $\{6,11\}$.  Thus $\ndim(\tDelta) = 5=\dim(\tDelta)$.
\item (Rays 2, 3, 7, 21, 36, 38, 40) The three-piece decomposition obtained as follows.  Consider the projection $\pi:\R^7 \to \R^3$ given by $(x_1,x_2,x_3,x_4,x_5,x_6,x_7) \mapsto (x_1+x_2+x_3,x_4+x_5,x_6+x_7) = (y_1,y_2,y_3)$.  The image $\pi(\Delta(3,7))$ can be divided into three pieces $\{y_1 \leq 1, y_1+y_2 \leq 2\}, \{y_2 \leq 1, y_2+y_3 \leq 2\}, \{y_3 \leq 1, y_1+y_3 \leq 2\}$.  The three pieces of $\tDelta$ are the preimages under $\pi$.  The dimensions $\dim(\M)$ of the participating positroids are $\{7,8,9\}$.  Thus $\ndim(\tDelta) = 6$ but $\dim(\tDelta) = 5$.
\item (Rays 17, 18, 26, 27, 29, 31, 33) The three-piece decomposition that are the preimages under $\pi$ of $\{y_2 \leq 1, y_1+y_2 \leq 2\}, \{y_3 \leq 1, y_2+y_3 \leq 2\}, \{y_1 \leq 1, y_1+y_3 \leq 2\}$.  The dimensions $\dim(\M)$ of the participating positroids are $\{7,8,9\}$.  Thus $\ndim(\tDelta) = 6$ but $\dim(\tDelta) = 5$.
\item (Rays 4, 8, 28, 30, 32, 37, 42) The four-piece decomposition obtained as follows.  Consider the projection $\kappa: \R^7 \to \R^4$ given by 
$$
(x_1,\ldots,x_7) \mapsto (x_1+x_2,x_3+x_4,x_5+x_6,x_7) = (y_1,y_2,y_3,y_4).
$$
The image $\kappa(\Delta(3,7))$ is contained inside the tetrahedron $Q$ given by $y_i \geq 0$ and $\sum_i y_i = 3$.  We cut $Q$ into four pieces $\{y_1+y_4 \leq 1, y_3 \geq 1\}, \{y_3+y_4 \leq 1, y_1 \geq 1\}, \{y_1 \leq 1, y_3 \leq 1, y_2 \geq 1\}, \{y_1+y_4 \geq 1, y_3+y_4 \geq 1, y_2 \leq 1\}$.  The four pieces of $\tDelta$ are the preimages under $\kappa$.  The dimensions $\dim(\M)$ of the participating positroids are $\{7,7,7,8\}$.  Thus $\ndim(\tDelta) = 5=\dim(\tDelta)$.
\end{enumerate}

\subsection{$(k,n) = (3,8)$}
 The polytope $P(3,8)$ has $f$-vector 
 $$(1, 13612, 57768, 100852, 93104, 48544,
14088, 2072, 120, 1).$$
There are thus 120 regular subdivisions of the hypersimplex into positroid polytopes, satisfying $\dim(\tDelta) = 7$.  These hypersimplex subdivisions come in 15 cyclic families, each of size 8.  Let us report on the dimensions of the positroids (for maximal faces) involved.
\begin{center}
\begin{tabular}{|c|c|c|}
\hline 
number of maximal faces & dimensions of positroids & $\ndim(\tDelta)$ \\
\hline
2 & $\{7,14\}$ & $7$\\
\hline
2 & $\{9,13\}$ & $8$\\
\hline
2 & $\{11,12\}$ & $9$\\
\hline
2 & $\{11,13\}$ & $10$\\
\hline
3 & $\{8,10,12\}$ & $9$\\
\hline
3 & $\{8,10,12\}$ & $9$\\
\hline
3 & $\{9,10,11\}$ & $9$\\
\hline
3 & $\{9,10,11\}$ & $9$\\
\hline
4 & $\{8,8,9,11\}$ & $8$\\
\hline
4 & $\{8,8,10,10\}$ & $8$\\
\hline
4 & $\{8,8,10,10\}$ & $8$\\
\hline
4 & $\{9,9,9,10\}$ & $9$\\
\hline
5 & $\{8,8,9,9,9\}$ & $8$\\
\hline
5 & $\{8,8,9,9,9\}$ & $8$\\
\hline
6 & $\{8,8,8,8,8,9\}$ & $7$\\
\hline
\end{tabular}
\end{center}


\begin{thebibliography}{xxxx}
\bibitem[ARW]{ARW} F.~Ardila, F.~Rinc\'on, and L.~Williams, Positroids and non-crossing partitions. Trans. Amer. Math. Soc. 368 (2016), no. 1, 337--363. 
\bibitem[ABL]{ABL} N.~Arkani-Hamed, Y.~Bai, and T.~Lam, Positive geometries and canonical forms, JHEP 2017, Article number: 39 (2017).
\bibitem[ABCGPT]{book} N.~Arkani-Hamed, J.~Bourjaily, F.~Cachazo, A.~Goncharov, A.~Postnikov, and J.~Trnka, Grassmannian geometry of scattering amplitudes. Cambridge University Press, Cambridge, 2016. ix+194 pp.
\bibitem[AHLa]{AHL1} N.~Arkani-Hamed, S.~He, and T.~Lam, Stringy canonical forms, preprint, 2019; {\tt arXiv:1912.08707}.
\bibitem[AHLb]{AHL2} N.~Arkani-Hamed, S.~He, and T.~Lam, Cluster configuration spaces of finite type, preprint, 2020; {\tt arXiv:2005.11419}.
\bibitem[ALS]{ALS} N.~Arkani-Hamed, T.~Lam, and M.~Spradlin, 
Non-perturbative geometries for planar ${\mathcal N}=4$ SYM amplitudes, preprint, 2019; {\tt arXiv:1912.08222}.
\bibitem[ALS+]{ALS2} N.~Arkani-Hamed, T.~Lam, and M.~Spradlin, in preparation.
\bibitem[AT]{AT} N.~Arkani-Hamed and J.~Trnka, The Amplituhedron, JHEP volume 2014, Article number: 30 (2014).
\bibitem[BC]{BC} F.~Borges and F.~Cachazo, Generalized Planar Feynman Diagrams: Collections, preprint, 2019; {\tt arXiv:1910.10674}.
\bibitem[BCL]{BCL} S.B.~Brodsky, C.~Ceballos, and J.-P.~Labb\'e,
Cluster algebras of type D4, tropical planes, and the positive tropical Grassmannian.
Beitr. Algebra Geom. 58 (2017), no. 1, 25--46.
\bibitem[Bro]{Brown} F.C.S.~Brown, Multiple zeta values and periods of moduli spaces $\M_{0,n}$, Ann. Sci. \'Ec. Norm. Sup\'er. (4) 42 (2009), no. 3, 371--489.
\bibitem[CEGM]{CEGM}
  F.~Cachazo, N.~Early, A.~Guevara and S.~Mizera,
  Scattering Equations: From Projective Spaces to Tropical Grassmannians,
  JHEP 1906, Article number: 39 (2019).
\bibitem[CGUZ]{CGUZ} F.~Cachazo, A.~Guevara, B.~Umbert, and Y.~Zhang, 
Planar Matrices and Arrays of Feynman Diagrams, preprint, 2019; {\tt arXiv:1912.09422}.
\bibitem[CR]{CR} F.~Cachazo and J.~M.~Rojas, Notes on Biadjoint Amplitudes, Trop $G(3,7)$ and $X(3,7)$ Scattering Equations, JHEP 2020, Article number: 176 (2020).
\bibitem[DFOKa]{DFOK1}
  J.~Drummond, J.~Foster, \"O.~G\"urdo\v{g}an and C.~Kalousios,
  Tropical Grassmannians, cluster algebras and scattering amplitudes,
  JHEP 2020, Article number: 146 (2020).
\bibitem[DFOKb]{DFOK2}
  J.~Drummond, J.~Foster, \"O.~G\"urdo\v{g}an and C.~Kalousios,
  Algebraic singularities of scattering amplitudes from tropical geometry,
  preprint, 2019; {\tt arXiv:1912.08217}.
  \bibitem[DFOKc]{DFOK3}
  J.~Drummond, J.~Foster, \"O.~G\"urdo\v{g}an and C.~Kalousios,
  Tropical fans, scattering equations and amplitudes,
  preprint, 2020; {\tt arXiv:2002.04624}.
\bibitem[Eara]{Ear} N.~Early, From weakly separated collections to matroid subdivisions, preprint, 2019; {\tt arXiv:1910.11522}.
\bibitem[Earb]{Ear2} N.~Early, Planar kinematic invariants, matroid subdivisions and generalized Feynman diagrams, preprint, 2019; {\tt arXiv:1912.13513}.
\bibitem[Ful]{Ful} W.~Fulton, Introduction to toric varieties. Annals of Mathematics Studies, 131. The William H. Roever Lectures in Geometry. Princeton University Press, Princeton, NJ, 1993. xii+157 pp. 
\bibitem[GKLa]{GKL} P.~Galashin, S.N.~Karp, and T.~Lam, The totally nonnegative Grassmannian is a ball, preprint, 2017; {\tt arXiv:1707.02010}.
\bibitem[GKLb]{GKL2} P.~Galashin, S.N.~Karp, and T.~Lam, Regularity theorem for totally nonnegative flag varieties, preprint, 2019; {\tt arXiv:1904.00527}.
\bibitem[GL]{GL} P.~Galashin and T.~Lam, Positroid varieties and cluster algebras, preprint, 2019; {\tt arXiv:1906.03501}. 
\bibitem[GGMS]{GGMS}  I. Gelfand, M. Goresky, R. Macpherson, and V. Serganova, Combinatorial geometries, convex polyhedra, and Schubert cells, Adv. in Math. 63 (1987), no. 3, 301--316.
\bibitem[GKZ]{GKZ} I.M.~Gelfand, M.M.~Kapranov, and A.V.~Zelevinsky, Discriminants, resultants, and multidimensional determinants. Mathematics: Theory and Applications. Birkh\"auser Boston, Inc., Boston, MA, 1994. x+523 pp.
\bibitem[GGSVV]{GGSVV} 
J. Golden, A. B. Goncharov, M. Spradlin, C. Vergu and
A. Volovich, Motivic Amplitudes and Cluster Coordinates, JHEP 1401, 091 (2014).
\bibitem[HP]{HP}
  N.~Henke and G.~Papathanasiou,
  How tropical are seven- and eight-particle amplitudes?,
  preprint, 2019; {\tt arXiv:1912.08254}.
\bibitem[HJJS]{HJJS} 
S.~Herrmann, A.~Jensen, M.~Joswig, and B.~Sturmfels, How to draw tropical planes. (English summary) Electron. J. Combin. 16 (2009), no. 2, Special volume in honor of Anders Bj\"orner, Research Paper 6, 26 pp.

\bibitem[HJS]{HJS} S.~Herrmann, M.~Joswig, and D.~Speyer, 
Dressians, Tropical Grassmannians and their rays, Forum Mathematicum 26 (2014), no. 6, 1853--1881.

\bibitem[Kap]{Kap} M.~Kapranov, Chow quotients of Grassmannians I, Advances in Soviet Mathematics, 16 (1993), 29--110.
\bibitem[Kar]{Kar} R.~Karpman, Bridge graphs and Deodhar parametrizations for positroid varieties, J. Combin. Theory, Series A Volume 142, August 2016, 113--146.
\bibitem[KT]{KT} S.~Keel and J.~Tevelev, Geometry of Chow quotients of Grassmannians. Duke Math. J. 134 (2006), no. 2, 259--311.
\bibitem[KLS]{KLS} A.~Knutson, T.~Lam, and D.E.~Speyer, Positroid varieties: juggling and geometry. Compos. Math. 149 (2013), no. 10, 1710--1752.
\bibitem[Laf]{Laf} L.~Lafforgue, Chirurgie des grassmanniennes. (French) [Surgery on Grassmannians] CRM Monograph Series, 19. American Mathematical Society, Providence, RI, 2003. xx+170 pp. 
\bibitem[Lam]{LamCDM} T.~Lam, Totally nonnegative Grassmannian and Grassmann polytopes. Current developments in mathematics 2014, 51--152, Int. Press, Somerville, MA, 2016.
\bibitem[LP]{LP1} T.~Lam and A.~Postnikov, Alcoved polytopes. I. Discrete Comput. Geom. 38 (2007), no. 3, 453--478.
\bibitem[LP+]{LP} T.~Lam and A.~Postnikov, Polypositroids, preprint, 2020; {\tt arXiv:2010.07120}.
\bibitem[LS]{LS} K.~Lee and R.~Schiffler, Positivity for cluster algebras. Ann. of Math. (2) 182 (2015), no. 1, 73--125. 
%\bibitem[Lin]{Lin} Linstrom
\bibitem[LPW]{LPW} 
T.~Lukowski, M.~Parisi, L.K.~Williams, The positive tropical Grassmannian, the hypersimplex, and the $m=2$ amplituhedron, preprint, 2020; {\tt arXiv:2002.06164}.
\bibitem[Lus]{Lus} G. Lusztig, Total positivity in reductive groups, Lie theory and geometry, 531--568, Progr. Math., 123, Birkh\"auser Boston, Boston, MA, 1994.
\bibitem[Mne]{Mn}  N. E. Mn\"ev: The universality theorems on the classification problem of configuration varieties and convex polytopes varieties. Topology and geometry--Rohlin Seminar, 527--543, Lecture Notes
in Math., 1346, Springer, Berlin, 1988.
\bibitem[MS]{MS} G.~Muller and D.E.~Speyer, The twist for positroid varieties. Proc. Lond. Math. Soc. (3) 115 (2017), no. 5, 1014--1071. 
\bibitem[Namp]{Namp} N.~Arkani-Hamed, Talk at {\it Amplitudes 2019}; \\ {\url{https://indico.cern.ch/event/750565/contributions/3439541/attachments/1873668/3084360/Arkani-Hamed.pdf}}.
\bibitem[Oh]{Oh} S.~Oh, Positroids and Schubert matroids. J. Combin. Theory Ser. A 118 (2011), no. 8, 2426--2435. 
\bibitem[OhPS]{OPS} S.~Oh, A.~Postnikov, and D.E.~Speyer, Weak separation and plabic graphs. Proc. Lond. Math. Soc. (3) 110 (2015), no. 3, 721--754.
\bibitem[Ola]{Ola} Jorge Alberto Olarte, PhD Thesis, preprint; {\url{https://sites.google.com/view/jaolarte/home}}.
\bibitem[OlPS]{OPS2} Jorge Alberto Olarte, Marta Panizzut, Benjamin Schr\"oter, On local Dressians of matroids, preprint, 2018; {\tt arXiv:1809.08965}.
\bibitem[Pos]{Pos} A. Postnikov. Total positivity, Grassmannians, and networks, preprint;
{\url{http://math.mit.edu/~apost/papers/tpgrass.pdf}}.
\bibitem[Rie]{Rie} K.~Rietsch, An algebraic cell decomposition of the nonnegative part of a flag variety. J. Algebra 213 (1999), no. 1, 144--154. 
\bibitem[SG]{SG} D.~G.~Sep\'ulveda and A.~Guevara, A Soft Theorem for the Tropical Grassmannian, preprint, 2019; {\tt arXiv:1909.05291}.
\bibitem[Spe]{Spe} D.E.~Speyer, Tropical linear spaces. SIAM J. Discrete Math. 22 (2008), no. 4, 1527--1558.
\bibitem[SS]{SS} D.~Speyer and B.~Sturmfels, The tropical Grassmannian. Adv. Geom. 4 (2004), no. 3, 389--411.
\bibitem[SW]{SW} D.~Speyer and L.~Williams, The tropical totally positive Grassmannian. J. Algebraic Combin. 22 (2005), no. 2, 189--210.
\bibitem[SW+]{SW2} D.~Speyer and L.~Williams, The positive Dressian equals the positive tropical Grassmannian, preprint, 2020; {\tt arXiv:2003.10231}.
\end{thebibliography}
\end{document}